\documentclass[11pt,reqno]{amsproc}
\usepackage[margin=1in]{geometry}
\usepackage{amsmath, amsthm, amssymb}
\usepackage{hyperref}
\hypersetup{colorlinks=true, pdfstartview=FitV, linkcolor=blue,citecolor=blue, urlcolor=blue}
\usepackage[abbrev,lite,nobysame]{amsrefs}
\usepackage{times}
\usepackage[usenames,dvipsnames]{color}
\usepackage{bbm}
\usepackage{bm}
\usepackage{subcaption}
\usepackage{mathtools}
\usepackage{pdfpages}
\usepackage{comment}
\usepackage{appendix}
\definecolor{labelkey}{rgb}{0,0,1}

\newcommand{\R}{\mathbb{R}}

\newcommand{\N}{\mathbb{N}}

\providecommand{\R}{\mathbb{R}}

\providecommand{\N}{\mathbb{N}}

\providecommand{\eps}{\varepsilon}

\newcommand{\pv}{\operatorname{p.\!v.}}

\renewcommand{\leq}{\leqslant}
\renewcommand{\geq}{\geqslant}

\renewcommand{\div}{\operatorname{div}}

\newcommand{\dist}{\operatorname{dist}}

\newcommand{\Id}{\operatorname{Id}}

\DeclareMathOperator{\supp}{supp}

\newcommand{\loc}{\mathrm{loc}}

  \newtheorem{thm}{Theorem}[section]
  \newtheorem{cor}[thm]{Corollary}
  \newtheorem{Lemma}[thm]{Lemma}
  \newtheorem{Proposition}[thm]{Proposition}
    \newtheorem{Definition}[thm]{Definition}
  \newtheorem{remark}[thm]{Remark}

\usepackage{etoolbox}
\patchcmd{\subsubsection}{\itshape}{\itshape\bfseries}{}{} % THIS CHANGE ITALIC FONT OF SUBSUBSECTION INTO BOLD-ITALIC

\title{Pointwise Hadamard variational formula for the fractional Laplacian}

\author{Sidy M. Djitte and Franck Sueur}

\address[S. M. Djitte]{Department of Mathematics
Maison du nombre, 6 avenue de la Fonte, 
University of Luxembourg,  
L-4364 Esch-aur-Alzette, Luxembourg}
\email{sidymoctar.djitte@uni.lu, sidy.m.djitte@aims-senegal.org}
\address[F. Sueur]{Department of Mathematics
Maison du nombre, 6 avenue de la Fonte, 
University of Luxembourg,  
L-4364 Esch-aur-Alzette, Luxembourg}
\email{Franck.Sueur@uni.lu}

\keywords{Fractional Laplacian, Shape derivative, Green function, Robin function}
\begin{document}

\begin{abstract}
We establish pointwise formulas for the shape derivative of solutions to the Dirichlet problem associated with the fractional Laplacian. Specifically, we consider the equation $(-\Delta)^s u = h$ in $\Omega$ and $u=0$ in $\Omega^c$, where the right-hand side $h$ is either a Dirac delta distribution or a Lipschitz function. In both cases, we prove that the corresponding solution is shape differentiable in every direction and we derive a formula for the pointwise value of its shape derivative.
%In the case where the right hand side is  a Dirac distribution, so that the corresponding solution is the Green function associated with the fractional Laplacian, we cover the case where the order $s$ of the fractional Laplace operator  $(-\Delta)^s$ is in 
% $(0,1)$. We use the latter to also prove a similar formula for the pointwise value of the shape derivative in the case where the right hand side is a regular function. 
These formulas involve integral on the domain's boundary and fractional Neumann's traces. This extends to the case of the fractional Laplacian the well-known Hadamard variational formula for the standard Laplacian. Our argument is in the spirit of \cite{Ushikoshi, Kozono-Ushikoshi} and is based on PDEs techniques.
\end{abstract}
%
%%%%%%%%%%%%%%%%%%%%%%%%%%%%%%%%%%%%%%
\maketitle

\setcounter{tocdepth}{3}
\tableofcontents

\section{Introduction and statement of the results}
The fractional Laplace operators appear in different disciplines
of mathematics: PDEs and probability theory among others, in various applications, issued from biology or finance, for example.
 The fractional Laplace operator  $(-\Delta)^s$ of order $s\in(0,1)$ can be defined by its action on real valued  functions $u$, defined on $\R^N$, with $N \geq 1$,  through the formula:
\begin{equation}\label{def-int}
     (-\Delta)^su (x) := c_{N,s}\,  \pv \int_{\R^N}\frac{u(x)-u(y)}{|x-y|^{N+2s}}dy,
\end{equation}
with 
 \begin{equation*}
c_{N,s} :=\pi^{-N/2}s4^s\frac{\Gamma(\frac{N+2s}{2})}{\Gamma(1-s)},
\end{equation*}
where $\Gamma$ is the gamma function, and where $\pv$ refers to the Cauchy principal value.  
The reason for the presence of the normalization constant $c_{N,s}$
 is to match with another natural definition, based on the Fourier transform which sets the fractional Laplace operator $(-\Delta)^s$ as the Fourier multiplier of symbol 
$|\xi|^{2s}$. 
Yet another viewpoint is to consider the operator  $(-\Delta)^s$ as the Euler-Lagrange equations associated with the  quadratic form: 
\begin{equation*}
    \mathcal E_s(u,v):=\frac{c_{N,s}}{2}\iint_{\R^{2N}}\frac{(u(x)-u(y))(v(x)-v(y))}{|x-y|^{N+2s}}dxdy.
\end{equation*}
First the latter can be used as the square of the norm of the fractional Sobolev space of order $s$.  We shall therefore denote by $H^s(\R^N)$ the set of all those $L^2(\R^N)$ functions $u$ for which $\mathcal E_s(u,u)<\infty$. 
Moreover this  quadratic form, or rather its associated symmetric bilinear form may be used 
to define a notion of weak solution of Dirichlet type problems. 
Indeed, for a given bounded open set $\Omega\subset\R^N$  and for a given source term $h$, a classical problem associated with the fractional Laplace operator is 
\begin{equation} \label{Dir-01}
\left\{ \begin{array}{rcll} (-\Delta)^s u&=& h  &\textrm{in }\;\;\Omega \\ u&=&0&
\textrm{in }\;\;\R^N\setminus\Omega. \end{array}\right. 
\end{equation}
Observe that, as opposed to the classical  Dirichlet problem for the Laplace operator, the 
second condition is set on the whole complement of $\Omega$, rather than solely on its boundary $\partial \Omega$. While the standard Laplace operator is not sensitive to the remote values of the test function, the fractional  Laplace operator is non-local in nature, as exemplified by the integral formula \eqref{def-int}, which leads to such adaptation of the Dirichlet problem. Letting  $\mathcal H^s_0(\Omega)$ be  the space of the elements of $ H^s(\R^N)$ for which $u\equiv 0$ in $\R^N\setminus\Omega$, we say that $u$ in  $\mathcal H^s_0(\Omega)$  is a  weak solution to \eqref{Dir-01} if 
$$
\mathcal E_s(u,v)=\int_{\Omega}h(y)v(y)dy\quad   \text{for all } \, v\in \mathcal H^s_0(\Omega). 
$$
By standard compactness argument, it is well known that if $h\in L^2(\Omega)$, then there exists a solution and it is unique by the maximum principle.
We refer to the recent nice book \cite{FRO} for a far more complete overview of the fractional  Laplace operator.
\par\;

% \subsection{Case of a regular right hand side}
Our aim in this work is to examine the effect of a variation of the 
domain $\Omega$ on the solution 
 to the Dirichlet problem \eqref{Dir-01}. 
More precisely, we aim to compute the shape derivative of the solution $u$ with respect to the variation of the domain $\Omega$ along a given vector field $Y$. 
We define this notion more precisely, following  \cite{Sokolowski}.
We suppose that  the source term $h$ is sufficiently regular in $D$, where  $\Omega\Subset D\subset \R^N$. 
\begin{Definition} \label{def-ad}
Let $Y$ be a globally $C^{1,1}$ vector field in $\R^N$ and  let $(\Phi_t)_{t\in (-1,1)}$ 
be a family of diffeomorphisms such that
\begin{equation}\label{transformations}
\text{for all $t\in (-1,1)$, the mapping}\quad t\mapsto \Phi_t\in C^{1,1}(\R^N,\R^N)\;\text{is of class $C^2$},
\end{equation}
with
\begin{equation*}
\Phi_0=\Id_{\R^N} \quad\text{and}\quad \partial_t{\Phi_t}\big|_{t=0}= Y. 
\end{equation*} 
For any $t\in (-1,1)$,   let $u_t$ the unique weak  solution of the Dirichlet problem:
\begin{equation*}
\left\{ \begin{array}{rcll} (-\Delta)^s u_t&=& h  &\textrm{in }\;\;\Omega_t \\ u_t&=&0&
\textrm{in }\;\;\R^N\setminus\Omega_t , \end{array}\right. 
\end{equation*}
where $\Omega_t$ denotes the transformed domain: 
\begin{equation*}
   \Omega_t:=\Phi_t(\Omega).  
\end{equation*}
For $t\in (-1,1)$, let $v_t$ the function on $\Omega$ 
 defined by $v_t:=u_t\circ \Phi_t$. 
Then the Lagrangian shape derivative of $u$ in the direction of the vector field $Y$ is defined as
$$v'=\partial_t{v_t}\big|_{t=0}.$$
The Eulerian shape derivative  of $u$ 
in the direction of the vector field $Y$ is then defined as 
\begin{equation}\label{shape-deriv}
u':=v'-\nabla u\cdot Y.
\end{equation} 
We say that $u$ is shape differentiable in the direction of $Y$ if $u'(x)$ exists for all $x\in \Omega$.
\end{Definition}
The interest of the functions $v_t$ is that they are all defined on the same domain, that is the initial domain $\Omega$. 
On the other hand the Eulerian shape derivative, which we will simply call the shape derivative in the rest of the paper, is a more natural notion since 
 by the chain rule we observe that for any $x$ in $\Omega$, 
\begin{align} \label{taudacroisse}
 u'(x)  :=\lim_{t\to 0}\frac{u_t (x) - u_0(x)}{t} .
\end{align}
We aim to find a pointwise formula for the shape derivative $u'$ when the data $h$ is sufficiently regular so that \eqref{shape-deriv} make sense. We will prove that the pointwise values of the shape derivative $u'$ at any $x\in\Omega$ can be written as an integral over $\partial\Omega$ of some fractional traces of both $u$ and the Green function $G^s_\Omega(\cdot,\cdot)$ associated to the operator $(-\Delta)^s$. 
The latter is, by definition, the solution to 
\begin{equation*}
    (-\Delta)^s G_\Omega^s(x,\cdot)=\delta_x \quad\text{in}\quad \mathcal D'(\Omega)\qquad\text{and}\qquad G_\Omega^s(x,\cdot)=0\quad\text{in}\quad \R^N\setminus\Omega.
\end{equation*}
Here $\delta_x$ denote the Dirac delta distribution at the position $x$ in $\Omega$. 
 The  fractional traces hinted above are defined in terms of limits of quotients by fractional powers of the distance to the boundary. From now on, $\Omega$ will be a bounded open set of $\R^N$ of class $C^{1,1}$. Let $\delta$ in $C^{1,1}(\R^N)$ which coincides with the signed distance function near the boundary. Moreover, we assume that $\delta$ is positive in $\Omega$ and negative in $\R^N\setminus\Omega$. We  denote by $\nu$ the outward unit normal to the boundary. We recall that by regularity theory (see e.g \cite{RS}),  the fractional Neumann traces 
$$
\gamma_0^s(G_\Omega^s(x,\cdot)):= \big(G_\Omega^s(x,\cdot)/\delta^s\big)\big|_{\partial\Omega} \text{ and } \gamma_0^s(u):=\big(u/\delta^s\big)\big|_{\partial\Omega}\quad\text{are well defined for each $x\in \Omega$.}
$$
 \par\;
 
Our first result can then be stated as follows.
\begin{thm}\label{THM-1.2}
Let $\Omega\subset\R^N$  be  a  bounded open set of class $C^{1,1}$ with $N>2s$ and let $s\in (0,1)$. 
Let $h$ be Lipschitz continuous in $D$, where  $\Omega\Subset D\subset \R^N$ and $u$ be the unique weak solution of \eqref{Dir-01}.
Let $Y$ be a globally $C^{1,1}$ vector field in $\R^N$. Then, $u$ is shape differentiable in the direction of $Y$. Moreover, for all $x\in\Omega$, there holds:
  \begin{equation}\label{repres-ahape-deriv}
      u'(x)= \Gamma^2(1+s)\int_{\partial\Omega}\gamma_0^s(G_\Omega^s(x,\cdot))\gamma_0^s(u)Y\cdot\nu \, d\sigma .
  \end{equation}
In the above, $\Gamma$ stands for the classical gamma function and $\nu$ denotes the outward unit normal to the boundary.
\end{thm}
\par\;
The proof of Theorem \ref{THM-1.2}, which will be a consequence of a more general result, is given in Section \ref{Proof-THM-1.2}.
\begin{remark}
By standard regularity theory, it is well known that if, for instance, $h\in L^\infty(D)$, then the Neumann trace $\gamma_0^s(u)$ exists and therefore the RHS of the identity \eqref{repres-ahape-deriv} still make sense. This indicates that, in principle, it should be possible to extends Theorem \ref{THM-1.2}, up to replacing the LHS of \eqref{repres-ahape-deriv} by \eqref{taudacroisse} (which might exist regardless of whether $\nabla u$ exists or not) to less regular functions than Lipschitz ones. However, for technical reasons, in here we only restrict ourselves to the case where the Dirichlet data $h$ is Lipschitz continuous.
\end{remark}
\par\;
%\subsection{Case of the Green function}
Our second result extends Theorem \ref{THM-1.2} --in some sense-- to the case where $h=\delta_x$ is the Dirac delta distribution so that the corresponding solution $u$ is the Green function $G_\Omega^s(x,\cdot)$  associated to the operator $(-\Delta)^s$. In the case of the classical laplacian, that is, when $s=1$, such result were obtained by Hadamard after  his pioneering work \cite{Hadamard}. More precisely, let $\Omega_t$ be defined from the initial domain  $\Omega$ and a vector field  of the form $Y=\alpha\nu$  
 with $\alpha\in C^\infty(\partial\Omega)$ and $\nu$ the outer unit normal, by setting 
$$
\partial\Omega_t=\Big\{y=x+t\alpha(x)\nu(x), \;x\in\partial\Omega\Big\}.
$$

 This is a particular case of the domain perturbations given by  Definition   \ref{def-ad}.  For any $x,y\in \Omega$, set 
 \begin{align} \label{Hadamard-Def}
  \partial_t G_{\Omega_t}(x,y){\big|_{t=0}}  :=\lim_{t\to 0}\frac{G_{\Omega_t}(x,y)-G_\Omega(x,y)}{t},
\end{align}
where $G_{\Omega_t}(\cdot,\cdot)$ is the Green function of the perturbed domain $\Omega_t$ associated to the Laplace operator associated with homogeneous Dirichlet condition, that is the unique solution to 
\begin{equation*}
    -\Delta G_{\Omega_t}(x,\cdot)=\delta_x \quad\text{in}\quad \mathcal D'(\Omega_t)\qquad\text{and}\qquad G_{\Omega_t}(x,\cdot)=0\quad\text{on}\quad \partial \Omega_t.
\end{equation*}
Then, Hadamard result says that the ratio \eqref{Hadamard-Def} exists and is given by
\begin{align} \label{Hadamard}
  \partial_t G_{\Omega_t}(x,y){\big|_{t=0}} 
    =\int_{\partial\Omega}\gamma_N( G_\Omega(x,\cdot))(z)\gamma_N( G_\Omega(y,\cdot))(z)\alpha(z)d\sigma(z) .
\end{align}
where $\gamma_N(G_1(x,\cdot))(z):= \nabla_z G_1(x, z)\cdot\nu(z)$ is the Neumann trace. Hadamard proved his result under the assumption that the weight function $\alpha$ does not change sign. A generalization for signed changing $\alpha's$ were later obtained by  Schiffer, Schiffer and Garabedian \cite{Schiffer, Garabedian} who also considered more general perturbations. After these pioneering works, shape derivative computations has received significant interest and several extensions of Hadamard's formula for more general elliptic boundary values problems. We refer here to \cite{HenrotPierre,Sokolowski} and  to \cite{Oz} for more. \medskip 
\par\;

Our next result extends the identity \eqref{Hadamard} --in some sense-- to the fractional Green function $G^s_\Omega(\cdot,\cdot)$.
We continue to assume that $\Omega$ is a bounded open set of class $C^{1,1}$ of $\R^N$. Because the Green function depends on two variables, the notion of directional shape derivative defined in \eqref{shape-deriv} has to be modified properly. 
\begin{Definition}Let $s$ in $(0,1)$ and $Y$ be a globally $C^{1,1}$ vector field in $\R^N$. We shall say that the Green function $G_\Omega^s(\cdot,\cdot)$ is shape differentiable in the direction of $Y$ if for all $x,y\in \Omega$ with $x\neq y$, the limit 
\begin{equation}\label{def-direc-shape-deriv}
 D G_\Omega^s(Y)(x,y):= \lim_{t\to 0}\frac{ G^s_{\Omega_t}(\Phi_t(x),\Phi_t(y))- G^s_{\Omega}(x,y)}{t}-\nabla_yG_\Omega^s(x,y)\cdot Y(y)-\nabla_xG_\Omega^s(x,y)\cdot Y(x),
\end{equation}
exists and is finite.    
\end{Definition}
 \par\;

We then have the following more general result.
\begin{thm}\label{THM-1.5} Let $s$ in $(0,1)$ and  $Y$ be a globally $C^{1,1}$ vector field in  $\R^N$ with $N>2s$. Then the Green function $G_\Omega^s(\cdot,\cdot)$ is shape differentiable in the direction of $Y$. Moreover, for all $x,y\in \Omega$ with $x\neq y$, there holds:
\begin{align} 
\label{var-green}
 D G_\Omega^s(Y)(x,y)=\Gamma^2(1+s)\int_{\partial\Omega}\gamma_0^s(G_\Omega^s(x,\cdot))\gamma_0^s(G^s_\Omega(y,\cdot))Y\cdot\nu \;d\sigma.
\end{align}
\end{thm}
The proof of Theorem \ref{THM-1.5} is given in Section \ref{Sec-proof-main-result}. 
\begin{remark}
Formula \eqref{var-green} can be regarded as the fractional analogue of the classical Hadamard formula \eqref{Hadamard}. Indeed, for every $x\neq y$ in $\Omega$, we have
\[
G^s_{\Omega_t}(x,y)
   = G^s_{\Omega_t}\!\big(\Phi_t(\cdot),\Phi_t(\cdot)\big)
     \circ (\Phi_t^{-1}(x),\Phi_t^{-1}(y)).
\]
If the mapping
\begin{equation}\label{CHG}
    t \mapsto 
    G^s_{\Omega_t}\!\left(\Phi_t(\cdot),\Phi_t(\cdot)\right)
    \in C(\Omega\times\Omega \setminus \{(x,x)\})
\end{equation}
were differentiable at $t=0$ away from the diagonal, then applying the chain rule to the identity above would give
\begin{equation}\label{MD==SD}
 \partial_t G^s_{\Omega_t}(x,y)\big|_{t=0}
     = D G_\Omega^s(Y)(x,y)=\Gamma^2(1+s)\int_{\partial\Omega}\gamma_0^s(G_\Omega^s(x,\cdot))\gamma_0^s(G^s_\Omega(y,\cdot))Y\cdot\nu \;d\sigma
\end{equation}
which yields the desired relation. Unfortunately, we are not able to justify the differentiability of the mapping in \eqref{CHG} at the strength required in order to apply the chain rule: the theorem guarantees differentiability only for the scalar-valued mapping
\[
t \mapsto G^s_{\Omega_t}\big(\Phi_t(x),\Phi_t(y)\big)\in\mathbb{R},
\]
which is weaker than the functional differentiability required in \eqref{CHG}. Nevertheless, one can show that
\[
t\mapsto 
G^s_{\Omega_t}\big(\Phi_t(\cdot),\Phi_t(\cdot)\big)\in C_{\mathrm{loc}}(\Omega\times\Omega\setminus\{(x,x)\})
\]
is differentiable. Consequently, for all $x,y\in K\Subset\Omega$ with $x\neq y$, the mapping $t\mapsto G_{\Omega_t}(x,y)$ is differentiable at zero and formula \eqref{MD==SD} holds true.
\end{remark}

As a consequence of Theorem \ref{THM-1.5}, we derive a formula for the shape derivative of the Robin function associated to the fractional Laplacian. Recall that when $N>2s$, the Green function $G_\Omega^s(x,\cdot)$ splits into:
\begin{equation}\label{Eq-apliting-of-Green}
    G_\Omega^s(x,\cdot)= F_s(x,\cdot)-H^s_\Omega(x,\cdot),
\end{equation}
where 
\begin{equation*}
   F_s(x,\cdot) :=\frac{b_{N,s}}{|x-\cdot|^{N-2s}}\quad\text{with}\quad b_{N,s} :=\pi^{-N/2} \, 4^{-s}  \, \frac{\Gamma(\frac{N-2s}{2})}{\Gamma(s)},
\end{equation*}
%
%with 
%
%\begin{equation}\label{fundaC}
%b_{N,s} :=\pi^{-N/2} \, 4^{-s}  \, \frac{\Gamma(\frac{N-2s}{2})}{\Gamma(s)},
%\end{equation}
%
 is the  fundamental solution of $(-\Delta)^s$ and 
$H^s_\Omega(x,\cdot)$ solves the equation
\begin{equation}\label{regular-part-of-Green}
\left\{ \begin{array}{rcll} (-\Delta)^s H^s_\Omega(x,\cdot)&=& 0  &\textrm{in }\Omega, \\ H^s_\Omega(x,z)&=&F_s(x,z)&%\textrm{for } x 
\textrm{in }\R^N\setminus\Omega. \end{array}\right. 
\end{equation}
By standard regularity theory we know that $H_\Omega^s(x,\cdot)\in C^\infty(\Omega)$. The Robin function $R^s_\Omega\in C^\infty(\Omega)$ is then defined by 
$$
R^s_\Omega(x) := H^s_\Omega(x,x).
$$
We then have the following result regarding the fractional Robin function. 
\begin{cor}\label{cor6}
 Let $s$ in $(0,1)$ and  $Y$ be a globally $C^{1,1}$ vector field in  $\R^N$ with $N>2s$. Then, the Robin function $R_\Omega^s$ is shape differentiable in the direction of $Y$ in the sense that, for all $x\in \Omega$, the limit 
 \begin{equation*}
DR_\Omega^s(Y)(x):=\lim_{t\to 0}\frac{ H^s_{\Omega_t}(\Phi_t(x),\Phi_t(x))- H^s_{\Omega}(x,x)}{t}-2\nabla_xH_\Omega^s(x,x)\cdot Y(x),
 \end{equation*}
exists and is finite. Moreover, there holds
\begin{align}\label{var-Robin-func}
DR_\Omega^s(Y)(x)= -\Gamma^2(1+s)\int_{\partial\Omega}\Big(\gamma_0^s(G_\Omega^s(x,\cdot))\Big)^2Y\cdot\nu \;d\sigma.
\end{align}
\end{cor}
\par\,
The proof of Corollary \ref{cor6} is given in Section \ref{sec-ftt}.
\par\;

\section{Notations}\label{notations}
For the reader convenience, we record here several notations and recall some classical estimates that will be used all along the manuscript. All along the paper, except otherwise stated, $\Omega$ will be a bounded open set of class $C^{1,1}$ in $\R^N$ and we shall denote by $\delta:=\dist(\R^N\setminus\Omega, \cdot)-\dist(\Omega,\cdot)$ the signed distance function to the boundary of the domain $\Omega$.
\par\;
We fix $\rho\in C^\infty_c(-2,+2)$ such that $\rho\equiv 1$ in $(-1,+1)$.  For any $k\in \mathbb N^*$, we define 
\begin{equation}\label{eta-k}
    \xi_k: \R^N\to \R,\; y\mapsto \xi_k(y):=1-\rho(k\delta(y)),
\end{equation}
where $\delta$ is the \textit{signed} distance function to the boundary. 
%Moreover, we assume that $\delta$ is positive in $\Omega$ and negative in $\R^N\setminus\Omega$. 
For a fixed $x\in \Omega$, and $\mu\in (0,1)$, we also define the function $\phi_x^{\mu}$ by:
\begin{equation}\label{psi-mu}
     \phi_x^{\mu}: \R^N\to \R,\; y\mapsto \phi_x^{\mu}(y):=1-\rho_\mu^x(y),
\end{equation}
where 
\begin{equation*}
\rho_\mu^x(y):=\rho\Big(\frac{8}{\delta^2_\Omega(x)}\frac{|x-y|^2}{\mu^2}\Big).
\end{equation*}
Here $\delta_\Omega:=\dist(\cdot,\partial \Omega)$ is the distance function to the boundary of the domain $\Omega$.
\par\;
Except otherwise stated, $Y$ will be a globally $C^{1,1}$ vector field in $\R^N$, and we define the weight functions
\begin{equation} \label{defoY-notations}
\omega_Y(y,z):=\frac{c_{N,s}}{2}\Big[\div Y(y)+\div Y(z)-(N+2s)\frac{(Y(y)-Y(z))\cdot(y-z)}{|y-z|^2}\Big],
\end{equation}
for all $y,z\in \R^N$ and 
\begin{equation} \label{defWsam}
 \omega^x_{Y}(y,z) := \frac{c_{N,s}}{2}\Big[ 2\div Y(x)-(N+2s)\frac{[DY(x)\cdot(y-z)]\cdot(y-z)}{|y-z|^2}\Big],  
\end{equation}
for a fixed $x\in\Omega$ and for all $y,z\in \R^N$. We also define the deformation kernel $\kappa_Y(\cdot,\cdot)$ by: 
\begin{equation}\label{kappaY}
\kappa_Y(y,z):=\omega_Y(y,z)|y-z|^{-2s-N}.  
\end{equation}

Next, we let $\textrm{Jac}_{\Phi_t}$ be the Jacobian of the transformation $\Phi_t$ defined in \eqref{transformations} and  define the family of symmetric kernels $\kappa_t(\cdot,\cdot)$ by:
\begin{equation}\label{ktyz}
    \kappa_t(y,z):=\frac{c_{N,s}}{2}\frac{\textrm{Jac}_{\Phi_t}(y)\textrm{Jac}_{\Phi_t}(z)}{|\Phi_t(y)-\Phi_t(z)|^{N+2s}}\quad\text{for $y,z\in \R^N$ with $y\neq z$.}
\end{equation}
\par\;

By direct computation using the hypothesis \eqref{transformations}  (see e.g \cite[Eq (3.3)]{DFW}) we have --for $|t|\ll 1$-- sufficiently small:
\begin{equation}\label{expansion-k-t-notations}
 \kappa_t(y,z)-\frac{c_{N,s}}{2}|y-z|^{-N-2s}-t\kappa_Y(y,z)=O(t^2)|y-z|^{-N-2s},
\end{equation}
uniformly for all $y\neq z\in \R^N$ where $\kappa_Y(\cdot,\cdot)$ is defined as in \eqref{kappaY}. 
\par\;

In view of \eqref{expansion-k-t-notations}, it is clear that the kernel $\kappa_t(\cdot,\cdot)$ is uniformly elliptic. In other words, we have:
\begin{equation}\label{ellptic-kernel}
    \lambda|y-z|^{-2s-N}\leq \kappa_t(y,z)\leq \Lambda|y-z|^{-2s-N}\qquad \forall\, y\neq z\in \R^N,
\end{equation}
for some $\lambda, \Lambda>0$ that are independent of $t$. 
\par\;

For all $\beta\in\R$, we also have the following asymptotic (for $|t|\ll 1$ small enough), see e.g \cite{DFW}:
\begin{equation}\label{UExp}
    \big|\Phi_t(y)-\Phi_t(z)\big|^{\beta}=|y-z|^{\beta}\Big(1+2t\frac{(Y(y)-Y(z))\cdot(y-z)}{|y-z|^2}+O(t^2)\Big)^{\frac{\beta}{2}}\quad\text{uniformly $\forall\,y\neq z\in\R^N$}.
\end{equation}
This, in particular, yields
\begin{align}\label{ellip-reg}
\frac{1}{c}|y-z|^\beta \leq  \big|\Phi_t(y)-\Phi_t(z)\big|^{\beta}\leq c |y-z|^\beta\quad\text{for some $c>0$ independent of $t$}.
\end{align}
\par\;

Finally, the following estimate will be important later on. For  $|t|\ll 1$ small enough, there holds:
\begin{equation}\label{diff-quotient-fund-sol}
    \Big||x_t-y_t|^{2s-N}-|x_t-z_t|^{2s-N}\Big|\leq C f_x(y,z),
\end{equation}
for some constant $C>0$ that is independent of $t,y$ and $z$ and where 
\par\;

\begin{equation}\label{f_xyz}
f_x(y,z)=
\begin{cases}
\max\big(|x-y|^{2s-N},|x-z|^{2s-N}\big) & \text{if } 2s<1,\\
\big|y-z\big|\max\big(|x-y|^{2s-N-1},|x-z|^{2s-N-1}\big) & \text{if } 2s>1,\\
|y-z|^\beta\max\big(|x-y|^{1-N-\beta},|x-z|^{1-N-\beta}\big)\quad\text{for all $\beta\in (0,1)$}&\text{if $s=1/2$.}
\end{cases}
\end{equation}
\par\;
\par\;

In the case for $2s<1$, the bound \eqref{diff-quotient-fund-sol} is a easy consequence of the uniform estimate \eqref{ellip-reg}. The case $2s>1$ follows by using the elementary estimate 
$$
|a^{2s-N}-b^{2s-N}|\leq C(N,s))|a-b|\max\Big(\frac{1}{a^{N-2s+1}}, \frac{1}{b^{N-2s+1}}\Big) ,
$$
with $a=|x_t-y_t|$ and $b=|x_t-z_t|$ combine with the uniform estimate \eqref{ellip-reg}. %Indeed, using the above estimate, we get
%\begin{align*}
%\Big||x_t-y_t|^{2s-N}-|x_t-z_t|^{2s-N}\Big|
%&\leq C\big|y_t-z_t\big|\max\Big(\big|x_t-y_t\big|^{2s-N-1},\big|x_t-z_t\big|^{2s-N-1}\Big),\\
%&\leq C\big|y-z\big|\max\Big(\big|x-%y\big|^{2s-N-1},\big|x-z\big|^{2s-N-1}\Big),
%\end{align*}
%where in the last line we have used again the uniform estimate \eqref{ellip-reg}. 
The case $s=1/2$ follows similarly by using instead the estimate 
\begin{align*}
\Big|a^{1-N}-b^{1-N}\Big|\leq C(N,\beta)\max\big(a^{1-N-\beta},b^{1-N-\beta}\big)|a-b|^\beta\quad\text{for all $\beta\in (0,1)$},
\end{align*}
applied with $a=|x_t-y_t|$ and $b=|x_t-z_t|$.
%Finally, we recall the following integral estimate (see e.g \cite[Lemma 1.4]{Abdellaoui}) that will also be use repeatedly in this paper.
%\par\;
%Let $x,y$ in $B_R(0)$ and $\alpha,\beta\in (-\infty,N)$. Then, there exists a constant $C=C(N,\alpha,\beta, R)>0$ such that 
%\begin{align}\label{Abdellaoui1}
 %  \int_{B_R(0)}\frac{dz}{|x-z|^\alpha|y-z|^\beta}\leq C(1+|x-y|^{N-\alpha-\beta})\qquad\text{if \; $N-\alpha-\beta\neq 0$},
%\end{align}
%and 
%\begin{align}\label{Abdellaoui2}
%\int_{B_R(0)}\frac{dz}{|x-z|^\alpha|y-z|^\beta}\leq C(1+\big|\log|x-y|\big|\Big)\qquad\text{if \; $N-\alpha-\beta= 0$}. 
%\end{align}
\section{Preliminary estimates}
In this section, we collect a couple of preliminary results that will be used later.
We start with two fundamental lemmas.
Their proofs, given respectively in Appendix \ref{Appendix.A} and Appendix \ref{Appendix.B}, rely on elementary calculus and the assumption \eqref{transformations}.
%estimates will play a crucial role in the computations that follows that will be used later in the paper. We start with the following simple estimate that will be important for the computations that follow. Its proof is given in the Appendix \ref{Appendix.A}.
%
\begin{Lemma}\label{Useful-bounds}
Let $\Phi_t$ satisfy the hypotheses \eqref{transformations} and let $\beta$ in $\R$. Then, for $|t|\ll 1/2$ sufficiently small and for all $y,z\in \R^N$, there holds:
\begin{align}\label{WWW}
\Big|\big|\Phi_t(y)-\Phi_t(y+z)\big|^{-\beta}-\big|\Phi_t(y)-\Phi_t(y-z)\big|^{-\beta}\Big|\leq C|z|^{1-\beta},    
\end{align} 
for some constant $C>0$ that is independent of $t,y$ and $z$. Moreover, letting $\kappa_t(\cdot,\cdot)$ be as in \eqref{ktyz}, we also have
\begin{equation}\label{Sauveur1}
    |\kappa_t(y+h,z+h)-\kappa_t(y,z)|\leq \frac{C|h|}{|y-z|^{N+2s}}\qquad\text{$\forall\,y,z,h\in \R^N$},
\end{equation}
for some $C>0$ that is independent of $y,z,h$ and $t$.
\end{Lemma}
\begin{Lemma}\label{FundLemma2}
    Let $\kappa_t(\cdot,\cdot)$ be given as in \eqref{ktyz} and set, for $t\neq 0$,
\begin{equation*}
\overline{\kappa_t}(y,z):=\frac{\kappa_t(y,z)-\kappa_0(y,z)}{t}\qquad\text{for $y\neq z$ in $\R^N$.}
\end{equation*}
Let $\overline{\kappa_t}^o(\cdot,\cdot)$ be the odd part of $z\mapsto \overline{\kappa_t}(y,y+z)$, i.e.
\begin{equation*}
\overline{\kappa_t}^o(y,z)=\frac{\overline{\kappa_t}(y,y+z)-\overline{\kappa_t}(y,y-z)}{2}.
\end{equation*}
Then there holds
\begin{equation}\label{SauVsec3}
    \big|\overline{\kappa_t}^o(y,z)\big|\leq C\min(|z|^{-2s-N},|z|^{1-2s-N})
\end{equation}
for some constant $C>0$ that is independent of $t$, $y$ and $z$.
\end{Lemma}

The following proposition will play an essential role in the proof of the differentiability of the mapping $t\mapsto H^s_{\Omega_t}(\Phi_t(x),\Phi_t(\cdot))\in C(\overline{\Omega})$ in Lemma \ref{diff-reg-part}. 
 \begin{Lemma}\label{LM1}
Let $s$ in $(0,1)$ and $x\in\Omega$ be fixed. Consider the function $f_t[H_\Omega^s(x,\cdot)]:\Omega\to \R$ defined by:
\begin{equation}
    f_t[H_\Omega^s(x,\cdot)](y):=p.v\int_{\R^N}(H_\Omega^s(x,y)-H_\Omega^s(x,z))\frac{\kappa_t(y,z)-\kappa_0(y,z)}{t}dz.
\end{equation}
Then, for $|t|\ll 1$ sufficiently small, there holds 
\begin{equation}\label{pointwise-est-fyreg}
    |f_t[H_\Omega^s(x,\cdot)](y)|\leq C\delta^{-s}_\Omega(y).
\end{equation}
for some $C>0$ that is independent of $t$ and $y$. Moreover, we have, as $t$ goes to zero,
\begin{align}\label{pointwise-ft-lim}
    f_t[H_\Omega^s(x,\cdot)]\quad\to\quad f_Y[H_\Omega^s(x,\cdot)]\quad\text{pointwise in $\Omega$}.
\end{align}
where we set 
\begin{equation}
    f_Y[H_\Omega^s(x,\cdot)](y):=p.v\int_{\R^N}\frac{H_\Omega^s(x,y)-H_\Omega^s(x,z)}{|y-z|^{N+2s}}\omega_Y(y,z)dz.
\end{equation}
with $\omega_Y(\cdot,\cdot)$ given as in \eqref{defoY-notations}. Moreover, for all $y$ in $\Omega$, we have
\begin{equation}\label{pointwise-ftY-es}
 |f_Y[H_\Omega^s(x,\cdot)] (y)|  \leq C\delta^{-s}_\Omega(y).
\end{equation}
\end{Lemma} 
\begin{proof}
We set
\begin{align*}
\overline{\kappa_t}(y,z):=\frac{\kappa_t(y,z)-\kappa_0(y,z)}{t}.
\end{align*}
For $z$ in $\R^N$ we define the even and odd parts of $\overline{\kappa_t}(y,y+z)$ by
\begin{align*}
\overline{\kappa_t}^{e}(y,z):=\frac{\overline{\kappa_t}(y,y+z)+\overline{\kappa_t}(y,y-z)}{2}
\qquad\text{and}\qquad
\overline{\kappa_t}^{o}(y,z):=\frac{\overline{\kappa_t}(y,y+z)-\overline{\kappa_t}(y,y-z)}{2}.
\end{align*}
With the change of variable $z\mapsto y+z$, we write
\begin{align*}
f_t[H_\Omega^s(x,\cdot)](y)
&=\mathrm{p.v.}\int_{\R^N}\big(H_\Omega^s(x,y)-H_\Omega^s(x,y+z)\big)\overline{\kappa_t}(y,y+z)\,dz\\
&=f_t^{e}(y)+f_t^{o}(y),
\end{align*}
where
\begin{align*}
f_t^{e}(y)
&:=\mathrm{p.v.}\int_{\R^N}\big(H_\Omega^s(x,y)-H_\Omega^s(x,y+z)\big)\overline{\kappa_t}^{e}(y,z)\,dz,\\
f_t^{o}(y)
&:=\mathrm{p.v.}\int_{\R^N}\big(H_\Omega^s(x,y)-H_\Omega^s(x,y+z)\big)\overline{\kappa_t}^{o}(y,z)\,dz.
\end{align*}
We fix $y$ in $\Omega$ and set $\rho:=\delta_\Omega(y)/2$.

\medskip

\noindent\textbf{Estimate of the even part.}
Since $\overline{\kappa_t}^{e}(y,\cdot)$ is even, we may symmetrize:
\begin{align*}
2f_t^{e}(y)=\int_{\R^N}\big(2H_\Omega^s(x,y)-H_\Omega^s(x,y+z)-H_\Omega^s(x,y-z)\big)\overline{\kappa_t}^{e}(y,z)\,dz.
\end{align*}
By \cite[Lemma A.1]{KN}, $H_\Omega^s(x,\cdot)\in C^2_{\loc}(\Omega)$, and we also have $H_\Omega^s(x,\cdot)\in C^s(\R^N)$.
Hence there are constants $C_1,C_2>0$ (depending on $x$, but independent of $t$, $y$) such that
\begin{align*}
\big|2H_\Omega^s(x,y)-H_\Omega^s(x,y+z)-H_\Omega^s(x,y-z)\big|
\le C_1|z|^{2}1_{B_{\rho/2}}(z)+C_2|z|^{s}1_{\R^N\setminus B_{\rho/2}}(z).
\end{align*}
Moreover, by \eqref{expansion-k-t-notations} we have $|\overline{\kappa_t}^{e}(y,z)|\le C|z|^{-N-2s}$.
Combining the above estimates yields
\begin{align*}
|f_t^{e}(y)|
&\le C\int_{B_{\rho/2}}|z|^{2-N-2s}\,dz + C\int_{\R^N\setminus B_{\rho/2}}|z|^{s-N-2s}\,dz\\
&\le C\rho^{2-2s}+C\rho^{-s}\le C\delta_\Omega(y)^{-s}.
\end{align*}

\medskip

\noindent\textbf{Estimate of the odd part.}
By \cite[Eq. (16)]{Sidy-Franck-BP} we have
\begin{align*}
|\overline{\kappa_t}^{o}(y,z)|\le C\min\big(|z|^{-N-2s},|z|^{1-N-2s}\big).
\end{align*}
Since $H_\Omega^s(x,\cdot)\in C^s(\R^N)$, we obtain
\begin{align*}
|f_t^{o}(y)|
&\le C\int_{\R^N}|z|^{s}\min\big(|z|^{-N-2s},|z|^{1-N-2s}\big)\,dz\\
&\le C\int_{B_1}|z|^{1-s-N}\,dz + C\int_{\R^N\setminus B_1}|z|^{-N-s}\,dz
\le C.
\end{align*}
In particular, $|f_t^{o}(y)|\le C\le C\delta_\Omega(y)^{-s}$ and therefore
\begin{align*}
\big|f_t[H_\Omega^s(x,\cdot)](y)\big|\le C\delta_\Omega(y)^{-s}.
\end{align*}

\medskip

\noindent\textbf{Convergence.}
We use the pointwise convergences following from \eqref{expansion-k-t-notations} (for the even part) and \cite[Eq. (16)]{Sidy-Franck-BP} (for the odd part), together with the previous integrable domination, and we apply the dominated convergence theorem.
We obtain $f_t^{e}(y)\to f_Y^{e}(y)$ and $f_t^{o}(y)\to f_Y^{o}(y)$, where
\begin{align*}
f_Y^{e}(y)
&:=\frac12\int_{\R^N}\frac{2H_\Omega^s(x,y)-H_\Omega^s(x,y+z)-H_\Omega^s(x,y-z)}{|z|^{N+2s}}\omega_Y^{e}(y,z)\,dz,\\
f_Y^{o}(y)
&:=\mathrm{p.v.}\int_{\R^N}\frac{H_\Omega^s(x,y)-H_\Omega^s(x,y+z)}{|z|^{N+2s}}\omega_Y^{o}(y,z)\,dz,
\end{align*}
with
\begin{align*}
\omega_Y^{e}(y,z):=\frac{\omega_Y(y,y+z)+\omega_Y(y,y-z)}{2}
\qquad\text{and}\qquad
\omega_Y^{o}(y,z):=\frac{\omega_Y(y,y+z)-\omega_Y(y,y-z)}{2}.
\end{align*}
Since $f_Y=f_Y^{e}+f_Y^{o}$, this coincides with the definition of $f_Y[H_\Omega^s(x,\cdot)](y)$.
Moreover, the same domination yields $|f_Y(y)|\le C\delta_\Omega(y)^{-s}$.

Finally, the same argument applies on compact subsets of $\Omega$, using the local $C^2$ regularity of $H_\Omega^s(x,\cdot)$, and we obtain convergence in $C^0_{\mathrm{loc}}(\Omega)$.
\end{proof}
The following simple result is fundamental. It says that even though $G_\Omega^s(x,\cdot)$ does not have a finite energy, the dual product $g[\psi](t):=\langle G^s_{\Omega_t}(x_t,(\cdot)_t),  \psi\rangle$ still make sense for any sufficiently regular function $\psi$. This observation will play a crucial for the proof of the results stated above. For the sake of simplicity, we will adopt the following notation .
\begin{equation}\label{z_t}
z_t:=\Phi_t(z),\quad g_t(x_t,z_t):=G^s_{\Omega_t}(x_t,z_t)\quad\text{and}\quad g_0(x,z)=G_\Omega^s(x,z)\quad\text{for $z\in \R^N$.}
\end{equation}
\begin{Proposition}\label{FundLemma}
  Let $s$ in $(0,1)$ and $x\in \Omega$. 
  Let $\psi$ in $C^\infty_0(\Omega)$ and define 
 \begin{align}\label{FundRest}
g[\psi](t,x):=\iint_{\R^{2N}} \frac{\big(g_t(x_t,y_t)-g_t(x_t,z_t)\big)\big(\psi(y)-\psi(z)\big)}{|y-z|^{N+2s}} 
\, dydz \quad\text{for $t\neq 0$},
\end{align} 
and 
\begin{align}\label{FundRes0-def}
g[\psi](0,x):=\iint_{\R^{2N}} \frac{\big(g_0(x,y)-g_0(x,z)\big)\big(\psi(y)-\psi(z)\big)}{|y-z|^{N+2s}}
\, dydz.
\end{align} 
Then we have $g[\psi](t,x)<\infty$ for all $|t|\ll 1$ sufficiently small.
\end{Proposition}
\begin{proof} We give the proof of $g[\psi](t,x)<+\infty$ for $t\neq 0$. The case $t=0$ is similar we therefore skip its proof and leave it to the interested reader.
\par\;

We use the decomposition $g_t(x_t,z_t) = b_{N,s}|x_t-z_t|^{2s-N}-H^s_{\Omega_t}(x_t,z_t)$ of the fractional Green function, see   \eqref{Eq-apliting-of-Green} to split the quantity at stake into the sum of two quantities:
$$
g[\psi](t,x)=g_1[\psi](t,x)-g_2[\psi](t,x),
$$
where $g_1[\psi](t,x)$ and $g_2[\psi](t,x)$ are respectively given by:
\begin{align*}
g_1[\psi](t,x)&:=\iint_{\R^{2N}}\big(|x_t-y_t|^{2s-N}-|x_t-z_t|^{2s-N}\big)\big(\psi(y)-\psi(z)\big)|y-z|^{-2s-N}dydz,\\
g_2[\psi](t,x)&:= \iint_{\R^{2N}} \big(H^s_{\Omega_t}(x_t,y_t)-H^s_{\Omega_t}(x_t,z_t)\big)\big(\psi(y)-\psi(z)\big) |y-z|^{-2s-N}dydz.
\end{align*}
We start by estimating the contribution coming from $g_1[\psi](t,x)$. We claim that
\begin{equation}\label{Fund-es-btx}
|g_1[\psi](t,x)|\leq C\qquad\text{for some $C>0$ that is independent of $t$.}
\end{equation}
To prove the claim, we recall that in view of \eqref{diff-quotient-fund-sol}, we have 
\begin{equation}\label{sub-Fund-es-btx}
 |g_1[\psi](t,x)|\leq C\iint_{\R^{2N}}\frac{\big|\psi(y)-\psi(z)\big|}{|y-z|^{N+2s}}f_x(y,z)dydz,
\end{equation}
where
\begin{equation}\label{f_xyz-1}
f_x(y,z)=
\begin{cases}
\max\big(|x-y|^{2s-N},|x-z|^{2s-N}\big) & \text{if } 2s<1,\\
\big|y-z\big|\max\big(|x-y|^{2s-N-1},|x-z|^{2s-N-1}\big) & \text{if } 2s>1,\\
|y-z|^\beta\max\big(|x-y|^{1-N-\beta},|x-z|^{1-N-\beta}\big)\quad\text{for all $\beta\in (0,1)$}&\text{if $s=1/2$.}
\end{cases}
\end{equation}
Now observe that RHS of the above identity is always finite. To see this, let $R\gg 1$ so that $\supp(\psi)\Subset B_{R/2}(0)$ and write
\begin{align}\label{funnny-decomp}
\iint_{\R^{2N}}\frac{\big|\psi(y)-\psi(z)\big|}{|y-z|^{N+2s}}f_x(y,z)dydz=\iint_{B_{R}\times B_R}\cdots dydz+2\int_{B_{R/2}}\int_{\R^N\setminus B_R}\cdots dydz.
\end{align}
%
%\underline{\textbf{Case 1}}: $2s<1$. In this case, in view of \eqref{ellip-reg}, we have the bound:
%\begin{align*}
 %   |g_1[\psi](t,x)|&\leq C \iint_{\R^{2N}}\big(|x-y|^{2s-N}+|x-z|^{2s-N}\big)\big|\psi(y)-\psi(z)\big||y-z|^{-2s-N}dydz,
%\end{align*}
%for some $C>0$ that is independent of $t$. Next, let $R\gg 1$ so that $\supp(\psi)\Subset B_{R/2}(0)$ and write 
%\begin{align*}
%&\iint_{\R^{2N}}\big(|x-y|^{2s-N}+|x-z|^{2s-N}\big)\big|\psi(y)-\psi(z)\big||y-z|^{-2s-N}dydz\\
%&=\iint_{B_{R}\times B_R}\big(|x-y|^{2s-N}+|x-z|^{2s-N}\big)\big|\psi(y)-\psi(z)\big||y-z|^{-2s-N}dydz\\
%&+2\int_{B_{R/2}}|\psi(y)|\int_{\R^N\setminus B_R}\frac{|x-y|^{2s-N}+|x-z|^{2s-N}}{|y-z|^{N+2s}}dzdy.
%\end{align*}
For $y\in B_{R/2}$ and $z\in \R^N\setminus B_R$,  we can find $C_1=C_1(R)>0$ so that $|y-z|>C_1(1+|z|^2)^{1/2}$. Consequently, we have  
\begin{align*}
\int_{B_{R/2}}\int_{\R^N\setminus B_R}\frac{\big|\psi(y)-\psi(z)\big|}{|y-z|^{N+2s}}f_x(y,z) dydz
\leq C\int_{B_{R/2}}|\psi(y)|\int_{\R^N\setminus B_R}\frac{f_x(y,z)}{(1+|z|^2)^{\frac{N+2s}{2}}}dzdy,
%\\
%&\leq C_2\Bigg(\int_{B_{R/2}}\frac{dy}{|x-y|^{N-2s}}\int_{\R^N\setminus B_R}\frac{dz}{(1+|z|^2)^{\frac{N+2s}{2}}}+\int_{\R^N\setminus B_R}\frac{dz}{|x-z|^{N-2s}(1+|z|^2)^\frac{N+2s}{2}}\Bigg),\\
%&<+\infty.
\end{align*}
But 
\begin{equation}\label{iR(x)-finite}
i_R(x):=\int_{B_{R/2}}\int_{\R^N\setminus B_R}\frac{f_x(y,z)}{(1+|z|^2)^{\frac{N+2s}{2}}}dzdy<+\infty    \quad\text{for all $s\in (0,1)$}.
\end{equation}
Indeed, in view of \eqref{f_xyz-1}, we have, for $2s<1$, that 
\begin{align*}
 i_R(x)&\leq C    \int_{B_{R/2}}\int_{\R^N\setminus B_R}\frac{|x-y|^{2s-N}+|x-z|^{2s-N}}{(1+|z|^2)^{\frac{N+2s}{2}}}dzdy\\
 &\leq 2C\int_{B_{R/2}}\frac{dy}{|x-y|^{N-2s}}\int_{\R^N\setminus B_R}\frac{1}{(1+|z|^2)^{\frac{N+2s}{2}}}dz<+\infty.
\end{align*}
In the case $2s>1$, we have instead 
\begin{align}
i_R(x)&\leq C_1 \int_{B_{R/2}}\int_{\R^N\setminus B_R}\frac{\big|y-z\big|\max\big(|x-y|^{2s-N-1},|x-z|^{2s-N-1}\big)}{(1+|z|^2)^{\frac{N+2s}{2}}}\nonumber\\
&\leq C_2\int_{B_{R/2}}\int_{\R^N\setminus B_R}\frac{|z|}{|x-y|^{N-(2s-1)}(1+|z|^2)^{\frac{N+2s}{2}}}dydz\nonumber\\
&\leq C_2\int_{B_{R/2}}\frac{dy}{|x-y|^{N-(2s-1)}}\int_{\R^N\setminus B_R}\frac{|z|}{(1+|z|^2)^{\frac{N+2s}{2}}}dz<+\infty.\nonumber
\end{align}
And finally if $s=1/2$, we have 
\begin{align*}
i_R(x)&\leq C_1\int_{B_{R/2}}\int_{\R^N\setminus B_R}\frac{|y-z|^\beta\max\big(|x-y|^{1-N-\beta},|x-z|^{1-N-\beta}\big)}{(1+|z|^2)^{\frac{N+2s}{2}}}dz\\
&\leq C_2\int_{B_{R/2}}\frac{dy}{|x-y|^{N-(1-\beta)}}\int_{\R^N\setminus B_R}\frac{|z|^\beta}{(1+|z|^2)^{\frac{N+1}{2}}}dz <+\infty.
\end{align*}
In all three cases we have proved that $i_R(x)$ has a finite value which gives \eqref{iR(x)-finite}.
\par\;
On the other hand, since $\psi\in C^\infty_0(\Omega)$, it is clear that 
\begin{align}\label{jR(x)-finite}
j_R(x):=\iint_{B_R\times B_R}\frac{\big|\psi(y)-\psi(z)\big|}{|y-z|^{N+2s}}f_x(y,z)\,dy\,dz 
\leq C\iint_{B_R\times B_R}\frac{f_x(y,z)}{|y-z|^{N+2s-1}}\,dy\,dz<+\infty.
\end{align}
The finiteness of the RHS of \eqref{jR(x)-finite} follows from \eqref{f_xyz} by checking the integrability near the diagonal $y=z$ (and the integrability of the mild singularities in $y=x$ or $z=x$).
If $2s<1$, the function $f_x$ is locally bounded on $B_R\times B_R$ and the integrand behaves like $|y-z|^{-N-2s+1}$, which is integrable around $y=z$ since $2s<1$.
If $2s>1$, we have $f_x(y,z)\leq C|y-z|\max\big(|x-y|^{2s-N-1},|x-z|^{2s-N-1}\big)$, hence the integrand behaves like $|y-z|^{-N-2s+2}$ and is integrable around $y=z$ since $2s<2$.
If $s=\frac12$, choosing $\beta\in (0,1)$ we have $f_x(y,z)\leq C|y-z|^\beta\max\big(|x-y|^{-N-\beta},|x-z|^{-N-\beta}\big)$, so the integrand behaves like $|y-z|^{-N+\beta}$ and is integrable around $y=z$ since $\beta>0$.
\begin{equation}\label{dual-prod-f_xyz-psi}
\iint_{\R^{2N}}\frac{\big|\psi(y)-\psi(z)\big|}{|y-z|^{N+2s}}f_x(y,z)dydz<+\infty.
\end{equation}
That $g_1[\psi](t,x)$ is finite now follows by combining \eqref{sub-Fund-es-btx}, \eqref{funnny-decomp}, \eqref{iR(x)-finite}, \eqref{jR(x)-finite} and \eqref{dual-prod-f_xyz-psi}
\par\;

We next show that 
\begin{equation}\label{FundRes00}
|g_2[\psi](t,x)|< +\infty.
\end{equation}
For that aim, we choose $K\Subset K'\Subset\Omega$ with $K, K'$ compact and such that $\supp \psi\subset K$ and write 
\begin{align}\label{Homegt-dec}
g_2[\psi](t,x)&=  \iint_{K'\times  K'} \big(H^s_{\Omega_t}(x_t,y_t)-H^s_{\Omega_t}(x_t,z_t)\big)\big(\psi(y)-\psi(z)\big)|y-z|^{-2s-N}   \, dy\, dz \nonumber\\
&+2\int_{\supp(\psi)}\psi(y)\int_{\R^N\setminus K'}\big(H^s_{\Omega_t}(x_t,y_t)-H^s_{\Omega_t}(x_t,z_t)\big)|y-z|^{-2s-N}\, dz\,dy.
\end{align}
For $t$ sufficiently small, we know by regularity (see e.g \cite[Proposition 1.3]{AROS}) that $H^s_{\Omega_t}(x_t,\cdot)\in C^s(\R^N)$. Therefore for some constant $C>0$ that might depend on $t$, we may write, for $|t|\ll 1$ small enough:
\begin{align}\label{Homegt-es1}
&\iint_{K'\times  K'} \big|H^s_{\Omega_t}(x_t,y_t)-H^s_{\Omega_t}(x_t,z_t)\big|\big|\psi(y)-\psi(z)\big||y-z|^{-2s-N}dydz\nonumber\\
&\leq C\iint_{K'\times  K'} \frac{|y_t-z_t|^s}{|y-z|^{N+2s-1}}dydz\nonumber\\
&\leq C\iint_{K'\times  K'} \frac{dy\,dz}{|y-z|^{N+s-1}}<+\infty,
\end{align}
where we used \eqref{ellip-reg}. On the other hand, 
since $|H^s_{\Omega_t}(x_t,z_t)|\leq C |x_t-z_t|^{2s-N}\leq C|x-z|^{2s-N}$, we have 
\begin{align}\label{Homegt-es2}
&\Bigg|\int_{\supp(\psi)}\psi(y)\int_{\R^N\setminus K'}\frac{H^s_{\Omega_t}(x_t,y_t)-H^s_{\Omega_t}(x_t,z_t)}{|y-z|^{N+2s}}\,dz\,dy \Bigg|\nonumber\\
&\leq C\int_{K}\int_{\R^N\setminus K'}\Big(\frac{1}{|x-y|^{N-2s}}+\frac{1}{|x-z|^{N-2s}}\Big)\frac{\, dy\, dz }{1+|z|^{N+2s}} <+\infty.   
\end{align}
Combining \eqref{Homegt-dec}, \eqref{Homegt-es1}, and \eqref{Homegt-es2} we obtain \eqref{FundRes00}. The estimate \eqref{Fund-es-btx} together with \eqref{FundRes00} shows that $g[\psi](t,x)<+\infty$. The proof is finished and we are done.
\end{proof}
\par\;
The following lemma will be useful in what follows, as it will allow us, thanks to a regularity result that we prove along the way, to obtain uniform control over the difference quotient defined in \eqref{defStx}. We continue to use the notation introduced in \eqref{z_t}.
\begin{Lemma}\label{unifest-gtx}
Let $x$ in $\Omega$ be fixed and define
\begin{equation}\label{defgtx}  
  g^t_x:\R^N\to\R, \; y\mapsto g^t_x(y):=b_{N,s}\frac{|x_t-y_t|^{2s-N}-|x-y|^{2s-N}}{t}\quad\text{for $t\neq 0$}. 
\end{equation}
Then $|g^t_x(y)|\leq C|x-y|^{2s-N}$ for all $y\neq x$ in $\R^N$. Moreover, for all $\alpha\in (0,1)$ and $|t|\ll 1$ sufficiently small, there holds
\begin{equation}\label{unif-est-extdata}
|g^t_x(y)|+\frac{|g^t_x(y)-g^t_x(z)|}{|y-z|^\alpha}   \leq C_0\qquad\textup{for all $y\in\partial\Omega, \,z\in \R^N\setminus\Omega$}, 
\end{equation}
for some $C_0=C_0(N,s,\alpha, x, \Omega)>0$ that is independent of $t$.
\end{Lemma}
\begin{proof}
The estimate $|g_x^t(y)|\leq C|x-y|^{2s-N}$ is an easy consequence of \eqref{UExp}. Next, let $x$ in $\R^N$ be fixed and set
$$
b_t^x(y):=|x_t-y_t|^{2s-N}.
$$
Then, up to a positive constant, we may write by the fundamental theorem of calculus that
\begin{equation}\label{expansion-bt}
g^t_x(y)= \frac{b_t^x(y)-b_0^x(y)}{t} = \frac{1}{t}\int_{0}^t\partial_rb_r^x(y)dr.
\end{equation}
Observe that to get \eqref{unif-est-extdata}, it is enough to consider $y\in \partial\Omega$ and $z\in B_{R}(0)\bigcap \Omega^c$ for some $R=R(\Omega)\gg 1$ sufficiently large so that $\Omega\Subset B_{\frac{R}{2}}(0)$. \par\;
Indeed, if $z\in \R^N\setminus B_{R}(0)$, then setting 
$$
a_r(x,y):=\frac{|\partial_rx_r-\partial_ry_r|}{|x_r-y_r|}, 
$$
A direct calculation gives 
\begin{align*}
&\frac{\big|\partial_rb_r^x(y)-\partial_rb_r^x(z)\big|}{N-2s}\leq a_r(x,y) \big|x_r-y_r\big|^{2s-N}+ a_r(x,z)\big|x_r-z_r\big|^{2s-N},
\end{align*}
Since $\partial_r\Phi_r\in C^{1,1}(\R^N,\R^N)$, and $|x_r-y_r|>C|x-y|$ by \eqref{ellip-reg}, we have $a_r(x,\cdot)\in L^\infty(\R^N)$. Moreover, in view of \eqref{ellip-reg} we also have 
$$
\big|x_r-y_r\big|^{2s-N}+ \big|x_r-z_r\big|^{2s-N}\leq C(s,N,x,y)\leq C(s,N,\Omega,x,\alpha,R)|y-z|^\alpha,
$$
for any $\alpha\in (0,1)$ since $|y-z|>R/2>0$. In conclusion we have 
$$
\frac{\big|\partial_rb_r^x(y)-\partial_rb_r^x(z)\big|}{N-2s}\leq C|y-z|^\alpha\qquad\text{for all $\alpha\in (0,1)$},
$$
for some $C=C(N,s,x, \Omega,R)>0$ that is independent of $r$. In view of the representation \eqref{expansion-bt}, this gives the estimate \eqref{unif-est-extdata} for $y\in \partial\Omega$ and $z\in \R^N\setminus B_{R}(0)$.
\par\;

Now if $y\in\partial\Omega$ and $z\in B_{R}(0)\bigcap \Omega^c$, we write
\begin{align}\label{dertbt}
&\;\quad\Big|\partial_rb_r^x(y)-\partial_rb_r^x(z)\Big|\nonumber\\
&= \Big|\frac{\big(\partial_rx_r-\partial_ry_r\big)\cdot\big(x_r-y_r\big)}{|x_r-y_r|^{N-2s+2}}-\frac{\big(\partial_rx_r-\partial_rz_r\big)\cdot\big(x_r-z_r\big)}{|x_r-z_r|^{N-2s+2}}\Big|\nonumber\\
&=\Big|\big|x_r-y_r\big|^{2s-2-N}-\big|x_r-z_r\big|^{2s-2-N}\Big|\big|\partial_rx_r-\partial_ry_r\big|\big|x_r-y_r\big|\nonumber\\
&\;+\dfrac{\Big|\big(\partial_rx_r-\partial_ry_r\big)\cdot\big(x_r-y_r\big)-\big(\partial_rx_r-\partial_rz_r\big)\cdot\big(x_r-z_r\big)\Big|}{|x_r-z_r|^{N-2s+2}}\nonumber\\
&\quad =:i_r^x(y,z)+j_r^x(y,z)\nonumber
\end{align}

Let $\beta:=2s-2-N<0$.
Since $x$ in $\Omega$ is fixed, we have $|x-y|,|x-z|\geq \delta_\Omega(x)$ for all $y$ in $\partial\Omega$ and $z$ in $\R^N\setminus\Omega$.
By \eqref{ellip-reg}, the same lower bound holds for the deformed points $x_r,y_r,z_r$ uniformly for $r$ in $[-\frac12,\frac12]$.
Therefore $r\mapsto r^{\beta}$ is Lipschitz on $[c\,\delta_\Omega(x),\infty)$ for some $c>0$ and, by the mean value theorem,
\begin{align*}
\Big||x_r-y_r|^{\beta}-|x_r-z_r|^{\beta}\Big|
&\leq C\Big||x_r-y_r|-|x_r-z_r|\Big|\\
&\leq C|y_r-z_r|\leq C|y-z|.
\end{align*}
Consequently,
\begin{align*}
i_r^x(y,z)\leq C|y-z|.
\end{align*}
Since $y$ in $\partial\Omega$ and $z$ in $B_R(0)$, we have $|y-z|\leq 2R$ and thus $|y-z|\leq (2R)^{1-\alpha}|y-z|^\alpha$.
It follows that
\begin{align*}
i_r^x(y,z)\leq C|y-z|^\alpha,
\end{align*}
for some $C>0$ that is independent of $r$.
\par\;
To control $j_r^x(y,z)$, we define
$$
m_r^x(y):=\big(\partial_rx_r-\partial_ry_r\big)\cdot\big(x_r-y_r\big)
$$
and apply the fundamental theorem of calculus to get
\begin{align*}
 j_r^x(y,z)&=  \dfrac{\Big|m_r^x(y)-m_r^x(z)\Big|}{|x_r-z_r|^{N-2s+2}}\\
 &\leq\dfrac{\int_{0}^1|\nabla m_r^x(\tau y+(1-\tau)z)||y-z|d\tau }{|x_r-z_r|^{N-2s+2}}\\
 &\leq C\frac{\sup_{r\in [-1/2,+1/2]}\|\nabla m_r^x\|_{L^\infty(\R^N)}}{|x_r-z_r|^{N-2s+2}}|y-z|\\
 &\leq C|y-z|\leq C|y-z|^\alpha,
\end{align*}
for some constant $C=C(x,y,s,N, \alpha, \Omega)>0$ where used that $m_r^x\in C^{1,1}(\R^N)$ and that $|x_r-z_r|\geq C|x-z|>\delta_\Omega(x)/2$. 
In conclusion we have proved that $\partial_r b_r^x:\R^N\setminus\Omega\to \R$ satisfies \eqref{unif-est-extdata} with a constant that is independent of $r$. Hence $g^t_x$ also satisfies \eqref{unif-est-extdata} as a consequence of the representation \eqref{expansion-bt}. The proof is finished.
\end{proof} 
\section{A useful regularity estimate}
In this section, we establish a regularity result for some nonlocal divergence-type equation with non-zero Dirichlet data. The proof relies on a recent result in \cite{KW}.
For the statement of the result, it will be convenient to make the following definition.
\begin{Definition} \label{defOp}
We denote by $\mathcal{G}_s(\lambda, \Lambda)$ the set of operators pointwise defined by 
$$
\mathcal{L}w(y):=p.v\int_{\R^N}(w(y)-w(y+z))\kappa(y,y+z)dz,
$$
with kernels $\kappa(\cdot,\cdot):\R^N\times\R^N\to [0,\infty]$ that are symmetric and uniformly elliptic; that is to say,
\begin{equation}
    \lambda |y-z|^{-2s-N}\leq \kappa(y,z)=\kappa(z,y)\leq \Lambda |y-z|^{-2s-N}\qquad\forall\, y,z\in \R^N,
\end{equation}
for some $\lambda, \Lambda>0$ and which, furthermore, satisfies:
\begin{equation}\label{Sauveur1b}
    |\kappa(y+h,z+h)-\kappa(y,z)|\leq C\frac{|h|^\sigma}{|y-z|^{N+2s}}\qquad\text{for all $y,z\in \R^N$ and all $h\in B_1(0)$, for some $\sigma\in (0,s)$.}
\end{equation}
\end{Definition}
Next, we recall the following result which is essentially contained in \cite{KW}.
\begin{thm}\label{KimWeidner}
Let $\Omega\subset \R^N$ be a bounded $C^{1,1}$ domain.
Let $f$ be a given function such that 
\begin{align}\label{normXOm}
    \|f\|_{X(\Omega)}:=\sup_{w\in H^s_0(\Omega)}\frac{\int_{\Omega}|fw|dx}{[w]_{H^s(\R^N)}}<+\infty.
\end{align}
Let $\mathcal{L}$ be given as in Definition \ref{defOp} and  $u$ be a weak solution (see \cite[Definition 2.1]{KW}) to 
\begin{equation*}
\left\{ \begin{array}{rcll} \mathcal{L} u&=& f  &\textrm{in }\Omega, \\ u&=&0&%\textrm{for } x 
\textrm{in }\R^N\setminus\Omega. \end{array}\right. 
\end{equation*}
Let $p$ in $(0,s]$ and $q\in (\frac{N}{s+p},\frac{N}{p})$. Then there holds $u\in C^{s+q-\frac{N}{p}}(\overline{\Omega})$. Moreover, 
\begin{equation}
    \|u\|_{C^{s+q-\frac{N}{p}}(\overline{\Omega})}\leq C\Big(\|u\|_{L^1_{2s}(\R^N)}+\|f\|_{X(\Omega)}\Big),
\end{equation}
for some $C$ that depends only on $N,s,\sigma,\lambda, \Lambda, p, q$ and $\Omega$. In the above, the norm $\|\cdot\|_{L^1_{2s}(\R^N)}$ is defined by 
$$
\|u\|_{L^1_{2s}(\R^N)}:=\int_{\R^N}\frac{|u(y)|}{1+|y|^{N+2s}}dy.
$$
\end{thm}
\begin{proof}
It follows from \cite[Theorem 1.7]{KW}, \cite[Remark 3.2]{KW}, the interior regularity estimate \cite[Proposition 3.5]{KW} and a standard covering argument (see e.g. \cite[Proof of Corollary 2.6.11]{FRO}).
Since $\Omega$ is $C^{1,1}$, after flattening $\partial\Omega$ and restricting to sufficiently small balls, the Lipschitz constant required in \cite[Theorem 1.7]{KW} can be made as small as needed.
\end{proof}
\par\;
Thanks to the above result, we can now  prove the following statement.
\begin{Proposition}\label{Dir-non-boundary-data}
Let $\Omega\subset \R^N$ be a bounded $C^{1,1}$ domain.
Let $s$ in $(0,1)$ and let $\mathcal{L}, \kappa, \lambda, \Lambda$ be as in Definition \ref{defOp}.  Assume moreover that 
\begin{equation}\label{Sauveur2}
|\kappa^o(y,z)|\leq C\min(|z|^{-2s-N},|z|^{1-2s-N})\qquad\text{for all $y,z\in \R^N$},
\end{equation}
where $\kappa^o(\cdot,\cdot)$ is the odd part of the kernel $\kappa(\cdot,\cdot)$, that is, 
$$
\kappa^o(y,z)=\frac{\kappa(y,y+z)-\kappa(y,y-z)}{2}.
$$
Let $f$ satisfy \eqref{normXOm} and assume that $g$ satisfies
\begin{equation}\label{unif-est-extdata-cond}
|g(y)|+\frac{|g(y)-g(z)|}{|y-z|^\alpha}   \leq C_0\qquad\textup{for some $\alpha$ in $(s,\min\{1,2s\})$ and for all $y\in\partial\Omega, \,z\in \R^N\setminus\Omega$}. 
\end{equation}
Let $u$ be a weak solution to 
\begin{equation}\label{divergence-type-equation}
\left\{ \begin{array}{rcll} \mathcal{L} u&=& f  &\textrm{in }\Omega, \\ u&=&g&%\textrm{for } x 
\textrm{in }\R^N\setminus\Omega. \end{array}\right. 
\end{equation}
\par\;

Let $p$ in $(0,s]$ and $q\in (\frac{N}{s+p},\frac{N}{p})$. Then there holds $u\in C^{s+q-\frac{N}{p}}(\overline{\Omega})$. Moreover, 
\begin{equation}
\|u\|_{C^{s+q-\frac{N}{p}}(\overline{\Omega})}\leq CC_0\Big(1+\|u\|_{L^1_{2s}(\R^N)}+\|f\|_{X(\Omega)}\Big),
\end{equation}
for some $C$ that depends only on $N,s,\sigma,\lambda, \Lambda, p, q$ and $\Omega$.
\end{Proposition}
\begin{proof}
The idea is to simply transform the equation \eqref{divergence-type-equation} to a Dirichlet problem with zero boundary data and then apply Theorem \ref{KimWeidner}. For that aim, we let $\bar{g}\in C^\infty(\Omega)\cap C^\alpha(\overline{\Omega})$ be an extension of $g$ to the whole of $\R^N$. One can take for instance $\bar{g}=w$ in $\Omega$, where $w$ is the solution to the elliptic equation 
  $$
  -\Delta w=0\quad\text{in $\Omega$}\qquad\text{and}\qquad w=g\quad\text{on $\partial\Omega$}.
  $$
  We then set $\bar{g}=w$ in $\Omega$ and $\bar{g}=g$ in $\R^N\setminus\Omega$.

  Note that $g\in C^\alpha(\partial\Omega)$  by the hypothesis \eqref{unif-est-extdata-cond} and therefore by the standard regularity theory, we have that $w\in C^\infty(\Omega)\cap C^\alpha(\overline{\Omega})$. Moreover, we have
  $$
   |D^2w|\leq C\delta_\Omega^{\alpha-2}\quad\text{in}\quad \Omega. 
  $$ 
\par\;

Next let $v=u-\bar{g}$. Then $v$ solves weakly the equation 
\begin{equation}\label{eq:v-problem}
\left\{ \begin{array}{rcll} \mathcal{L} v&=& f-\mathcal{L}(\bar{g})  &\textrm{in }\Omega, \\ v&=&0&%\textrm{for } x 
\textrm{in }\R^N\setminus\Omega. \end{array}\right. 
\end{equation}
\par\;

Next we claim that
\begin{equation}\label{lbarg-est}
    |\mathcal{L}\bar{g}|\leq CC_0\delta_\Omega^{\alpha-2s} \quad\text{in}\quad  \Omega.
\end{equation} 

Indeed, let $\kappa^e(\cdot,\cdot)$ and $\kappa^o(\cdot,\cdot)$ denote the even and odd parts of $\kappa(\cdot,\cdot)$, respectively; that is,

$$
\kappa^e(y,z)=\frac{\kappa(y,y+z)+\kappa(y,y-z)}{2}\qquad\text{and} \qquad \kappa^o(y,z)=\frac{\kappa(y,y+z)-\kappa(y,y-z)}{2}.
$$
Then we may write
\begin{align*}
    (\mathcal{L}\bar{g})(y)&=p.v\int_{\R^N}(\bar{g}(y)-\bar{g}(y+z))\kappa^e(y,z)dz+p.v\int_{\R^N}(\bar{g}(y)-\bar{g}(y+z))\kappa^o(y,z)dz\\
    &:=(\mathcal{L}^e\bar{g})(y)+(\mathcal{L}^o\bar{g})(y).
\end{align*}
Because $\kappa^e(\cdot,\cdot)$ satisfies the hypothesis \cite[Equation (1.3)]{AROS}, we have by \cite[Lemma 2.6, (b)]{AROS} that 
\begin{equation}\label{Levenbarg}
|(\mathcal{L}^e\bar{g})(y)|\leq CC_0\delta_\Omega^{\alpha-2s}(y).
\end{equation}

In order to estimate $(\mathcal{L}^o\bar{g})(y)$, we let $\rho=\delta_\Omega(y)$ and write 
\begin{align}\label{Loddbarg-dec}
(\mathcal{L}^o\bar{g})(y)&=p.v\int_{B_{\rho/2}(0)}(\bar{g}(y)-\bar{g}(y+z))\kappa^o(y,z)dz+p.v\int_{\R^N\setminus B_{\rho/2}(0)}(\bar{g}(y)-\bar{g}(y+z))\kappa^o(y,z)dz\nonumber\\
&=(1)(y)+(2)(y).
\end{align}
Because $\bar{g}\in C^\alpha(\overline{\Omega})$ for any $\alpha\in (0,1)$, it follows in view of \eqref{Sauveur2}, that
\begin{align}\label{Loddbarg1}
|(1)(y)|\leq \int_{B_{\rho/2}(0)}|z|^\alpha|z|^{1-2s-N}dz\leq C\rho^{1+\alpha-2s}.
\end{align}
Next, let $y_0\in\partial\Omega$ such that $|y-y_0|=\rho$ and write 
\begin{align}\label{Loddbarg2}
|(2)(y)|
&\leq \Big|\int_{\R^N\setminus B_{\rho/2}(0)}(\bar{g}(y)-\bar{g}(y+z))\kappa^o(y,z)\,dz\Big|\nonumber\\
&\leq \int_{\R^N\setminus B_{\rho/2}(0)}\Big(|\bar{g}(y)-\bar{g}(y_0)|+|\bar{g}(y_0)-\bar{g}(y+z)|\Big)|\kappa^o(y,z)|\,dz\nonumber\\
&\leq C_0\int_{\R^N\setminus B_{\rho/2}(0)}\frac{\rho^{\alpha}+|y-y_0+z|^\alpha}{|z|^{N+2s}}\,dz\nonumber\\
&\leq C_0C_1\int_{\R^N\setminus B_{\rho/2}(0)}\frac{\rho^{\alpha}+|z|^\alpha}{|z|^{N+2s}}\,dz\nonumber\\
&\leq C_0C_2\rho^{\alpha-2s}.
\end{align}
where in the third line we used that $\bar{g}\in C^\alpha(\overline{\Omega})$ and \eqref{unif-est-extdata-cond}.
Combining \eqref{Loddbarg-dec}, \eqref{Loddbarg1} and \eqref{Loddbarg2} gives
\begin{equation}\label{Loddbarg}
|(\mathcal{L}^o\bar{g})(y)|\leq CC_0\delta_\Omega^{\alpha-2s}(y).
\end{equation}
The equation \eqref{Levenbarg} together with \eqref{Loddbarg} gives \eqref{lbarg-est} which proves the claim. Next, it is clear that 
\begin{align*}
\|\mathcal{L}\bar{g}\|_{X(\Omega)}&=\sup_{w\in H^s_0(\Omega)}\frac{\int_{\Omega}|\mathcal{L}(\bar{g})w|dx}{[w]_{H^s(\R^N)}}\nonumber\\
&\leq C_1C_0\sup_{w\in H^s_0(\Omega)}\frac{\int_{\Omega}|\delta_\Omega^{\alpha-2s}w|dx}{[w]_{H^s(\R^N)}}\nonumber\\
&\leq C_2C_0\sup_{w\in H^s_0(\Omega)}\frac{\int_{\Omega}|\delta_\Omega^{-s}w|dx}{[w]_{H^s(\R^N)}}<C_3C_0,
\end{align*}
where we used that $\alpha\in (s,2s)$.
It therefore follows, in view of Theorem \ref{KimWeidner}, that
\begin{align*}
\|v\|_{C^{s+q-\frac{N}{p}}(\overline{\Omega})}&\leq C\Big(\|u\|_{L^1_{2s}(\R^N)}+\|\bar{g}\|_{L^1_{2s}(\R^N)}+\|f\|_{X(\Omega)}+\|\mathcal{L}(\bar{g})\|_{X(\Omega)}\Big),\\
&\leq CC_0\Big(1+\|u\|_{L^1_{2s}(\R^N)}+\|f\|_{X(\Omega)}\Big).
\end{align*}
Consequently, since $\alpha>s>s-(\frac{N}{p}-q)$ and $\bar{g}\in C^\alpha(\overline{\Omega})$, we conclude that 
\begin{align*}
\|u\|_{C^{s+q-\frac{N}{p}}(\overline{\Omega})}    &\leq \|v\|_{C^{s+q-\frac{N}{p}}(\overline{\Omega})}+\|\bar{g}\|_{C^{s+q-\frac{N}{p}}(\overline{\Omega})}\\
&\leq C_1\Big(\|v\|_{C^{s+q-\frac{N}{p}}(\overline{\Omega})}+\|\bar{g}\|_{C^{\alpha}(\overline{\Omega})}\Big)\\
&\leq C_2C_0\Big(1+\|u\|_{L^1_{2s}(\R^N)}+\|f\|_{X(\Omega)}\Big).
\end{align*}
This is what we want to prove. The proof is therefore finished.
\end{proof}
\section{Some convergence results}\label{Sec5}
\begin{Lemma}\label{firstconvergenceresult}
Let $\rho\in C_c^\infty(-2,2)$ and let $\upsilon\in\R^N\setminus\{0\}$. Set
\begin{align}\label{h-est-2}
a[\rho](\mu,\upsilon):=\iint_{\R^{2N}}\frac{\big|\rho(|y|^2)-\rho(|z|^2)\big|\,\big||\upsilon+\mu y|^{2s-N}-|\upsilon+\mu z|^{2s-N}\big|}{|z|^{N-2s}|y-z|^{N+2s}}\,dy\,dz.
\end{align}
Then
\begin{align}\label{limAmu-v}
\lim_{\mu\to 0^+}a[\rho](\mu,\upsilon)=0.
\end{align}
\end{Lemma}
\begin{proof}
We recall that $\rho(|\cdot|^2)$ vanishes identically outside the ball $B_2$. For any $\mu>0$, we decompose
\begin{align*}
a[\rho](\mu,\upsilon)=(1)_\mu+(2)_\mu+(3)_\mu,
\end{align*}
where
\begin{align*}
(1)_\mu&:=\iint_{B_4\times B_4}\frac{\big|\rho(|y|^2)-\rho(|z|^2)\big|\,\big||\upsilon+\mu y|^{2s-N}-|\upsilon+\mu z|^{2s-N}\big|}{|z|^{N-2s}|y-z|^{N+2s}}\,dy\,dz,\\
(2)_\mu&:=\int_{B_2}\rho(|y|^2)\int_{\R^N\setminus B_4}\frac{\big||\upsilon+\mu y|^{2s-N}-|\upsilon+\mu z|^{2s-N}\big|}{|z|^{N-2s}|y-z|^{N+2s}}\,dz\,dy,\\
(3)_\mu&:=\int_{B_2}\frac{\rho(|z|^2)}{|z|^{N-2s}}\int_{\R^N\setminus B_4}\frac{\big||\upsilon+\mu y|^{2s-N}-|\upsilon+\mu z|^{2s-N}\big|}{|y-z|^{N+2s}}\,dy\,dz.
\end{align*}
We first observe that $(1)_\mu\to 0$ as $\mu\to 0^+$. Indeed, for $\mu<|\upsilon|/8$ and all $y,z$ in $B_4$ we have
$|\upsilon+\mu y|\geq |\upsilon|/2$ and $|\upsilon+\mu z|\geq |\upsilon|/2$, hence the map $w\mapsto |w|^{2s-N}$ is Lipschitz on $B_{|\upsilon|/2}(\upsilon)$, with a Lipschitz constant bounded independently of $\mu$. Using also that $\rho(|\cdot|^2)$ is Lipschitz, we obtain
\begin{align*}
(1)_\mu
&\leq C\,\mu\iint_{B_4\times B_4}\frac{\big|\rho(|y|^2)-\rho(|z|^2)\big|\,|y-z|}{|z|^{N-2s}|y-z|^{N+2s}}\,dy\,dz\\
&\leq C\,\mu\iint_{B_4\times B_4}\frac{dy\,dz}{|z|^{N-2s}|y-z|^{N+2s-2}}\to 0
\qquad\hbox{as }\mu\to 0^+.
\end{align*}
Next we treat $(2)_\mu$. We fix $\alpha$ in $(2s/(N+2),2s/(N+1))$ and we set
\begin{align}\label{gns}
\gamma_{N,s}(\alpha):=N-2s+\alpha.
\end{align}
We use the elementary estimate
\begin{equation}\label{elementary-est}
|a^\beta-b^\beta|\leq C\frac{|a-b|^\alpha}{a^{\gamma_{N,s}(\alpha)}+b^{\gamma_{N,s}(\alpha)}},
\qquad \beta=2s-N,
\end{equation}
valid for all $a,b>0$.
For $\mu<|\upsilon|/8$, all $y$ in $B_2$ and all $z$ in $\R^N\setminus B_4$, we have $|\upsilon+\mu y|\geq |\upsilon|/2$ and hence
\begin{align*}
\big||\upsilon+\mu y|^{2s-N}-|\upsilon+\mu z|^{2s-N}\big|
&\leq C\,\mu^\alpha |y-z|^\alpha\Big(\frac{1}{|\upsilon+\mu y|^{\gamma_{N,s}(\alpha)}}+\frac{1}{|\upsilon+\mu z|^{\gamma_{N,s}(\alpha)}}\Big)\\
&\leq C\,\mu^\alpha\big(1+|z|^\alpha\big)\Big(1+\frac{1}{|\upsilon+\mu z|^{\gamma_{N,s}(\alpha)}}\Big).
\end{align*}
Since $|z|>4$ and $y$ in $B_2$, we also have $|y-z|\geq |z|/2$, hence
\begin{align*}
(2)_\mu
&\leq C\,\mu^\alpha\int_{\R^N\setminus B_4}\frac{1+|z|^\alpha}{(1+|z|^2)^{\frac{N+2s}{2}}}
\Big(1+\frac{1}{|\upsilon+\mu z|^{\gamma_{N,s}(\alpha)}}\Big)\,dz\\
&=:i(\mu,\upsilon)+j(\mu,\upsilon).
\end{align*}
Since $\frac{1+|z|^\alpha}{(1+|z|^2)^{\frac{N+2s}{2}}}$ is integrable and $i(\mu,\upsilon)=O(\mu^\alpha)$, we have $i(\mu,\upsilon)\to 0$.
We now estimate $j(\mu,\upsilon)$.

Assume $\mu<|\upsilon|/8$ and introduce the sets
\begin{align*}
A_\mu&:=\Big\{z\in\R^N\setminus B_4:\, \big|z+\upsilon/\mu\big|\geq \frac{|\upsilon|}{2\mu}\Big\},\\
B_\mu&:=\Big(\R^N\setminus B_4\Big)\setminus A_\mu=B_{\frac{|\upsilon|}{2\mu}}\Big(-\frac{\upsilon}{\mu}\Big).
\end{align*}
On $A_\mu$ we have $|\upsilon+\mu z|=\mu|z+\upsilon/\mu|\geq |\upsilon|/2$, hence
\begin{align*}
\mu^\alpha\int_{A_\mu}\frac{1+|z|^\alpha}{(1+|z|^2)^{\frac{N+2s}{2}}}\frac{dz}{|\upsilon+\mu z|^{\gamma_{N,s}(\alpha)}}
\leq C(\upsilon)\,\mu^\alpha\int_{\R^N}\frac{1+|z|^\alpha}{(1+|z|^2)^{\frac{N+2s}{2}}}\,dz\to 0.
\end{align*}
On $B_\mu$ we make the change of variables $w=z+\upsilon/\mu$. Then $|\upsilon+\mu z|=\mu|w|$ and, since $|w|\leq |\upsilon|/(2\mu)$, we have
$|z|=\big|w-\upsilon/\mu\big|\geq |\upsilon|/(2\mu)$ and $|z|\leq 3|\upsilon|/(2\mu)$.
Consequently,
\begin{align*}
\frac{1+|z|^\alpha}{(1+|z|^2)^{\frac{N+2s}{2}}}
\leq C(\upsilon)\,|z|^{-N-2s+\alpha}\leq C(\upsilon)\,\mu^{N+2s-\alpha}
\qquad \hbox{for all }z\in B_\mu.
\end{align*}
Using also that $\gamma_{N,s}(\alpha)<N$ because $\alpha<2s/(N+1)$, we obtain
\begin{align*}
\mu^\alpha\int_{B_\mu}\frac{1+|z|^\alpha}{(1+|z|^2)^{\frac{N+2s}{2}}}\frac{dz}{|\upsilon+\mu z|^{\gamma_{N,s}(\alpha)}}
&\leq C(\upsilon)\,\mu^{N+2s-\gamma_{N,s}(\alpha)}\int_{B_{\frac{|\upsilon|}{2\mu}}(0)}\frac{dw}{|w|^{\gamma_{N,s}(\alpha)}}\\
&\leq C(\upsilon)\,\mu^{N+2s-\gamma_{N,s}(\alpha)}\Big(\frac{|\upsilon|}{2\mu}\Big)^{N-\gamma_{N,s}(\alpha)}
\leq C(\upsilon)\,\mu^{2s}\to 0.
\end{align*}
This proves $j(\mu,\upsilon)\to 0$ and hence $(2)_\mu\to 0$.

The term $(3)_\mu$ is treated in the same way by exchanging the roles of $y$ and $z$. Altogether, we conclude that $a[\rho](\mu,\upsilon)\to 0$ as $\mu\to 0^+$.
\end{proof}
\begin{Lemma}\label{LimBC-mu-k-gamma}
Let $x,y$ in $\Omega$ with $x\neq y$ and recall the definition of $\omega^x_{Y}(\cdot,\cdot)$ in \eqref{defWsam}. Next let $\rho\in C_c^\infty(-2,2)$, and set
\begin{align}\label{def-a-y-rho}
a_Y[\rho](x):=b_{N,s}\iint_{\R^{2N}}\frac{\big(\rho(|q|^2)-\rho(|p|^2)\big)\big(|p|^{2s-N}-|q|^{2s-N}\big)}{|p-q|^{N+2s}}\omega^x_{Y}(p,q)\,dp\,dq.
\end{align}
For $\gamma>0$ we set
\begin{equation}\label{defgy-gamma}
g_{\gamma}(x,z):=g_0(x,z)\phi_x^\gamma(z).
\end{equation}
where $\phi_x^\gamma:=1-\rho_\gamma^x$ and $\rho_\gamma^x$ is defined in \eqref{psi-mu}.
Let $k>0$ and define
\begin{align}\label{def-G-la}
B_Y[\phi_x^{\mu},g_{0}(x,\cdot)](k,\gamma,y):=\iint_{\R^{2N}}\xi_k(q)g_{\gamma}(y,q)\big(\phi_x^{\mu}(p)-\phi_x^{\mu}(q)\big)\big(g_0(x,p)-g_0(x,q)\big)\kappa_Y(p,q)\,dp\,dq,
\end{align}
and
\begin{align}\label{def-H-la}
C_Y[\phi_x^{\mu},g_{\gamma}(y,\cdot)](k,x):=\iint_{\R^{2N}}\xi_k^2(q)g_0(x,q)\big(\phi_x^{\mu}(p)-\phi_x^{\mu}(q)\big)\big(g_{\gamma}(y,p)-g_{\gamma}(y,q)\big)\kappa_Y(p,q)\,dp\,dq.
\end{align}
Then we have
\begin{align}\label{LimCxy-mu-k-gamma}
\lim_{\mu\to 0^+}\lim_{k\to\infty} C_Y[\phi_x^{\mu},g_{\gamma}(y,\cdot)](k,x)=0,
\end{align}
and
\begin{equation}\label{lem-lim-g-k}
\lim_{\gamma\to 0^+}\lim_{\mu\to 0^+}\lim_{k\to\infty}B_Y[\phi_x^{\mu},g_{0}(x,\cdot)](k,\gamma,y)=g_0(x,y)a_Y[\rho](x).
\end{equation}
\end{Lemma}
\begin{remark}
We note that the integrals in \eqref{def-G-la} and \eqref{def-H-la} are well defined. This can be proved in the same way as in Lemma \ref{FundLemma}.
\end{remark}
\begin{proof}
We first prove \eqref{lem-lim-g-k}.
By dominated convergence, we have
\begin{align*}
\lim_{k\to\infty}B_Y[\phi_x^{\mu},g_{0}(x,\cdot)](k,\gamma,y)
=B_Y[\phi_x^{\mu},g_{0}(x,\cdot)](\gamma,y),
\end{align*}
where
\begin{align*}
B_Y[\phi_x^{\mu},g_{0}(x,\cdot)](\gamma,y)
:=\iint_{\R^{2N}}g_{\gamma}(y,q)\big(\phi_x^{\mu}(p)-\phi_x^{\mu}(q)\big)\big(g_0(x,p)-g_0(x,q)\big)\kappa_Y(p,q)\,dp\,dq.
\end{align*}
Following \cite[Lemma 2.2]{Sidy-Franck-BP}, to compute $\lim_{\mu\to 0^+}B_Y[\phi_x^{\mu},g_{0}(x,\cdot)](\gamma,y)$, it suffices to restrict the integral to balls around $x$.
Let $\varepsilon>0$ be such that $B_{2\varepsilon}(x)\subset\Omega$ and define
\begin{align}\label{lim-B-mu}
B_Y[\phi_x^{\mu},g_{0}(x,\cdot)](\varepsilon,\gamma,y)
:=\iint_{B_{\varepsilon}(x)\times B_{2\varepsilon}(x)}g_{\gamma}(y,q)\big(\phi_x^{\mu}(p)-\phi_x^{\mu}(q)\big)\big(g_0(x,p)-g_0(x,q)\big)\kappa_Y(p,q)\,dp\,dq.
\end{align}
Then
\begin{equation}\label{proof58}
\lim_{\mu\to 0^+}B_Y[\phi_x^{\mu},g_{0}(x,\cdot)](\gamma,y)=\lim_{\mu\to 0^+}B_Y[\phi_x^{\mu},g_{0}(x,\cdot)](\varepsilon,\gamma,y).
\end{equation}
We set $\widetilde{\mu}:=\delta_\Omega(x)\mu/(2\sqrt{2})$ and we change variables $p=x-\widetilde{\mu}p$ and $q=x-\widetilde{\mu}q$ in \eqref{lim-B-mu}, then we rename the new variables $(p,q)$.
This yields
\begin{align}\label{1st-deduc}
B_Y[\phi_x^{\mu},g_{0}(x,\cdot)](\varepsilon,\gamma,y)
=\iint_{\R^{2N}} g_\gamma(y,x-\widetilde{\mu}q)F_x^y(\varepsilon,\widetilde{\mu};\,p,q)\,dp\,dq,
\end{align}
where
\begin{align*}
F_x^y(\varepsilon,\widetilde{\mu};\,p,q)
:=
1_{B_{\varepsilon/\widetilde{\mu}}(0)}(p)\,1_{B_{2\varepsilon/\widetilde{\mu}}(0)}(q)\,
\widetilde{\mu}^{N-2s}\big(\rho(|q|^2)-\rho(|p|^2)\big)
\big(g_0(x,x-\widetilde{\mu}p)-g_0(x,x-\widetilde{\mu}q)\big)\kappa_Y(\widetilde{\mu};p,q)
\end{align*}
and
\begin{align}\label{def-kappa-Y-mu}
\kappa_Y(\widetilde{\mu};\,p,q)
:=|p-q|^{-2s-N}\omega_Y(x-\widetilde{\mu}p,x-\widetilde{\mu}q).
\end{align}
We now show that the integrand in \eqref{1st-deduc} is dominated by an integrable function independent of $\mu$.
We note that $g_\gamma(y,\cdot)$ is bounded on $B_{2\varepsilon}(x)$ because $x\neq y$. Hence there exists $C>0$, independent of $\mu$, such that
\begin{equation}\label{g-y-q-bound}
|g_\gamma(y,x-\widetilde{\mu}q)|\leq C
\qquad \hbox{for all }q\in B_{2\varepsilon/\widetilde{\mu}}(0).
\end{equation}
Recall that
\begin{equation}\label{def-w-Y}
\omega_Y(y,z)=\frac{c_{N,s}}{2}\Big(\div Y(y)+\div Y(z)-(N+2s)\frac{\big(Y(y)-Y(z)\big)\cdot(y-z)}{|y-z|^2}\Big).
\end{equation}
Since $(x-\widetilde{\mu}p)-(x-\widetilde{\mu}q)=\widetilde{\mu}(q-p)$, we have
\begin{align*}
&\omega_Y(x-\widetilde{\mu}p,x-\widetilde{\mu}q)\\
&=\frac{c_{N,s}}{2}\Big(\div Y(x-\widetilde{\mu}p)+\div Y(x-\widetilde{\mu}q)
+(N+2s)\frac{\big(Y(x-\widetilde{\mu}p)-Y(x-\widetilde{\mu}q)\big)\cdot(p-q)}{\widetilde{\mu}|p-q|^2}\Big).
\end{align*}
In particular, $\omega_Y(x-\widetilde{\mu}p,x-\widetilde{\mu}q)$ is bounded uniformly in $\mu$, hence
\begin{equation}\label{kappa-Y-bound}
\big|\kappa_Y(\widetilde{\mu};p,q)\big|\leq C|p-q|^{-2s-N}.
\end{equation}
Using also that $g_0(x,x-\widetilde{\mu}\cdot)$ is Hölder continuous away from the origin, we obtain
\begin{equation}\label{01reg-es}
\widetilde{\mu}^{N-2s}\big|g_0(x,x-\widetilde{\mu}p)-g_0(x,x-\widetilde{\mu}q)\big|\leq C\big||p|^{2s-N}-|q|^{2s-N}\big|+C\widetilde{\mu}^{N-2s+1}|p-q|.
\end{equation}
To justify \eqref{01reg-es}, we write
\begin{align*}
&\widetilde{\mu}^{N-2s}\big|g_0(x,x-\widetilde{\mu}p)-g_0(x,x-\widetilde{\mu}q)\big|\\
&=\widetilde{\mu}^{N-2s}\Big|b_{N,s}\big(|\widetilde{\mu}p|^{2s-N}-|\widetilde{\mu}q|^{2s-N}\big)-\big(H_\Omega^s(x,x-\widetilde{\mu}p)-H_\Omega^s(x,x-\widetilde{\mu}q)\big)\Big|\\
&\leq b_{N,s}\big||p|^{2s-N}-|q|^{2s-N}\big|+\widetilde{\mu}^{N-2s}\big|H_\Omega^s(x,x-\widetilde{\mu}p)-H_\Omega^s(x,x-\widetilde{\mu}q)\big|.
\end{align*}
By \cite[Lemma A.1]{KN}, for all $x,p$ in $\Omega$ we have $|\nabla_pH_\Omega^s(x,p)|\leq C\,\delta_\Omega(x)^{2s-N-1}$. Hence, by the fundamental theorem of calculus,
\begin{align*}
\widetilde{\mu}^{N-2s}\big|H_\Omega^s(x,x-\widetilde{\mu}p)-H_\Omega^s(x,x-\widetilde{\mu}q)\big|
&\leq C\widetilde{\mu}^{N-2s+1}|p-q|
\leq C|p-q|.
\end{align*}
Combining \eqref{g-y-q-bound}, \eqref{kappa-Y-bound} and \eqref{01reg-es} we deduce
\begin{align}\label{domination-f}
\big|g_\gamma(y,x-\widetilde{\mu}q)F_x^y(\varepsilon,\widetilde{\mu};p,q)\big|
\leq \frac{\big|\rho(|p|^2)-\rho(|q|^2)\big|}{|p-q|^{N+2s}}
\Big(\big||p|^{2s-N}-|q|^{2s-N}\big|+|p-q|\Big).
\end{align}
The right-hand side is integrable on $\R^{2N}$, hence dominated convergence applies in \eqref{1st-deduc}.
We also note that
\begin{align}\label{lim-B-mu-integrand}
\lim_{\mu\to 0^+}g_\gamma(y,x-\widetilde{\mu}q)F_x^y(\varepsilon,\widetilde{\mu};p,q)
=g_\gamma(y,x)b_{N,s}\frac{\big(\rho(|q|^2)-\rho(|p|^2)\big)\big(|p|^{2s-N}-|q|^{2s-N}\big)}{|p-q|^{N+2s}}\omega^x_{Y}(p,q),
\end{align}
where the limit of $\omega_Y$ follows from the Taylor expansion of $Y$ at $x$.
Therefore, we obtain
\begin{align}\label{proof59}
\lim_{\mu\to 0^+}B_Y[\phi_x^{\mu},g_{0}(x,\cdot)](\varepsilon,\gamma,y)=g_\gamma(y,x)a_Y[\rho](x).
\end{align}
Combining \eqref{proof58} and \eqref{proof59} yields
\begin{align*}
\lim_{\mu\to 0^+}\lim_{k\to\infty}B_Y[\phi_x^{\mu},g_{0}(x,\cdot)](k,\gamma,y)=g_\gamma(y,x)a_Y[\rho](x).
\end{align*}
Letting $\gamma\to 0^+$, we have $g_\gamma(y,x)\to g_0(y,x)=g_0(x,y)$ since $x\neq y$, which proves \eqref{lem-lim-g-k}.

We now prove \eqref{LimCxy-mu-k-gamma}.
By dominated convergence, we have
\begin{align*}
\lim_{k\to\infty}C_Y[\phi_x^{\mu},g_{\gamma}(y,\cdot)](k,x)
=C_Y[\phi_x^{\mu},g_{\gamma}(y,\cdot)](x),
\end{align*}
where
\begin{align}\label{Cxy-mu}
C_Y[\phi_x^{\mu},g_{\gamma}(y,\cdot)](x)
:=\iint_{\R^{2N}}g_0(x,q)\big(\phi_x^{\mu}(p)-\phi_x^{\mu}(q)\big)\big(g_{\gamma}(y,p)-g_{\gamma}(y,q)\big)\kappa_Y(p,q)\,dp\,dq.
\end{align}
Using again \cite[Lemma 2.2]{Sidy-Franck-BP}, it suffices to restrict the integral to balls around $x$. With the same $\varepsilon>0$ as above, we define
\begin{align*}
C_Y[\phi_x^{\mu},g_{\gamma}(y,\cdot)](\varepsilon,x)
:=\iint_{B_{\varepsilon}(x)\times B_{2\varepsilon}(x)}g_0(x,q)\big(\phi_x^{\mu}(p)-\phi_x^{\mu}(q)\big)\big(g_{\gamma}(y,p)-g_{\gamma}(y,q)\big)\kappa_Y(p,q)\,dp\,dq,
\end{align*}
so that
\begin{equation}\label{proofC1}
\lim_{\mu\to 0^+}C_Y[\phi_x^{\mu},g_{\gamma}(y,\cdot)](x)=\lim_{\mu\to 0^+}C_Y[\phi_x^{\mu},g_{\gamma}(y,\cdot)](\varepsilon,x).
\end{equation}
With the same change of variables $p=x-\widetilde{\mu}p$ and $q=x-\widetilde{\mu}q$ as above (and renaming $(p,q)$), we obtain
\begin{align}\label{Cxy-mu-rescaled}
C_Y[\phi_x^{\mu},g_{\gamma}(y,\cdot)](\varepsilon,x)
=\iint_{\R^{2N}}\widetilde{\mu}^{N-2s}g_0(x,x-\widetilde{\mu}q)G_x^y(\varepsilon,\widetilde{\mu};\,p,q)\,dp\,dq,
\end{align}
where
\begin{align}\label{defGxy}
G_x^y(\varepsilon,\widetilde{\mu};\,p,q)
:=
1_{B_{\varepsilon/\widetilde{\mu}}(0)}(p)\,1_{B_{2\varepsilon/\widetilde{\mu}}(0)}(q)\,
\big(\rho(|q|^2)-\rho(|p|^2)\big)
\big(g_{\gamma}(y,x-\widetilde{\mu}p)-g_{\gamma}(y,x-\widetilde{\mu}q)\big)\kappa_Y(\widetilde{\mu};p,q).
\end{align}
Since $g_0(x,\cdot)$ vanishes identically outside $\Omega$, we have
\begin{equation}\label{bx-es}
\widetilde{\mu}^{N-2s}|g_0(x,x-\widetilde{\mu}q)|\leq C|q|^{2s-N}.
\end{equation}
Write $\upsilon:=x-y$. We estimate, for $p,q$ in $B_{2\varepsilon/\widetilde{\mu}}(0)$,
\begin{align*}
\big|g_\gamma(y,x-\widetilde{\mu}p)-g_\gamma(y,x-\widetilde{\mu}q)\big|
&=\big|g_0(x-\widetilde{\mu}p,y)\phi_y^\gamma(x-\widetilde{\mu}p)-g_0(x-\widetilde{\mu}q,y)\phi_y^\gamma(x-\widetilde{\mu}q)\big|\\
&\leq A(\widetilde{\mu};p,q)+B(\widetilde{\mu};p,q),
\end{align*}
where
\begin{align*}
A(\widetilde{\mu};p,q)&:=\big|g_0(x-\widetilde{\mu}p,y)-g_0(x-\widetilde{\mu}q,y)\big|,\\
B(\widetilde{\mu};p,q)&:=\big|\phi_y^\gamma(x-\widetilde{\mu}p)-\phi_y^\gamma(x-\widetilde{\mu}q)\big|\,|g_0(x-\widetilde{\mu}q,y)|.
\end{align*}
Since $\phi_y^\gamma=1-\rho_\gamma^y$, we have
\begin{align*}
B(\widetilde{\mu};p,q)
&=\big|\rho_\gamma^y(x-\widetilde{\mu}p)-\rho_\gamma^y(x-\widetilde{\mu}q)\big|\,|g_0(x-\widetilde{\mu}q,y)|.
\end{align*}
Using the definition of $\rho_\gamma^y$ in \eqref{psi-mu}, the mean value theorem and the fact that $|x-\widetilde{\mu}p-y|$ stays bounded away from $0$ for $p,q$ in $B_{2\varepsilon/\widetilde{\mu}}(0)$ (since $x\neq y$), we obtain
\begin{align}\label{C-uni-es-1}
B(\widetilde{\mu};p,q)\leq C\,\widetilde{\mu}|p-q|.
\end{align}
On the other hand,
\begin{align*}
A(\widetilde{\mu};p,q)
&=\Big|b_{N,s}\big(|\upsilon+\widetilde{\mu}p|^{2s-N}-|\upsilon+\widetilde{\mu}q|^{2s-N}\big)
-\big(H_\Omega^s(y,x-\widetilde{\mu}p)-H_\Omega^s(y,x-\widetilde{\mu}q)\big)\Big|\\
&\leq b_{N,s}\big||\upsilon+\widetilde{\mu}p|^{2s-N}-|\upsilon+\widetilde{\mu}q|^{2s-N}\big|
+\big|H_\Omega^s(y,x-\widetilde{\mu}p)-H_\Omega^s(y,x-\widetilde{\mu}q)\big|.
\end{align*}
Since $H_\Omega^s(y,\cdot)$ is Lipschitz on $B_{2\varepsilon}(x)$, we get
\begin{equation}\label{C-uni-es-2}
A(\widetilde{\mu};p,q)\leq C\big||\upsilon+\widetilde{\mu}p|^{2s-N}-|\upsilon+\widetilde{\mu}q|^{2s-N}\big|+C\,\widetilde{\mu}|p-q|.
\end{equation}
Combining \eqref{kappa-Y-bound}, \eqref{C-uni-es-1} and \eqref{C-uni-es-2}, we obtain
\begin{align}\label{Gxymu-es}
\big|\widetilde{\mu}^{N-2s}g_0(x,x-\widetilde{\mu}q)G_x^y(\varepsilon,\widetilde{\mu};p,q)\big|
\leq C\frac{W_x^y(\widetilde{\mu};p,q)\,\big|\rho(|p|^2)-\rho(|q|^2)\big|}{|q|^{N-2s}|p-q|^{N+2s}},
\end{align}
where
\begin{align*}
W_x^y(\widetilde{\mu};p,q)
:=\widetilde{\mu}|p-q|\big(1+|\upsilon+\widetilde{\mu}q|\big)+\big||\upsilon+\widetilde{\mu}p|^{2s-N}-|\upsilon+\widetilde{\mu}q|^{2s-N}\big|.
\end{align*}
We claim that
\begin{equation}\label{LimRmu-v}
\lim_{\widetilde{\mu}\to 0^+}R(\widetilde{\mu},\upsilon)=0.
\end{equation}
Here
\begin{align*}
R(\widetilde{\mu},\upsilon)
:=\iint_{\R^{2N}}1_{B_{\varepsilon/\widetilde{\mu}}(0)}(p)\,1_{B_{2\varepsilon/\widetilde{\mu}}(0)}(q)\,
\frac{W_x^y(\widetilde{\mu};p,q)\,\big|\rho(|p|^2)-\rho(|q|^2)\big|}{|q|^{N-2s}|p-q|^{N+2s}}\,dp\,dq.
\end{align*}
Indeed, by the definition of $W_x^y$ we can write $R(\widetilde{\mu},\upsilon)=R_1(\widetilde{\mu},\upsilon)+R_2(\widetilde{\mu},\upsilon)$, where
\begin{align*}
R_1(\widetilde{\mu},\upsilon)
&:=\iint_{\R^{2N}}1_{B_{\varepsilon/\widetilde{\mu}}(0)}(p)\,1_{B_{2\varepsilon/\widetilde{\mu}}(0)}(q)\,
\frac{\big|\rho(|p|^2)-\rho(|q|^2)\big|\,\big||\upsilon+\widetilde{\mu}p|^{2s-N}-|\upsilon+\widetilde{\mu}q|^{2s-N}\big|}{|q|^{N-2s}|p-q|^{N+2s}}\,dp\,dq,\\
R_2(\widetilde{\mu},\upsilon)
&:=\widetilde{\mu}\iint_{\R^{2N}}1_{B_{\varepsilon/\widetilde{\mu}}(0)}(p)\,1_{B_{2\varepsilon/\widetilde{\mu}}(0)}(q)\,
\frac{\big(1+|\upsilon+\widetilde{\mu}q|\big)\big|\rho(|p|^2)-\rho(|q|^2)\big|}{|q|^{N-2s}|p-q|^{N+2s-1}}\,dp\,dq.
\end{align*}
Since the indicators are bounded by $1$, we have $R_1(\widetilde{\mu},\upsilon)\leq a[\rho](\widetilde{\mu},\upsilon)$, hence $R_1(\widetilde{\mu},\upsilon)\to 0$ by Lemma \ref{firstconvergenceresult}.

To estimate $R_2$, we note that $|q|\leq 2\varepsilon/\widetilde{\mu}$ implies $1+|\upsilon+\widetilde{\mu}q|\leq 1+|\upsilon|+2\varepsilon$.
Moreover, since $\rho(|\cdot|^2)$ is Lipschitz and bounded, there exists $C>0$ such that
\begin{equation*}
\big|\rho(|p|^2)-\rho(|q|^2)\big|\leq C\min\{1,|p-q|\}
\qquad \hbox{for all }p,q\in\R^N.
\end{equation*}
Therefore,
\begin{align*}
R_2(\widetilde{\mu},\upsilon)
&\leq C\,\widetilde{\mu}\iint_{\R^{2N}}1_{B_{\varepsilon/\widetilde{\mu}}(0)}(p)\,1_{B_{2\varepsilon/\widetilde{\mu}}(0)}(q)\,
\frac{\min\{1,|p-q|\}}{|q|^{N-2s}|p-q|^{N+2s-1}}\,dp\,dq.
\end{align*}
Since $\rho(|\cdot|^2)$ is supported in $B_{\sqrt{2}}(0)$, the integrand is nonzero only if $p\in B_{\sqrt{2}}(0)$ or $q\in B_{\sqrt{2}}(0)$.
Accordingly, we write $R_2\leq C\widetilde{\mu}(I_1+I_2)$ with
\begin{align*}
I_1&:=\int_{B_{\sqrt{2}}(0)}\frac{dq}{|q|^{N-2s}}\int_{B_{\varepsilon/\widetilde{\mu}}(0)}\frac{\min\{1,|p-q|\}}{|p-q|^{N+2s-1}}\,dp,\\
I_2&:=\int_{B_{\sqrt{2}}(0)}dp\int_{B_{2\varepsilon/\widetilde{\mu}}(0)\setminus B_{\sqrt{2}}(0)}\frac{dq}{|q|^{N-2s}|p-q|^{N+2s-1}}.
\end{align*}
For $I_1$, since $q$ stays bounded and $p$ ranges in a ball of radius $\varepsilon/\widetilde{\mu}$, we have
\begin{align*}
\int_{B_{\varepsilon/\widetilde{\mu}}(0)}\frac{\min\{1,|p-q|\}}{|p-q|^{N+2s-1}}\,dp
\leq C\int_{B_{\varepsilon/\widetilde{\mu}+2}(0)}\frac{\min\{1,|r|\}}{|r|^{N+2s-1}}\,dr
\leq C\Big(1+\Big(\frac{\varepsilon}{\widetilde{\mu}}\Big)^{\max\{0,1-2s\}}\Big),
\end{align*}
(with the obvious $\log(\varepsilon/\widetilde{\mu})$ modification in the borderline case $2s=1$).
Since $\int_{B_{\sqrt{2}}(0)}|q|^{-(N-2s)}\,dq<\infty$, it follows that $\widetilde{\mu}I_1\to 0$.

For $I_2$, since $p$ is bounded and $|q|\geq \sqrt{2}$, we have $|p-q|\geq |q|/2$, hence
\begin{align*}
I_2\leq C\int_{B_{\sqrt{2}}(0)}dp\int_{B_{2\varepsilon/\widetilde{\mu}}(0)\setminus B_{\sqrt{2}}(0)}\frac{dq}{|q|^{2N-1}}
\leq C\int_{\sqrt{2}}^{2\varepsilon/\widetilde{\mu}}r^{-N}\,dr,
\end{align*}
which is bounded uniformly in $\widetilde{\mu}$ if $N\geq 2$, and grows at most like $\log(1/\widetilde{\mu})$ if $N=1$. In both cases, $\widetilde{\mu}I_2\to 0$.
Thus $R_2(\widetilde{\mu},\upsilon)\to 0$, which proves \eqref{LimRmu-v}.
Combining \eqref{Gxymu-es}, \eqref{LimRmu-v} and dominated convergence in \eqref{Cxy-mu-rescaled}, we obtain
\begin{align}\label{LimCxy}
\lim_{\mu\to 0^+}C_Y[\phi_x^{\mu},g_{\gamma}(y,\cdot)](\varepsilon,x)=0.
\end{align}
Finally, \eqref{proofC1} and \eqref{LimCxy} yield \eqref{LimCxy-mu-k-gamma}.
\end{proof}
With a similar argument as above, we also prove the following which will be used in the proof of Lemma \ref{main-res-sec4.4} as well. We keep using the notations introduced in the section above:
\begin{Lemma}\label{Lem.sec-4.3}
Fix $x,y$ in $\Omega$ with $x\neq y$.
Let $u:=\xi_k g_{\gamma}(y,\cdot)$ and $v:=\xi_k \phi_x^\mu$.
Then
\begin{equation}\label{C2}
\lim_{\gamma\to 0^+}\lim_{\mu\to 0^+}\lim_{k\to \infty}D_Y[v](\gamma,y)=g_0(x,y)a_Y[\rho](x),
\end{equation}
where
\begin{equation}\label{C3}
D_Y[v](\gamma,y)=\mathrm{p.v.}\iint_{\R^{2N}}\big(v(p)-v(q)\big)\big(u(q)g_0(x,p)-u(p)g_0(x,q)\big)\kappa_Y(p,q)\,dp\,dq.
\end{equation}
\end{Lemma}
\begin{proof}
We first expand
\begin{align}\label{align}
\big(v(p)-v(q)\big)\big(u(q)g_0(x,p)-u(p)g_0(x,q)\big)
&=u(q)\big(v(p)-v(q)\big)\big(g_0(x,p)-g_0(x,q)\big)\nonumber\\
&\quad +g_0(x,q)\big(v(p)-v(q)\big)\big(u(q)-u(p)\big),
\end{align}
and we use the decomposition
\begin{equation}\label{E}
v(p)-v(q)=\phi_x^\mu(p)\big(\xi_k(p)-\xi_k(q)\big)+\xi_k(q)\big(\phi_x^\mu(p)-\phi_x^\mu(q)\big).
\end{equation}
We also set $g_\mu(x,\cdot):=g_0(x,\cdot)\phi_x^\mu(\cdot)$.
Then
\begin{align*}
\phi_x^\mu(p)\big(g_0(x,p)-g_0(x,q)\big)
=g_\mu(x,p)-g_\mu(x,q)+g_0(x,q)\big(\phi_x^\mu(q)-\phi_x^\mu(p)\big).
\end{align*}
Inserting the last identity into \eqref{align}, we obtain
\begin{align*}
D_Y[v](\gamma,y)=E_Y[g_\mu]-F_Y[\phi_x^\mu]+G_Y[u]+H_Y[\phi_x^\mu,g_0(x,\cdot)]+I_Y[\phi_x^\mu,u],
\end{align*}
where
\begin{align}\label{definitions-EFGHI}
E_Y[g_\mu](\gamma,y)
&:=\mathrm{p.v.}\iint_{\R^{2N}}u(q)\big(\xi_k(p)-\xi_k(q)\big)\big(g_\mu(x,p)-g_\mu(x,q)\big)\kappa_Y(p,q)\,dp\,dq,\\
F_Y[\phi_x^\mu](\gamma,y)
&:=\mathrm{p.v.}\iint_{\R^{2N}}u(q)g_0(x,q)\big(\xi_k(p)-\xi_k(q)\big)\big(\phi_x^\mu(q)-\phi_x^\mu(p)\big)\kappa_Y(p,q)\,dp\,dq,\\
G_Y[u](\gamma,y)
&:=\mathrm{p.v.}\iint_{\R^{2N}}\phi_x^\mu(p)g_0(x,q)\big(\xi_k(p)-\xi_k(q)\big)\big(u(q)-u(p)\big)\kappa_Y(p,q)\,dp\,dq,\\
H_Y[\phi_x^\mu,g_0(x,\cdot)](\gamma,y)
&:=\mathrm{p.v.}\iint_{\R^{2N}}\xi_k(q)u(q)\big(\phi_x^\mu(p)-\phi_x^\mu(q)\big)\big(g_0(x,p)-g_0(x,q)\big)\kappa_Y(p,q)\,dp\,dq,\\
I_Y[\phi_x^\mu,u](\gamma,y)
&:=\mathrm{p.v.}\iint_{\R^{2N}}\xi_k(q)g_0(x,q)\big(\phi_x^\mu(p)-\phi_x^\mu(q)\big)\big(u(q)-u(p)\big)\kappa_Y(p,q)\,dp\,dq.
\end{align}
We claim that
\begin{equation}\label{lim-EFG}
\lim_{k\to\infty}E_Y[g_\mu]=\lim_{k\to\infty}F_Y[\phi_x^\mu]=\lim_{k\to\infty}G_Y[u]=0.
\end{equation}
Indeed, since $g_\gamma(y,\cdot)$ is bounded and $g_\mu(x,\cdot)$ is H\"older continuous on $\overline{\Omega}$, Lemma \ref{lem-crucial} yields
\begin{equation*}
\big|E_Y[g_\mu](\gamma,y)\big|\leq C\int_{\Omega}g_0(x,q)\delta_\Omega(q)^{-s}\,dq<\infty,
\end{equation*}
and dominated convergence gives $\lim_{k\to\infty}E_Y[g_\mu]=0$.

Since $g_0(x,\cdot)$ and $u$ vanish identically outside $\Omega$, we have
\begin{equation*}
\big|F_Y[\phi_x^\mu](\gamma,y)\big|
\leq C\int_{\Omega}g_0(x,q)\delta_\Omega(q)^{-s}\,dq<\infty,
\end{equation*}
hence dominated convergence gives $\lim_{k\to\infty}F_Y[\phi_x^\mu]=0$.

Finally, we use the splitting
\begin{equation}\label{split-u}
u(p)-u(q)=g_{\gamma}(y,p)\big(\xi_k(p)-\xi_k(q)\big)+\xi_k(q)\big(g_{\gamma}(y,p)-g_{\gamma}(y,q)\big).
\end{equation}
Using \eqref{split-u}, we bound
\begin{align*}
\big|G_Y[u](\gamma,y)\big|
&\leq C\iint_{\R^{2N}}|g_0(x,q)\phi_x^\mu(p)|\,\big|\xi_k(p)-\xi_k(q)\big|^2\,\frac{dp\,dq}{|p-q|^{N+2s}}\\
&\quad +C\iint_{\R^{2N}}|g_0(x,q)\phi_x^\mu(p)|\,\big|\xi_k(p)-\xi_k(q)\big|\,\frac{|g_\gamma(y,p)-g_\gamma(y,q)|}{|p-q|^{N+2s}}\,dp\,dq.
\end{align*}
Both integrals are finite and tend to $0$ as $k\to\infty$ by Lemma \ref{lem-crucial} and dominated convergence, hence $\lim_{k\to\infty}G_Y[u]=0$.
This proves \eqref{lim-EFG}.

We now treat the remaining terms.
Since $u=\xi_k g_\gamma(y,\cdot)$, we have $\xi_k(q)u(q)=\xi_k^2(q)g_\gamma(y,q)$ and therefore
\begin{align*}
H_Y[\phi_x^\mu,g_0(x,\cdot)](\gamma,y)
=\mathrm{p.v.}\iint_{\R^{2N}}\xi_k^2(q)g_\gamma(y,q)\big(\phi_x^\mu(p)-\phi_x^\mu(q)\big)\big(g_0(x,p)-g_0(x,q)\big)\kappa_Y(p,q)\,dp\,dq.
\end{align*}
Letting $k\to\infty$ and using dominated convergence, we obtain
\begin{align*}
\lim_{k\to\infty}H_Y[\phi_x^\mu,g_0(x,\cdot)](\gamma,y)
=\mathrm{p.v.}\iint_{\R^{2N}}g_\gamma(y,q)\big(\phi_x^\mu(p)-\phi_x^\mu(q)\big)\big(g_0(x,p)-g_0(x,q)\big)\kappa_Y(p,q)\,dp\,dq,
\end{align*}
which coincides with $\lim_{k\to\infty}B_Y[\phi_x^\mu,g_0(x,\cdot)](k,\gamma,y)$.
Therefore, by \eqref{lem-lim-g-k},
\begin{equation}\label{lim-H}
\lim_{\gamma\to 0^+}\lim_{\mu\to 0^+}\lim_{k\to\infty}H_Y[\phi_x^\mu,g_0(x,\cdot)](\gamma,y)=g_0(x,y)a_Y[\rho](x).
\end{equation}
Next we write
\begin{align*}
u(q)-u(p)=\xi_k(q)\big(g_\gamma(y,q)-g_\gamma(y,p)\big)+g_\gamma(y,p)\big(\xi_k(q)-\xi_k(p)\big),
\end{align*}
so that $I_Y[\phi_x^\mu,u]=I_{Y,1}+I_{Y,2}$ with
\begin{align*}
I_{Y,1}
&:=\mathrm{p.v.}\iint_{\R^{2N}}\xi_k^2(q)g_0(x,q)\big(\phi_x^\mu(p)-\phi_x^\mu(q)\big)\big(g_\gamma(y,q)-g_\gamma(y,p)\big)\kappa_Y(p,q)\,dp\,dq,\\
I_{Y,2}
&:=\mathrm{p.v.}\iint_{\R^{2N}}\xi_k(q)g_0(x,q)\big(\phi_x^\mu(p)-\phi_x^\mu(q)\big)g_\gamma(y,p)\big(\xi_k(q)-\xi_k(p)\big)\kappa_Y(p,q)\,dp\,dq.
\end{align*}
Since $I_{Y,1}=-C_Y[\phi_x^\mu,g_\gamma(y,\cdot)](k,x)$, Lemma \ref{LimBC-mu-k-gamma} gives
\begin{equation}\label{lim-I1}
\lim_{\mu\to 0^+}\lim_{k\to\infty}I_{Y,1}=0.
\end{equation}
The term $I_{Y,2}$ contains the factor $\xi_k(q)-\xi_k(p)$ and is handled in the same way as $F_Y[\phi_x^\mu]$ by Lemma \ref{lem-crucial}; in particular,
\begin{equation}\label{lim-I2}
\lim_{k\to\infty}I_{Y,2}=0.
\end{equation}
Combining \eqref{lim-EFG}, \eqref{lim-H}, \eqref{lim-I1} and \eqref{lim-I2} yields \eqref{C2}.
\end{proof}
Having gathered all necessary preliminary results, we are ready to prove the main results.
\section{Differentiability of the mapping $t\mapsto H^s_{\Omega_t}(\Phi_t(x),\Phi_t(y))$}
Here we prove that, for all $x,y\in \Omega$, the real valued function $m_x^y:t\mapsto H^s_{\Omega_t}(\Phi_t(x),\Phi_t(y))$ is differentiable at zero. More generally, we will prove the following
%Then, the differentiability of $t\mapsto G^s_{\Omega_t}(x,y)$ at zero will follows from the later by using Leibnitz differentiation rule and the decomposition
%\begin{equation}\label{Decomp}
 %   G_{\Omega_t}(x,y)=b_{N,s}|x-y|^{2s-N}-H_{\Omega_t}^s(x,y)\quad\text{for all $x,y\in \R^N\setminus\{x,x\}$}.
%\end{equation}
%See Section \ref{...} for details.
%
\begin{Lemma}\label{diff-reg-part}
Let $s$ in $(0,1)$ and $x\in \Omega$ be fixed. Then, there exists $\gamma\in (0, s)$ such that, the mapping 
$$
H_x: t\to C^\gamma(\overline{\Omega}), \quad t\mapsto H_{\Omega_t}^s\big(\Phi_t(x), \Phi_t(\cdot)\big),
$$
is differentiable at $0$. 
\par\;
In particular, for any $y\in \Omega$, the real valued function $m_x^y:t\mapsto H^s_{\Omega_t}(\Phi_t(x),\Phi_t(y))\in \R$ is differentiable at zero.
%Moreover, the derivative satisfies
%$$
%\partial_t H^s_{\Omega_t}(\Phi_t(x),\Phi_t(\cdot))\Big|_{t=0}\in C^{3s}(\Omega)\cap L^\infty(\Omega).
%$$
%In particular, the derivative is in $C^1(\Omega)$ provided that $s>1/3$.
\end{Lemma}
\begin{proof}
For simplicity, we shall use the notation 
$$
z_t:=\Phi_t(z)\qquad\text{for \; $z\in\R^N$}.
$$
For a fixed $x$ in $\Omega$, we consider the function 
\begin{equation}\label{defStx}
S_x^t:\R^N\to\R, \quad y\mapsto    S_x^t(y):=\frac{H_{\Omega_t}^s(x_t,y_t)-H_\Omega^s(x,y)}{t}\quad\text{ for $t\neq 0$}.
\end{equation}
We are going to prove that $S^x_t$ can be controlled uniformly in $C^\gamma(\overline{\Omega})$ for some $\gamma\in (0,s)$.
\par\;

For that aim, we define the family of $t$-dependent uniformly elliptic operators $\mathcal{L}_t$ by
\begin{align}\label{defopl}
\big\langle\mathcal{L}_t[u],v\big\rangle:=\iint_{\R^{2N}}(u(y)-u(z))(v(y)-v(z))\kappa_t(y,z)dydz.
\end{align}
where the kernel $\kappa_t(\cdot,\cdot)$ is given by \eqref{ktyz}. We also set $\kappa_0:=\kappa_{t=0}$, i.e.\ $\kappa_0(y,z)=\frac{c_{N,s}}{2}|y-z|^{-N-2s}$.
\par\;

For $t\neq 0$, we consider the two functions $f_t[H_\Omega^s(x,\cdot)]:\Omega\to\R$ and $g^t_x:\R^N\setminus\Omega\to\R$ defined respectively by:
\begin{align}
f_t[H_\Omega^s(x,\cdot)](y)&:=p.v\int_{\R^N}(H_\Omega^s(x,y)-H_\Omega^s(x,z))\frac{\kappa_t(y,z)-\kappa_0(y,z)}{t}dz,\label{def-ftx}\\
g^t_x(y)&:=b_{N,s}\frac{|x_t-y_t|^{2s-N}-|x-y|^{2s-N}}{t}. 
\end{align}
With these notations, we easily check that $S_x^t\in C^s(\overline{\Omega})$ is a weak solution of 
\begin{equation}\label{dest-eq-Stx}
\left\{ \begin{array}{rcll} \mathcal{L}_t[S_x^t]&=& f_t[H_\Omega^s(x,\cdot)]  &\textrm{in }\Omega, \\ S_x^t&=&g^t_x&%\textrm{for } x 
\textrm{in }\R^N\setminus\Omega.
\end{array}\right. 
\end{equation}
\par\;

In other words, we have $S_x^t(y)=g^t_x(y)$ for all $y\in \R^N\setminus\Omega$ and 
\begin{equation}\label{eqStx-weak}
\big\langle\mathcal{L}_t[S_x^t],\psi\big\rangle=\int_{\Omega}f_t[H_\Omega^s(x,\cdot)](y)\psi(y)\,dy\qquad\forall \;\psi\in C^\infty_0(\Omega),
\end{equation} 
%
%where in the above, the operator $\mathcal{L}_t$ acting on $\psi$ is pointwise defined by
%$$
%\frac{1}{2}\mathcal{L}_t[\psi](y):=p.v\int_{\R^N}(\psi(y)-\psi(z))\kappa_t(y,z)dz.
%$$
%We note that in view of Proposition \ref{FundLemma}, the duality pairing in the LHS of \eqref{eqStx-weak}.
\par\;

To check the identity \eqref{eqStx-weak}, we note that if $\psi\in C^\infty_0(\Omega)$ then $\psi_t:=\psi\circ\Phi_t^{-1}\in C^{1,1}_0(\Omega_t)$. Using the latter as a test function into the equation

\begin{equation}\label{regular-part-of-Green-t}
\left\{ \begin{array}{rcll} (-\Delta)^s H^s_{\Omega_t}(x_t,\cdot)&=& 0  &\textrm{in }\Omega_t, \\ H^s_{\Omega_t}(x_t,z)&=&F_s(x_t,z)&%\textrm{for } x 
\textrm{in }\R^N\setminus\Omega_t. \end{array}\right. 
\end{equation}
and changing variables, we get, for $|t|\ll 1$ small enough,
\begin{align*}
    0&=\int_{\Omega_t}\psi_t(-\Delta)^s H^s_{\Omega_t}(x_t,\cdot)dy=\int_{\R^N}\psi_t(-\Delta)^s H^s_{\Omega_t}(x_t,\cdot)dy,\nonumber\\
    &=\frac{c_{N,s}}{2}\iint_{\R^{2N}}\frac{\big(H^s_{\Omega_t}(x_t,y)-H^s_{\Omega_t}(x_t,z)\big)\big(\psi_t(y)-\psi_t(z)\big)}{|y-z|^{N+2s}}dydz,\nonumber\\
    &=\frac{c_{N,s}}{2}\iint_{\R^{2N}}\big(H^s_{\Omega_t}(x_t,y_t)-H^s_{\Omega_t}(x_t,z_t)\big)\big(\psi(y)-\psi(z)\big)\frac{\textrm{Jac}_{\Phi_t}(y)\textrm{Jac}_{\Phi_t}(z)}{|y_t-z_t|^{N+2s}}dydz,\nonumber\\
    &=\iint_{\R^{2N}}\big(H^s_{\Omega_t}(x_t,y_t)-H^s_{\Omega_t}(x_t,z_t)\big)\big(\psi(y)-\psi(z)\big)\kappa_t(y,z)dydz,
\end{align*}
As above, we note that the integrals above are all well defined by Proposition \ref{FundLemma}.  On the other hand, we also have 
\begin{equation*}
    0=\iint_{\R^{2N}}\big(H^s_{\Omega}(x,y)-H^s_{\Omega}(x,z)\big)\big(\psi(y)-\psi(z)\big)\kappa_0(y,z)dydz,
\end{equation*}
for all $\psi\in C^\infty_0(\Omega)$. Equating the above two identities and rearranging leads to \eqref{eqStx-weak}
%Note that the bracket $\big\langle\mathcal{L}_t[S_x^t],\psi\big\rangle$ is well-defined for $|t|<\eps$ small enough, see e.g proof of Lemma \ref{lemma2.4}. 
The second equation in \eqref{dest-eq-Stx} follows immediately from the second equation in \eqref{regular-part-of-Green-t}. This proves the claim.
\par\;

Next, we let $G^{s,t}_\Omega(x,\cdot)$ and  $\operatorname{P}^{s,t}_x$ be respectively the Green function with singularity at $x$ and the Poisson kernel associated to the operator $\mathcal{L}_t$. Then by \cite[Theorem 1.1 and Corollary 9.6]{KW}, we have the estimate 
\par\;

\begin{align}
    &G^{s,t}_\Omega(x,y)\leq C_1|x-y|^{2s-N}\min\Big(1,\frac{\delta_\Omega(x)}{|x-y|}\Big)^s\min\Big(1,\frac{\delta_\Omega(y)}{|x-y|}\Big)^s\qquad\text{for $x,y\in \Omega$},\label{Gts-es}\\
    &P^{s,t}(y,z)\leq C_2|y-z|^{-N}\frac{\delta_\Omega^s(y)}{\delta_\Omega^s(z)\big(1+\delta_\Omega(z)\big)^s},\qquad\text{for $y\in \Omega$ and $z\in \R^N\setminus\Omega$}\label{Pts-es}
\end{align}
for some $C_1, C_2=C(N,s,\lambda, \Lambda, \Omega)>0$ where $\lambda, \Lambda$ are the elliptic constant given in \eqref{ellptic-kernel}.
\par\;

It is a standard argument in potential theory that the solution $S_x^t$ to the equation \eqref{dest-eq-Stx} can be written, for $y\in \Omega$, as:  
\begin{align}
S_x^t(y) = \int_{\Omega}f_t[H_\Omega^s(x,\cdot)](z)G^{s,t}_\Omega(y,z)dz+\int_{\R^N\setminus \Omega}g_x^t(z)\operatorname{P}^{s,t}(y,z)dz.
\end{align}
In view of \eqref{Gts-es} and \eqref{Pts-es} and recalling Lemma \ref{LM1} and Lemma \ref{unifest-gtx}, we obtain the bound
\begin{align}\label{Sxt-Linfty0}
|S_x^t(y)| &\leq \int_{\Omega}|f_t[H_\Omega^s(x,\cdot)](z)|G^{s,t}_\Omega(y,z)dz+\int_{\R^N\setminus \Omega}|g_x^t(z)|P^{s,t}(y,z)dz,\nonumber\\
&\leq C_3\int_{\Omega}\frac{dz}{|y-z|^{N-s}}+C_4\delta_\Omega^s(y)\int_{\R^N\setminus\Omega}\frac{dz}{\delta_\Omega^s(z)|x-z|^{N-2s}|y-z|^{N}},\nonumber\\
&\leq C_5+C_6\delta_\Omega^s(y)\int_{\R^N\setminus \Omega}\frac{dz}{\delta_\Omega^s(z)|y-z|^N }.
\end{align}
To estimate the last integral, we let $r_\Omega=1+\textrm{diam($\Omega)$}$ and write 
\begin{align*}
\int_{\R^N\setminus \Omega}\frac{dz}{\delta_\Omega^s(z)|y-z|^N }    = \int_{B_{r_\Omega}(y)\setminus \Omega}\frac{dz}{\delta_\Omega^s(z)|y-z|^N } +\int_{\R^N\setminus B_{r_\Omega}(y)}\frac{dz}{\delta_\Omega^s(z)|y-z|^N }.
\end{align*}
On the other hand, by \cite[Lemma 2.3]{SvenSaldanaTobias}, we have 
\begin{equation*}
\int_{B_{r_\Omega}(y)\setminus \Omega}\frac{dz}{\delta_\Omega^s(z)|y-z|^N }\leq C_7(1+\delta_\Omega^{-s}(y)).    
\end{equation*}
On the other hand, we may find $C_8>0$ such that $\delta_\Omega(z)\geq C_8|y-z|$ for all $z\in \R^N\setminus B_{r_\Omega}(y)$. Consequently
\begin{align*}
\int_{\R^N\setminus B_{r_\Omega}(y)}\frac{dz}{\delta_\Omega^s(z)|y-z|^N }&\leq C_8^s\int_{\R^N\setminus B_{r_\Omega}(y)}\frac{dz}{|y-z|^{N+s}}\nonumber\\
&\leq C_8^s\int_{\R^N\setminus B_{r_\Omega}(0)}\frac{dz}{|z|^{N+s}}\leq C_9.
\end{align*}
In conclusion, we have 
\begin{align*}
\int_{\R^N\setminus \Omega}\frac{dz}{\delta_\Omega^s(z)|y-z|^N }\le C_{10}(1+\delta_\Omega^{-s}(y)).
\end{align*}
Plug this into \eqref{Sxt-Linfty0} gives 
\begin{equation}\label{Sxt-Linfty1}
    \|S_x^t\|_{L^\infty(\Omega)}\leq C_{11},
\end{equation}
for some $C_{11}>0$ that is independent of $t$.
\par\;

By \eqref{ellptic-kernel}, Lemma \ref{Useful-bounds}, and the definition \eqref{ktyz}, the family of kernels $\kappa_t(\cdot,\cdot)$ satisfies the assumptions in Proposition \ref{Dir-non-boundary-data} (in particular, \eqref{Sauveur1} and \eqref{Sauveur2}) uniformly for $|t|\ll 1$. Moreover, Lemma \ref{unifest-gtx} shows that the exterior datum $g_x^t$ satisfies \eqref{unif-est-extdata-cond}. We can therefore use the latter to get, for $p\in (0,s]$ and $q\in  (\frac{N}{s+p},\frac{N}{p})$, that
\begin{align}
\|S_x^t\|_{C^{s+q-\frac{N}{p}}(\overline{\Omega})}&\leq C_1C_0\Big(1+\|S_x^t\|_{L^1_{2s}(\R^N)}+\|f_t[H_\Omega^s(x,\cdot)]\|_{X(\Omega)}\Big),\nonumber\\
&\leq C_2C_0\Big(1+\|S_x^t\|_{L^\infty(\Omega)}+\sup_{w\in H^s_0(\Omega)}\frac{\int_{\Omega}|f_t[H_\Omega^s(x,\cdot)](z)||w(z)|dz}{[w]_{H^s(\R^N)}}\Big)\nonumber\\
&\leq C_3C_0\Big(1+\|S_x^t\|_{L^\infty(\Omega)}+\sup_{w\in H^s_0(\Omega)}\frac{\int_{\Omega}\delta^{-s}_\Omega(z)|w(z)|dz}{[w]_{H^s(\R^N)}}\Big)\nonumber\\
&\leq C_4,
\end{align}
for some $C_4>0$ that is independent of $t$. In the second line we used that $\|S_x^t\|_{L^1_{2s}(\R^N)}\leq C\big(1+\|S_x^t\|_{L^\infty(\Omega)}\big)$. Indeed, since $\Omega$ is bounded and $S_x^t=g_x^t$ on $\R^N\setminus\Omega$, Lemma \ref{unifest-gtx} yields $|g_x^t(y)|\leq C|x-y|^{2s-N}$ and therefore
\begin{align*}
\|S_x^t\|_{L^1_{2s}(\R^N)}
&\leq \int_{\Omega}|S_x^t(y)|\,dy +\int_{\R^N\setminus\Omega}\frac{|g_x^t(y)|}{1+|y|^{N+2s}}\,dy\\
&\leq C\|S_x^t\|_{L^\infty(\Omega)}+C\int_{\R^N\setminus\Omega}\frac{|x-y|^{2s-N}}{1+|y|^{N+2s}}\,dy
\leq C\big(1+\|S_x^t\|_{L^\infty(\Omega)}\big),
\end{align*}
where the last integral is finite because $|x-y|\geq \delta_\Omega(x)$ for $y$ in $\R^N\setminus\Omega$ and the integrand behaves like $|y|^{-2N}$ at infinity. In the last line we used \eqref{Sxt-Linfty1} and that $\int_{\Omega}\delta^{-s}_\Omega(z)|w(z)|dz\leq |\Omega|^{1/2}\Big(\int_{\Omega}\delta^{-2s}_\Omega(z)|w(z)|^2dz\Big)^{1/2}\leq |\Omega|^{1/2}C[w]_{H^s(\R^N)}$ by the classical Hardy inequality.
\par\;

Then by the Arzel\`a--Ascoli theorem, there exists $S_0^x\in C^{s+q-\frac{N}{p}}(\overline{\Omega})$ such that, up to passing to a subsequence $(S_x^{t'})_{t'}$, we have:
\begin{equation}\label{Holder-convergence}
    S^{t'}_x\quad\to\quad S^0_x\qquad\text{in $C^\beta(\overline{\Omega})$ for all $0\leq \beta<s-(\frac{N}{p}-q)$.}
\end{equation}

Next, we need to find the equation that the limiting function solves. For that, we recall that 

\begin{align}\label{D1}
\iint_{\R^{2N}}(S^{t'}_x(y)-S^{t'}_x(z))(\psi(y)-\psi(z))\kappa_{t'}(y,z)\,dy\,dz=\int_{\Omega}f_{t'}[H_\Omega^s(x,\cdot)](z)\psi(z)\,dz.
\end{align}
\par\;

By \eqref{pointwise-ft-lim} and \eqref{pointwise-ftY-es}, we know that
\par\;
$$
|f_{t'}[H_\Omega^s(x,\cdot)](z)|\leq C\delta_\Omega^{-s}(z)\qquad\text{and}\qquad f_{t'}[H_\Omega^s(x,\cdot)](z)\to f_Y[H_\Omega^s(x,\cdot)](z)\quad \text{as $t'\to 0$}.
$$
It therefore follows by the Lebesgue dominated convergence theorem that
\begin{equation}\label{D2}
\int_{\Omega}f_{t'}[H_\Omega^s(x,\cdot)](z)\psi(z)\,dz \quad\to\quad \int_{\Omega}f_Y[H_\Omega^s(x,\cdot)](z)\psi(z)\,dz\qquad\text{as\; $t'\to 0$}.
\end{equation}
On the other hand, because $\psi$ is compactly supported in $\Omega$, we may write:
\begin{align}\label{D3}
&\iint_{\R^{2N}}(S^{t'}_x(y)-S^{t'}_x(z))(\psi(y)-\psi(z))\kappa_{t'}(y,z)\,dy\,dz\nonumber\\ 
&=\iint_{\Omega\times \Omega}(S^{t'}_x(y)-S^{t'}_x(z))(\psi(y)-\psi(z))\kappa_{t'}(y,z)\,dy\,dz\nonumber\\
&\quad+2\int_{\supp(\psi)}\psi(y)\int_{\R^N\setminus \Omega}(S^{t'}_x(y)-S^{t'}_x(z))\kappa_{t'}(y,z)\,dz\,dy.
\end{align}
In view of \eqref{Holder-convergence} and the smoothness of $\psi$, for every $\beta$ in $(0,s-(\frac{N}{p}-q))$ there holds
$$
|S^{t'}_x(y)-S^{t'}_x(z)|\,|\psi(y)-\psi(z)|\,|\kappa_{t'}(y,z)|\leq C|y-z|^{\beta+1-N-2s}\qquad\text{for all $y,z$ in $\Omega$.}
$$
Choosing $\beta>2s-1$ (which is possible since $s<1$ and we may take $\beta$ close to $s$), the right-hand side is integrable on $\Omega\times\Omega$. Since $\kappa_{t'}(y,z)\to \kappa_0(y,z)$ pointwise as $t'\to 0$ by \eqref{ktyz}, the dominated convergence theorem gives that 
\begin{align}\label{D4}
&\lim_{t'\to 0}\iint_{\Omega\times \Omega}(S^{t'}_x(y)-S^{t'}_x(z))(\psi(y)-\psi(z))\kappa_{t'}(y,z)\,dy\,dz\nonumber\\
&=\iint_{\Omega\times \Omega}(S^0_x(y)-S^0_x(z))(\psi(y)-\psi(z))\kappa_0(y,z)\,dy\,dz.
\end{align}
On the other hand, using that $S_x^t=g^t_x$ in $\R^N\setminus\Omega$, we get --for $y\in \supp(\psi)\Subset \Omega$ and $z\in \R^N\setminus\Omega$-- that 
\begin{align*}
    |S^{t'}_x(y)-S^{t'}_x(z)||\kappa_{t'}(y,z)|&=|S^{t'}_x(y)-g^{t'}_x(z)||\kappa_{t'}(y,z)|\\
    &\leq C_1(1+|x-z|^{2s-N})|y-z|^{-2s-N}\\
    &\leq C_2(1+|x-z|^{2s-N})(1+|z|^2)^{-\frac{2s+N}{2}}
\end{align*}
Again, since 
$$
\int_{\Omega}\int_{\R^N\setminus\Omega}(1+|x-z|^{2s-N})(1+|z|^2)^{-\frac{2s+N}{2}}dzdy<+\infty,
$$
we deduce again by the Lebesgue dominated convergence theorem that 
\begin{align}\label{D5}
&\lim_{t'\to 0}\int_{\supp(\psi)}\psi(y)\int_{\R^N\setminus \Omega}(S^{t'}_x(y)-S^{t'}_x(z))\kappa_{t'}(y,z)\,dz\,dy\nonumber\\
&=\int_{\supp(\psi)}\psi(y)\int_{\R^N\setminus \Omega}(S^{0}_x(y)-g_Y^x(z))\kappa_0(y,z)\,dz\,dy.
\end{align}
\par\;

Next define
\begin{equation*}
S_x(z):=
\begin{cases}
S_x^0(z) & \text{if } z\text{ in }\Omega,\\
g_Y^x(z) & \text{if } z\text{ in }\R^N\setminus\Omega,
\end{cases}
\end{equation*}
where
\begin{equation*}
g_Y^x(z)=(2s-N)b_{N,s}\frac{(x-z)\cdot(Y(x)-Y(z))}{|x-z|^{N-2s+2}}.
\end{equation*}
\par\;

Combining \eqref{D1}, \eqref{D2}, \eqref{D3}, \eqref{D4} and \eqref{D5}, we get
\begin{align*}
\frac{c_{N,s}}{2}\iint_{\R^{2N}}\frac{(S_x(y)-S_x(z))(\psi(y)-\psi(z))}{|y-z|^{N+2s}}dydz=\int_{\Omega}f_Y[H_\Omega^s(x,\cdot)](z)\psi(z)\,dz\quad\text{for all $\psi\in C^\infty_0(\Omega)$.}  
\end{align*}
\par\;

  In other words, the limiting function $S_x$ is a distributional solution of 
\begin{equation}\label{EqS0x}
\left\{ \begin{array}{rcll} (-\Delta)^sS_x&=& f_Y[H_\Omega^s(x,\cdot)]  &\;\textrm{in }\;\Omega, \\ S_x&=&g_Y^x&%\textrm{for } x 
\;\textrm{in }\;\R^N\setminus\Omega.
\end{array}\right. 
\end{equation}
 To recap, we have proved that along a subsequence $t'$,
\begin{equation}
S_x^{t'}:=\frac{H^s_{\Omega_{t'}}(\Phi_{t'}(x),\Phi_{t'}(\cdot))-H_\Omega^s(x,\cdot)}{t'} \quad\to\quad  S_x \quad\text{in $C^\beta(\overline{\Omega})$ for all $\beta\in\big(0,s-(\frac{N}{p}-q)\big)$.}
\end{equation}
 and that $S_x$ solves \eqref{EqS0x} in the sense of distribution. By the standard maximum principle \cite[Theorem 2.6.12]{FRO}, we know that the latter has a unique solution. Therefore every convergent subsequence of $(S_x^t)$ must converge to $S_x$. Since we already have a convergent subsequence, we conclude that the whole sequence $S_x^t$ converges to $S_x$ in $C^\beta(\overline{\Omega})$. In other words, we have proved that, the mapping 
\begin{align}\label{Holder-est-Stx}
H_x^t:(-\epsilon_0,+\epsilon_0)\to C^\beta(\overline{\Omega}), \;t\mapsto H^s_{\Omega_t}(\Phi_t(x),\Phi_t(\cdot))
\end{align}
is differentiable at zero for all $\beta \in [0,s+q-\frac{N}{p})$ and $p\in (0,s]$, $q\in (\frac{N}{s+p},\frac{N}{p})$ sufficiently close to $N/p$. The proof is therefore finished.
\end{proof}

\section{First step toward the derivation of the formula \eqref{var-green}}

%From the section above, we know that the function .... is differentiable....blabla ...

Consider the function $g_s(x,\cdot):\R^N\to \R$ defined, for $y\in \Omega$, by  
\begin{equation}\label{gsxy}
    g_s(x,y)=\partial_t \Big[G_{\Omega_t}(\Phi_t(x),\Phi_t(\cdot))\Big]\big|_{t=0}(y),
\end{equation}
and extended by zero outside of $\Omega$. In this section we look for the equation that $g_s(x,\cdot):\R^N\to\R$ solves. This is the first step toward the derivation of the formula \eqref{var-green} in Theorem \ref{THM-1.5}. We note that $g_s(x,\cdot)\in L^1(\Omega)$ as a consequence of  Lemma \ref{diff-reg-part}  and the decomposition \eqref{Eq-apliting-of-Green}. In fact, we have even more. We have indeed 
\begin{equation}\label{l1-deltas-reg}
g_s(x,\cdot)\in L^1(\Omega; \delta^{-s}_\Omega):=\Big\{u:\Omega\to\R:\, u(\cdot)\delta_\Omega^{-s}(\cdot)\in L^1(\Omega)\Big\}. 
\end{equation}
This follows from the decomposition \eqref{Eq-apliting-of-Green},  Lemma \ref{diff-reg-part} and the fact that 
$$
\Big|\delta^{-s}_\Omega(\cdot)\partial_t|\Phi_t(x)-\Phi_t(\cdot)|^{2s-N}\Big|_{t=0}\Big|\leq C\delta^{-s}_\Omega(\cdot)|x-\cdot|^{2s-N}\in L^1(\Omega),
$$
and this fact will be used in the proof of Lemma \ref{main-res-sec4.4} below.

For the sake of simplicity, we let 
 \begin{align}\label{duality-pairing}
 E_Y\big(u,v\big):=
\iint_{\R^{2N}}\big(u(y)- u(z)\big)\big(v(y)-v(z)\big)\kappa_Y(y,z)\, dydz,
\end{align}   
for any sufficiently regular functions $u,v$ and where  $\kappa_Y(\cdot,\cdot)$ is given as in \eqref{kappaY}.
\begin{Lemma}\label{lem-eq-ahape-deriv-Green}
Let $s$ in $(0,1)$ and $ x\in \Omega$ be fixed. Let $g_s(x,\cdot):\R^N\to \R$ be given as above. Then, for all $\psi\in C^\infty_0(\Omega)$, we have 
\begin{align}\label{eq-ahap-green}
\int_{\Omega}g_s(x,\cdot)(-\Delta)^s\psi\,dz
=- E_Y\big(G_\Omega^s(x,\cdot),\psi\big) .
\end{align}
\end{Lemma}
\par\;

We note that since $\kappa_Y(y,z)$ is comparable to $|y-z|^{-N-2s}$, the RHS of the identity above is well defined thanks to Lemma \ref{FundLemma}. Moreover, since $(-\Delta)^s\psi\in L^\infty(\Omega)$, we also know that the LHS is well defined by because $g_s(x,\cdot)\in L^1(\Omega)$.
\begin{proof}[Proof of Lemma \ref{lem-eq-ahape-deriv-Green}]
Let $x$ in $\Omega$ be fixed. For the sake of simplicity, we set
$$
g_t(x_t,z_t):=G_{\Omega_t}^s(\Phi_t(x),\Phi_t(z))\quad\text{and}\quad g_0(x,z)=G_\Omega^s(x,z)\quad\text{for all $z\in\R^N$}.
$$
Next, we consider the function $q^t_x:\R^N\to\R$ defined by
\begin{equation}\label{Qtx}
q^t_x(z):=\frac{g_t(x_t,z_t)-g_0(x,z)}{t}\qquad\textup{for $t\neq 0$}.
\end{equation} 
For any $\psi\in C^\infty_0(\Omega)$, we also define
\begin{align}\label{def-R-t}
 \big\langle q^t_x,\psi\big\rangle := \iint_{\R^{2N}}\big(q^t_x(y)-q^t_x(z)\big)\big(\psi(y)-\psi(z)\big)\big(\kappa_t(y,z)-\kappa_0(y,z)\big)dydz,
\end{align}
\begin{align}\label{def-a-t}
\big\langle g_0(x,\cdot),\psi\big\rangle_t:=\iint_{\R^{2N}}\big(g_0(x,y)-g_0(x,z)\big)\big(\psi(y)-\psi(z)\big)\frac{\kappa_t(y,z)-\kappa_0(y,z)}{t}dydz.
\end{align}
Then, observe that
\,\,
\begin{align} \label{Eq-distr-Gt}
\int_{\Omega}q^t_x(-\Delta)^s \psi\,dy=-\big\langle q^t_x,\psi\big\rangle-\big\langle g_0(x,\cdot),\psi\big\rangle_t\quad\text{for all $\psi\in C^\infty_0(\Omega)$.}
\end{align}
To see this, for $\psi\in C^\infty_0(\Omega)$, we set
\begin{gather*}
 \mathcal E_t(x) := 
\iint_{\R^{2N}}\big(g_t(x_t,y_t)-g_t(x_t,z_t)\big)\big(\psi(y)-\psi(z)\big)\kappa_t(y,z)dydz,
\end{gather*}
where we recall that  $\kappa_t$ is the kernel  defined in \eqref{ktyz}.
We note that the integral above is finite for $t$ small by Proposition \ref{FundLemma}.  Then,   we note that for any $x$ in $\Omega$ and for any small $t$, we have
\begin{equation} \label{same}
\psi(x)=   \mathcal E_t(x) . 
\end{equation}
  Indeed, 
    by using the Green representation formula in the domain $\Omega_t$, since $\psi_t:=\psi\circ \Phi_t^{-1}\in C^{1,1}_0(\Omega_t)$, we have
\begin{align}
\psi(x)&=  \int_{\Omega_t}G^s_{\Omega_t}(x_t,y)(-\Delta)^s\psi_t(y)dy,\nonumber\\
&= \iint_{\R^{2N}}\big(G^s_{\Omega_t}(x_t,y)-G^s_{\Omega_t}(x_t,z)\big)\big(\psi_t(y)-\psi_t(z)\big) \kappa_0(y,z)
\,   dydz,\nonumber
\end{align}     which, by changing variables, leads to \eqref{same}. 
\par\;

Now to see \eqref{Eq-distr-Gt}, we are going to compute  in two different ways the following ratio: 
\begin{equation} \label{ratio}
    \frac{\mathcal E_t(x)-\mathcal E_0(x)}{t} \quad\text{for $t\neq 0$}
\end{equation}
 The first way is very easy: it follows from 
\eqref{same} that the ratio in  \eqref{ratio} is $0$ since the left hand side of 
\eqref{same}  does not depend on $t$. 
Next, we claim that the ratio in  \eqref{ratio} is also equal to
\begin{align} \label{60}
\big\langle g_0(x,\cdot),\psi\big\rangle_t +\big\langle q^t_x,\psi\big\rangle+\int_{\Omega}q^t_x(y)(-\Delta)^s \psi(y)dy 
\end{align}
Indeed, the ratio in  \eqref{ratio} is equal to
\begin{align}\label{60B}
&t^{-1}\iint_{\R^{2N}}\Big\{\big(g_t(x_t,y_t)-g_t(x_t,z_t)\big)\kappa_t(y,z)-\big(g_0(x,y)-g_0(x,z)\big)\kappa_0(y,z)\Big\}\big(\psi(y)-\psi(z)\big)dydz,\nonumber\\
&=\iint_{\R^{2N}}\big(g_0(x,y)-g_0(x,z)\big)\big(\psi(y)-\psi(z)\big)\frac{\kappa_t(y,z)-\kappa_0(y,z)}{t}dydz,\nonumber\\
&+\iint_{\R^{2N}}\big(q^t_x(y)-q^t_x(z)\big)\big(\psi(y)-\psi(z)\big)\kappa_t(y,z)dydz\nonumber\\
&=\big\langle g_0(x,\cdot),\psi\big\rangle_t+\iint_{\R^{2N}}\big(q^t_x(y)-q^t_x(z)\big)\big(\psi(y)-\psi(z)\big)\kappa_0(y,z)dydz+\big\langle q^t_x,\psi\big\rangle,
\end{align}
what leads to \eqref{60}. Combining \eqref{same}, \eqref{ratio}, \eqref{60}, and \eqref{60B} gives \eqref{Eq-distr-Gt}.
\par\;

We are now going to pass to the limit as $t \rightarrow 0$ into \eqref{Eq-distr-Gt}. 
We start with the first term. 
By \eqref{expansion-k-t-notations}, we have 
\begin{equation}
 \kappa_t(y,z)-\frac{c_{N,s}}{2}|y-z|^{-N-2s}-t\kappa_Y(y,z)=O(t^2)|y-z|^{-N-2s}
\end{equation}
uniformly for all $y\neq z\in \R^N$.
Consequently, and since 
$$
\iint_{\R^{2N}}\big|g_0(x,y)-g_0(x,z)\big|\big|\psi(y)-\psi(z)\big||y-z|^{-N-2s}dydz<+\infty\quad\text{(by Proposition \ref{FundLemma}),}
$$
 we get by the Lebesgue dominated convergence theorem that 
\begin{align}\label{1st-lim}
    \lim_{t\to 0} \big\langle g_0(x,\cdot),\psi\big\rangle_t&=\lim_{t\to 0}\iint_{\R^{2N}}\big(g_0(x,y)-g_0(x,z)\big)\big(\psi(y)-\psi(z)\big)\frac{\kappa_t(y,z)-\kappa_0(y,z)}{t}dydz,\nonumber\\
    &=\iint_{\R^{2N}}\big(g_0(x,y)-g_0(x,z)\big)\big(\psi(y)-\psi(z)\big)\kappa_Y(y,z)dydz,\nonumber\\
    &=
     E_Y\big(g_0(x,\cdot), \psi\big)= E_Y\big(G_\Omega^s(x,\cdot), \psi\big)
\end{align}
Next, since $g_t(x_t,z_t)=b_{N,s}|x_t-z_t|^{2s-N}-H_{\Omega_t}^s(x_t,z_t)$, we may write
\begin{align}\label{TGH}
    |q^t_x(z)|&=  \Big|\frac{g_t(x_t,z_t)-g_0(x,z)}{t}\Big|\nonumber\\
    &\leq b_{N,s}\Big|\frac{|x_t-z_t|^{2s-N}-|x-z|^{2s-N}}{t}\Big|+\Big|\frac{H_{\Omega_t}^s(x_t,z_t)-H_{\Omega}^s(x,z)}{t}\Big|\nonumber\\
    &\leq C(1+|x-z|^{2s-N}),
\end{align}
where we used \eqref{UExp} and \eqref{Sxt-Linfty1}. Since $|x-\cdot|^{2s-N}\in L^1(\Omega)$, we may use the Lebesgue dominated convergence theorem that
\begin{align}\label{2nd-lim}
   \lim_{t\to 0}\int_{\Omega}q^t_x(-\Delta)^s \psi\,dz =\int_{\Omega} g_s(x,\cdot)(-\Delta)^s\psi\,dz.
\end{align}
And finally we claim that 
\begin{align}\label{3th-limit}
\lim_{t\to 0}\big\langle q^t_x,\psi\big\rangle=0 .
\end{align}
The identity \eqref{eq-ahap-green}, tested with the function $\psi$, then follows by combining \eqref{Eq-distr-Gt}, \eqref{1st-lim}, \eqref{2nd-lim} and \eqref{3th-limit}, so that, up to proving \eqref{3th-limit} the proof of Lemma \ref{lem-eq-ahape-deriv-Green} is finished. 
\par\;

\underline{\textbf{Proof of \eqref{3th-limit}}:} We use the decomposition of the fractional Green function $G_\Omega^s(x,\cdot)$ 
into the sum of the fundamental solution $b_{N,s}|x-\cdot|^{2s-N}$ and of the $H^s_\Omega(x,\cdot)$ and write $q^t_x = g^t_x-S_x^t$ so that
\begin{equation}\label{decopm-gxt-prod-psi}
 \big\langle q^t_x,\psi\big\rangle=\big\langle g^t_x,\psi\big\rangle-\big\langle S_x^t,\psi\big\rangle   
\end{equation}
where $g^t_x$ and $S_x^t$ are respectively defined in \eqref{defgtx} and \eqref{defStx}. We start by estimating $\big\langle S_x^t,\psi\big\rangle$ which we break into two pieces as follows:
\begin{align}\label{bracketst-phi}
 \big|\big\langle S_x^t,\psi\big\rangle \big|&:= \Big|\iint_{\R^{2N}}\big(S_x^t(y)-S_x^t(z)\big)\big(\psi(y)-\psi(z)\big)\big(\kappa_t(y,z)-\kappa_0(y,z)\big)dydz\Big|,\nonumber\\
 &\leq \iint_{\Omega\times\Omega}\Big|\big(S_x^t(y)-S_x^t(z)\big)\big(\psi(y)-\psi(z)\big)\big(\kappa_t(y,z)-\kappa_0(y,z)\big)\Big|dydz,\nonumber\\
&\quad +2\int_{\Omega}|\psi(y)S_x^t(y)|\int_{\R^N\setminus\Omega}|\kappa_t(y,z)-\kappa_0(y,z)|dydz.
\end{align}

By \eqref{Holder-est-Stx}, we know that $S_x^t\in C^{\beta}(\overline{\Omega})$ for all $0\leq\beta<{s-(\frac{N}{p}-q)}$ with $p\in (0,s]$ and $q\in (\frac{N}{s+p}, \frac{N}{p})$ sufficiently close to $N/p$. Moreover, $\|S_x^t\|_{C^{s-(\frac{N}{p}-q)}(\overline{\Omega})}\leq C$ for some $C>0$ independent of $t$. Consequently, we have 

\begin{equation*}
\Big|\big(S_x^t(y)-S_x^t(z)\big)\big(\psi(y)-\psi(z)\big)\big(\kappa_t(y,z)-\kappa_0(y,z)\big)\Big|\leq C|y-z|^{s-(\frac{N}{p}-q)}|y-z|^{-2s-N}|\psi(y)-\psi(z)|   
\end{equation*}
\par\;

for some constant $C>0$ independent of $t,y$ and $z$ where we used that $\big|\kappa_t(y,z)-\kappa_0(y,z)\big|\leq C|y-z|^{-2s-N}$ the latter being a simple consequence of \eqref{expansion-k-t-notations}. We also have that
\begin{equation*}
\Big|\big(S_x^t(y)-S_x^t(z)\big)\big(\psi(y)-\psi(z)\big)\big(\kappa_t(y,z)-\kappa_0(y,z)\big)\Big|\leq C    \big|\kappa_t(y,z)-\kappa_0(y,z)\big|\quad \to \quad 0.
\end{equation*}
Since 
$$
\iint_{\Omega\times\Omega}|y-z|^{s-(\frac{N}{p}-q)}|y-z|^{-2s-N}|\psi(y)-\psi(z)| dydz<+\infty,
$$
we may use the Lebesgue dominated convergence theorem to conclude that
\begin{align}\label{bracketst-phi-1}
\iint_{\Omega\times\Omega}\Big|\big(S_x^t(y)-S_x^t(z)\big)\big(\psi(y)-\psi(z)\big)\big(\kappa_t(y,z)-\kappa_0(y,z)\big)\Big|dydz\to 0\quad\text{as $t\to 0$}.
\end{align}
On the other hand, since $\|S_x^t\|_{L^\infty(\Omega)}\leq C$, we also have
\begin{align}\label{bracketst-phi-2}
&\int_{\Omega}|\psi(y)S_x^t(y)|\int_{\R^N\setminus\Omega}\big|\kappa_t(y,z)-\kappa_0(y,z)\big|dz\,dy\nonumber\\
&\leq C\int_{\Omega}\int_{\R^N\setminus\Omega}\big|\kappa_t(y,z)-\kappa_0(y,z)\big|dydz \quad\to\quad 0\qquad\text{as $t\to 0$.}
\end{align}
Combining \eqref{bracketst-phi}, \eqref{bracketst-phi-1} and \eqref{bracketst-phi-2} we get 
\begin{equation}\label{lim-sxt-prod-psi}
\big|\big\langle S_x^t,\psi\big\rangle \big|\quad\to\quad 0\quad\textrm{as $t\to 0$}.
\end{equation}
\par\;

It remains to estimate $ \big|\big\langle g^t_x,\psi\big\rangle \big|$. For this, first note that since $t^{-1}\big|\kappa_t(y,z)-\kappa_0(y,z)\big|\leq C|y-z|^{-2s-N}$, we have 
\begin{align}\label{JM0}
 \big|\big\langle g^t_x,\psi\big\rangle \big|&:= \Big|\iint_{\R^{2N}}t\big(g^t_x(y)-g^t_x(z)\big)\big(\psi(y)-\psi(z)\big)t^{-1}\big(\kappa_t(y,z)-\kappa_0(y,z)\big)dydz\Big|,\nonumber\\
 &\leq \iint_{\R^{2N}}t\big|g^t_x(y)-g^t_x(z)\big|\big|\psi(y)-\psi(z)|y-z|^{-2s-N}dydz.
\end{align}
\par\;

Next observe that, by \eqref{diff-quotient-fund-sol}, we have:
\begin{align}\label{JM1}
t\big|g^t_x(y)-g^t_x(z)\big| &=  b_{N,s}   \Big||x_t-y_t|^{2s-N}-|x-y|^{2s-N}-\big(|x_t-z_t|^{2s-N}-|x-z|^{2s-N}\big)\Big|\nonumber\\
&\leq  \Big||x_t-y_t|^{2s-N}-|x_t-z_t|^{2s-N}\Big|+  \Big||x-y|^{2s-N}-|x-z|^{2s-N}\Big|\nonumber\\
&\leq Cf_x(y,z),
\end{align}
for some constant $C>0$ independent of $t,y$ and $z$ where 
\begin{equation*}
f_x(y,z)=
\begin{cases}
\max\big(|x-y|^{2s-N},|x-z|^{2s-N}\big) & \text{if } 2s<1,\\
\big|y-z\big|\max\big(|x-y|^{2s-N-1},|x-z|^{2s-N-1}\big) & \text{if } 2s>1,\\
|y-z|^\beta\max\big(|x-y|^{1-N-\beta},|x-z|^{1-N-\beta}\big)\quad\text{for all $\beta\in (0,1)$}&\text{if $s=1/2$}.
\end{cases}
\end{equation*}
By \eqref{dual-prod-f_xyz-psi}, we also know that 
\begin{align}\label{JM2}
 \iint_{\R^{2N}}\frac{|\psi(y)-\psi(z)|}{|y-z|^{N+2s}}f_x(y,z)dydz<+\infty.
\end{align}
And finally we have 
\begin{equation}\label{JM3}
\frac{t|g^t_x(y)-g^t_x(z)|}{b_{N,s}}\leq \Big||x_t-y_t|^{2s-N}-|x-y|^{2s-N}\Big|+\Big||x_t-z_t|^{2s-N}-|x-z|^{2s-N}\Big|\;\to\; 0\quad\text{as $t\to 0$}.
\end{equation}
Combining \eqref{JM0},\eqref{JM1},  \eqref{JM2} and \eqref{JM3} we get by the Lebesgue dominated convergence theorem that 
\begin{equation}\label{lim-gxt-prod-psi}
    \big|\big\langle g^t_x,\psi\big\rangle \big|\quad\to\quad 0.
\end{equation}
Combining \eqref{decopm-gxt-prod-psi}, \eqref{lim-sxt-prod-psi} and \eqref{lim-gxt-prod-psi} we get \eqref{3th-limit}. This finishes the proof of the lemma.
\end{proof}
\par\;

We end this section with the following convergence result.  For the sake of simplicity, we shall use the following notations:
\begin{align}
%&g_s(x,\cdot):\R^N\to\R,\; y\mapsto g_s(x,y):=\partial_t\Big[G^s_{\Omega_t}\Big(\Phi_t(x),\Phi_t(\cdot)\Big)\Big]\Big|_{t=0}(y),\label{defg}\\
&h_s(x,\cdot):\Omega\to\R,\; y\mapsto h_s(x,y):=\partial_t\Big[H^s_{\Omega_t}\Big(\Phi_t(x),\Phi_t(\cdot)\Big)\Big]\Big|_{t=0}(y),\label{defh}\\
&f_s(x,\cdot):\Omega\to\R,\; y\mapsto f_s(x,y):=(2s-N)b_{N,s}\frac{(x-y)\cdot(Y(x)-Y(y))}{|x-y|^{N-2s+2}}\label{deff}.
\end{align}
Note that $g_s(x,\cdot)=f_s(x,\cdot)-h_s(x,\cdot)\in C^0(\Omega\setminus\{x\})$. With this,  we can now state and prove the following result:
\begin{Lemma}\label{main-res-sec4.4}
Let $x,y$ in $\Omega$ with $x\neq y$, and recall the notation  $g_\gamma(y,\cdot):=G_\Omega^s(y,\cdot)\phi_y^\gamma$ and $g_s(x,\cdot):\R^N\to\R$ defined by \eqref{gsxy}. Next, we set
\begin{equation}
    b_Y[\rho](x):=b_{N,s}(N-2s)\int_{\R^N}\frac{[DY(x)\cdot z]\cdot z}{|z|^{N-2s+2}}(-\Delta)^s(\rho\circ |\cdot|^2)(z)dz
\end{equation}
and defined
\begin{align}\label{def-J}
J^{k}_{\mu,\gamma}(x,y) :=  \int_{\Omega}g_s(x,\cdot)(-\Delta)^s\big(\xi_k^2 g_\gamma(y,\cdot)\phi_x^{\mu}\big)\,dz.
\end{align}
Then we have 
\begin{align}\label{C}
\lim_{\gamma\to 0^+}\lim_{\mu\to 0^+}\lim_{k\to\infty}J^k_{\mu,\gamma}(x,y)
= g_s(x,y)+G^s_\Omega(x,y)b_Y[\rho](x).
\end{align}
\end{Lemma}
\par\;

For the proof, we will need the following result the proof of which follows by a similar argument as in \cite[Lemma 2.4]{Sidy-Franck-BP}.
\begin{Lemma}\label{lemma-green}
Fix $x$ in $\Omega$ and let $j(x,\cdot)\in C(\Omega\setminus\{x\})$. Then for all $y\in\Omega$ with $y\neq x$, there holds 
\begin{equation}
    \lim_{\gamma\to 0^+}\int_{\Omega}j(x,\cdot)(-\Delta)^sg_\gamma(y,\cdot) dz = j(x,y).
\end{equation}
\end{Lemma}
With this lemma at hand, we can now proceed with the proof of Lemma \ref{main-res-sec4.4}.
\begin{proof}[Proof of Lemma \ref{main-res-sec4.4}]
We recall the product rule for the fractional Laplacian:
 \begin{equation}\label{pl}
(-\Delta)^s (uv) =  v(-\Delta)^s u + u (-\Delta)^s v  -\mathcal I_s [u,v],
\end{equation}
where
\begin{align}\label{def-I-a}
\mathcal I_s [u,v] :=  c_{N,s} \, \pv\int_{\R^N} \frac{(u(\cdot)-u(y)) (v(\cdot)-v(y))}{|\cdot-y|^{N+2s}}dy .
\end{align}
 We use the product rule above to write 
 \begin{align}\label{C1-1}
(-\Delta)^s\big(\xi_k^2 \phi_x^{\mu}g_\gamma(y,\cdot)\big)&=\xi_k^2\phi_x^{\mu}(-\Delta)^sg_\gamma(y,\cdot)+g_\gamma(y,\cdot)(-\Delta)^s\big( \xi_k^2\phi_x^{\mu}\big)  -\mathcal I_s\big[\xi_k^2\phi_x^{\mu};\;g_\gamma(y,\cdot) \big].
 \end{align}
 \par\;
 
First observe that since $g_s(x,\cdot)\in C\big(\Omega\setminus\{x\}\big)$, then Lemma \ref{lemma-green} yields
 \begin{align}\label{C1-2}
     &\lim_{\gamma\to 0^+}\lim_{\mu\to 0^+}\lim_{k\to\infty}\int_{\Omega}\xi_k^2(z)\phi_x^{\mu}(z)g_s(x,z)(-\Delta)^sg_\gamma(y,\cdot)(z)dz\nonumber\\
     &=\lim_{\gamma\to 0^+}\lim_{\mu\to 0^+}\int_{\Omega}\phi_x^{\mu}(z)g_s(x,z)(-\Delta)^sg_\gamma(y,\cdot)(z)dz\nonumber\\
     &=g_s(x,y).
 \end{align} 
\par\;
Next we claim that 
\begin{align}\label{probleme}
\lim_{k\to \infty}\int_{\Omega}g_s(x,\cdot)(z)\mathcal I_s\big[\xi_k^2\phi_x^{\mu};\;g_\gamma(y,\cdot) \big](z)dz=0.
\end{align}
To prove the claim, we write 
\begin{align}\label{final-apliting}
\mathcal I_s\big[\xi_k^2\phi_x^{\mu}; \;g_\gamma(y,\cdot) \big](z)&=c_{N,s}\int_{\R^N}\xi_k^2(p)\frac{\big(\phi_x^{\mu}(z)-\phi_x^{\mu}(p)\big)\big(g_\gamma(y,z)-g_\gamma(y,p)\big)}{|z-p|^{N+2s}}dp\nonumber\\
&\quad+\phi_x^{\mu}(z)\mathcal I_s\big[\xi_k^2;\;g_\gamma(y,\cdot) \big](z)\nonumber\\
&=:L^{k,y}_{\mu,\gamma}(x,z)+\phi_x^{\mu}(z)\mathcal I_s\big[\xi_k^2;\;g_\gamma(y,\cdot) \big](z).
\end{align}
There holds:
\begin{align}\label{C1-4}
\lim_{\gamma\to 0^+} \lim_{\mu\to 0^+}\lim_{k\to\infty}\int_{\Omega}g_s(x,\cdot)(z) L^{k,y}_{\mu,\gamma}(x,z)dz=0.
\end{align}
\underline{\textbf{Proof of \eqref{C1-4}}}:
By the Lebesgue dominated convergence theorem, we have 
$$
L^{k,y}_{\mu,\gamma}(x,z)\quad\to\quad \mathcal I_s\big[\phi_x^{\mu}, g_\gamma(y,\cdot)\big](z)
\qquad\text{as}\qquad k\to\infty,
$$
for every $z$ in $\R^N$, and therefore
$$
\lim_{\gamma\to 0^+}\lim_{\mu\to 0^+}\lim_{k\to\infty}\int_{\Omega}g_s(x,z) L^{k,y}_{\mu,\gamma}(x,z)dz
=\lim_{\gamma\to 0^+}\lim_{\mu\to 0^+}\int_{\Omega}g_s(x,z)\,\mathcal I_s\big[\phi_x^{\mu}, g_\gamma(y,\cdot)\big](z)\,dz.
$$
We set $\widetilde\mu:=\delta_\Omega(x)\mu/(2\sqrt{2})$.
We fix $\varepsilon>0$ such that $B_{2\varepsilon}(x)\subset \Omega$ and $y\notin B_{2\varepsilon}(x)$.
For $\mu>0$ small, the support of $\rho_\mu^x$ is contained in $B_{\varepsilon}(x)$, hence $\phi_x^{\mu}\equiv 1$ on $\Omega\setminus B_{\varepsilon}(x)$.

We split the integral into
$$
\int_{\Omega} g_s(x,z)\,\mathcal I_s\big[\phi_x^{\mu}, g_\gamma(y,\cdot)\big](z)\,dz
=\int_{B_{\varepsilon}(x)}\cdots+\int_{\Omega\setminus B_{\varepsilon}(x)}\cdots.
$$
If $z$ in $\Omega\setminus B_{\varepsilon}(x)$, then $\phi_x^{\mu}(z)=1$ and $\phi_x^{\mu}(p)=1$ unless $p$ in $\supp(\rho_\mu^x)\subset B_{\sqrt{2}\widetilde\mu}(x)$.
Thus
$$
\Big|\mathcal I_s\big[\phi_x^{\mu}, g_\gamma(y,\cdot)\big](z)\Big|
\leq C\int_{B_{\sqrt{2}\widetilde\mu}(x)}\frac{|g_\gamma(y,z)-g_\gamma(y,p)|}{|z-p|^{N+2s}}\,dp
\leq C\,\widetilde\mu^N,
$$
and therefore
$$
\int_{\Omega\setminus B_{\varepsilon}(x)}|g_s(x,z)|\,\Big|\mathcal I_s\big[\phi_x^{\mu}, g_\gamma(y,\cdot)\big](z)\Big|\,dz
\leq C\,\widetilde\mu^N\|g_s(x,\cdot)\|_{L^1(\Omega)}\to 0
\qquad\text{as}\qquad \mu\to 0^+.
$$
For the integral over $B_{\varepsilon}(x)$, we use that $x\neq y$ and $B_{2\varepsilon}(x)\subset \Omega$ imply that $g_\gamma(y,\cdot)$ is $C^1$ on $B_{2\varepsilon}(x)$, uniformly for $\gamma>0$ small.
In particular,
$$
|g_\gamma(y,p)-g_\gamma(y,q)|\leq C|p-q|,
\qquad p,q \;\text{in}\, B_{2\varepsilon}(x).
$$
For $z=x-\widetilde\mu p$ in $B_{\varepsilon}(x)$, we change variables in the definition of $\mathcal I_s$ and obtain
$$
\mathcal I_s\big[\phi_x^{\mu}, g_\gamma(y,\cdot)\big](x-\widetilde\mu p)
=c_{N,s}\widetilde\mu^{-2s}\int_{\R^N}\frac{(\rho(|q|^2)-\rho(|p|^2))\big(g_\gamma(y,x-\widetilde\mu p)-g_\gamma(y,x-\widetilde\mu q)\big)}{|p-q|^{N+2s}}\,dq.
$$
Hence
$$
\Big|\mathcal I_s\big[\phi_x^{\mu}, g_\gamma(y,\cdot)\big](x-\widetilde\mu p)\Big|
\leq C\,\widetilde\mu^{1-2s}\int_{\R^N}\frac{|\rho(|q|^2)-\rho(|p|^2)|}{|p-q|^{N+2s-1}}\,dq=:C\,\widetilde\mu^{1-2s}A(p).
$$
By the decomposition $g_s(x,\cdot)=f_s(x,\cdot)-h_s(x,\cdot)$ and the boundedness of $h_s(x,\cdot)$, we also have
$$
|g_s(x,x-\widetilde\mu p)|\leq C\,\widetilde\mu^{2s-N}|p|^{2s-N}+C.
$$
Therefore
$$
\int_{B_{\varepsilon}(x)}|g_s(x,z)|\,\Big|\mathcal I_s\big[\phi_x^{\mu}, g_\gamma(y,\cdot)\big](z)\Big|\,dz
\leq C\,\widetilde\mu \int_{\R^N}|p|^{2s-N}A(p)\,dp + C\,\widetilde\mu^{N+1-2s}\int_{\R^N}A(p)\,dp.
$$
Both integrals in the right-hand side are finite because $\rho$ is smooth, compactly supported, and equals $1$ near $0$.
We conclude that the right-hand side is $o(1)$ as $\mu\to 0^+$, uniformly in $\gamma>0$ small, which proves \eqref{C1-4}.
\par\;
On the other hand, by Remark \ref{remark-Append}, we know:
$$
\Big|\mathcal I_s\big[\xi_k^2;\;g_\gamma(y,\cdot) \big](z)\Big|\leq C\delta^{-s}_\Omega(z)\qquad\text{and}\qquad \lim_{k\to\infty}\mathcal I_s\big[\xi_k^2;\;g_\gamma(y,\cdot) \big](z)=0.
$$
And because $g_s(x,\cdot)\delta^{-s}_\Omega(\cdot)\in L^1(\Omega)$ (recall  \eqref{l1-deltas-reg}), we may use the Lebesgue dominated convergence theorem to conclude that
\begin{align}\label{Eq-pb}
\lim_{k\to\infty}\int_{\Omega}g_s(x,\cdot)(z)\mathcal I_s\big[\xi_k^2;\;g_\gamma(y,\cdot) \big](z)dz=0.
\end{align}
The claim \eqref{probleme} follows from \eqref{final-apliting}, \eqref{C1-4} and \eqref{Eq-pb}. In view of \eqref{C1-1}, \eqref{C1-2},  and \eqref{probleme}, to finish the proof of \eqref{C} it is sufficient to prove that
 \begin{align}  \label{eq-15-01}
    &\lim_{\gamma\to 0^+} \lim_{\mu\to 0^+}\lim_{k\to\infty}\int_{\Omega}g_s(x,\cdot)(z)g_\gamma(y,z) (-\Delta)^s\big( \xi_k^2\phi_x^{\mu}\big)(z)dz\\
    &=G^s_\Omega(x,y)\,b_Y[\rho](x).     \nonumber     
 \end{align}
 To see this, we use the product rule  \eqref{pl} once more to split: 
$$
(-\Delta)^s\big[ \xi_k^2\phi_x^{\mu}\big]=
\xi_k^2(-\Delta)^s\phi_x^{\mu}
+\phi_x^{\mu}(-\Delta)^s\xi_k^2
-\mathcal I_s[\xi_k^2,\phi_x^{\mu}].
$$
By Lemma \ref{lem-crucial} and Remark \ref{remark-Append}, we have the estimates:
$$
\Big|\phi_x^{\mu}(z)(-\Delta)^s\xi_k^2(z)-\mathcal I_s[\xi_k^2,\phi_x^{\mu}](z)\Big|\leq C\big[\delta^{-2s}_\Omega(z)+\delta^{-s}_\Omega(z)\big]\leq C\delta^{-2s}_\Omega(z).
$$
Next, since $\phi^y_\gamma$ vanishes identically near $y$, then in view of the standard estimate o the Green function (see e.g \eqref{green100} in the appendix), we have 
$$
\delta^{-s}_\Omega(\cdot)g_\gamma(y,\cdot)\in L^\infty(\Omega).
$$
Combining the above two estimates, we get 
$$
\Big|g_s(x,\cdot)g_\gamma(y,\cdot) \Big[\phi_x^{\mu}(-\Delta)^s\xi_k^2-\mathcal I_s[\xi_k^2,\phi_x^{\mu}]\Big]\Big|\leq C\delta^{-s}_\Omega(\cdot)g_s(x,\cdot)\in L^1(\Omega).
$$
Since 
$$
\phi_x^{\mu}(z)(-\Delta)^s\xi_k^2(z)-\mathcal I_s[\xi_k^2,\phi_x^{\mu}](z)\quad\to\quad 0\qquad\text{as}\qquad k\to \infty,
$$
we may apply the Lebesgue dominated convergence theorem to get that 
\begin{align}\label{C1-3}
&\lim_{k\to\infty}\int_{\Omega}g_s(x,\cdot)(z)g_\gamma(y,z) \Big[\phi_x^{\mu}(-\Delta)^s\xi_k^2-\mathcal I_s[\xi_k^2,\phi_x^{\mu}]\Big](z)dz=0.
\end{align}

Next, recalling the decomposition $g_s(x,\cdot)=f_s(x,\cdot)-h_s(x,\cdot)$ we write:
\begin{align}
\int_{\Omega}\xi_k^2(z)g_s(x,\cdot)(z)g_\gamma(y,z) (-\Delta)^s \phi_x^{\mu}(z)dz
&=\int_{\Omega}\xi_k^2(z)f_s(x,\cdot)(z)g_\gamma(y,z) (-\Delta)^s \phi_x^{\mu}(z)dz \nonumber\\
&-\int_{\Omega}\xi_k^2(z)h_s(x,\cdot)(z)g_\gamma(y,z) (-\Delta)^s \phi_x^{\mu}(z)dz. \label{inquiet}
\end{align}
\par\;

As already remarked, when computing the limits, the integrals above can be reduced to an arbitrary small neighborhood of $x$, i.e, 
\begin{align}
    &\lim_{\gamma\to 0^+} \lim_{\mu\to 0^+}\lim_{k\to\infty}\int_{\Omega}\xi_k^2(z)f_s(x,\cdot)(z)g_\gamma(y,z) (-\Delta)^s \phi_x^{\mu}(z)dz, \nonumber\\
    &=\lim_{\gamma\to 0^+} \lim_{\mu\to 0^+}\int_{\Omega}f_s(x,\cdot)(z)g_\gamma(y,z) (-\Delta)^s \phi_x^{\mu}(z)dz, \nonumber\\
    &=\lim_{\gamma\to 0^+} \lim_{\mu\to 0^+}\int_{B_\varepsilon(x)}f_s(x,z)g_\gamma(y,z) (-\Delta)^s \phi_x^{\mu}(z)dz,  \label{tart1}
\end{align}
and 
\begin{align}
    &\lim_{\gamma\to 0^+} \lim_{\mu\to 0^+}\lim_{k\to\infty}\int_{\Omega}\xi_k^2(z)h_s(x,z)g_\gamma(y,z) (-\Delta)^s \phi_x^{\mu}(z)dz\nonumber\\
    &=\lim_{\gamma\to 0^+} \lim_{\mu\to 0^+}\int_{\Omega}h_s(x,z)g_\gamma(y,z) (-\Delta)^s \phi_x^{\mu}(z)dz\nonumber\\
    &=\lim_{\gamma\to 0^+} \lim_{\mu\to 0^+}\int_{B_\varepsilon(x)}h_s(x,z)g_\gamma(y,z) (-\Delta)^s \phi_x^{\mu}(z)dz, \label{tart2}
\end{align}
for all $\varepsilon>0$. Next using \eqref{tart1}
and changing variables, we get
\begin{align*}
    &\lim_{\gamma\to 0^+}\lim_{\mu\to 0^+}\int_{\Omega}f_s(x,z)g_\gamma(y,z) (-\Delta)^s \phi_x^{\mu}(z)dz\nonumber\\
    &=-\lim_{\gamma\to 0^+}\lim_{\mu\to 0^+}\int_{B_{\frac{\varepsilon}{\widetilde\mu}}(0)}\widetilde\mu^{N-2s}f_s(x,x-\widetilde{\mu}z)g_\gamma(y,x-\widetilde\mu z) (-\Delta)^s\big(\rho\circ |\cdot|^2\big)(z)dz ,
\end{align*}
where we recall that $g_\gamma(y,\cdot)\in L^\infty(\Omega)$. Now since 

$$
\Big|\widetilde\mu^{N-2s}f_s(x,x-\widetilde{\mu}z)g_\gamma(y,x-\widetilde\mu z)\Big|\leq C|z|^{2s-N},
$$
and
$$
\int_{\R^N}|z|^{2s-N}\big|(-\Delta)^s\big(\rho\circ |\cdot|^2\big)(z)\big|dz<+\infty,
$$
it follows from the Lebesgue dominated convergence theorem that 
\begin{align}
	&=-\lim_{\gamma\to 0^+}\lim_{\mu\to 0^+}\int_{B_{\frac{\varepsilon}{\widetilde\mu}}(0)}\widetilde\mu^{N-2s} f_s(x,x-\widetilde\mu z)g_\gamma(y,x-\widetilde\mu z)(-\Delta)^s(\rho\circ |\cdot|^2)(z)dz     \nonumber\\
	&=G^s_\Omega(x,y)\,b_Y[\rho](x). \label{star}
 \end{align}
where we used that 
$$
\widetilde\mu^{N-2s}f_s(x,x-\widetilde{\mu}z)\quad\to\quad (2s-N)b_{N,s}\frac{[DY(x)\cdot z]\cdot z}{|z|^{N-2s+2}}\qquad\text{as}\qquad\mu\to 0^+,
$$
and 
$$
\lim_{\gamma\to 0^+}\lim_{\mu\to 0^+}g_\gamma(y,x-\widetilde{\mu}z)=\lim_{\gamma\to 0^+}\lim_{\mu\to 0^+}G^s_\Omega(y,x-\widetilde\mu z)\phi_y^{\gamma}(x-\widetilde\mu z)= G^s_\Omega(y,x).
$$
Proceeding as above and using $h_s(x,\cdot)\in C(\Omega)\cap L^\infty(\Omega)$ (see e.g Lemma \ref{diff-reg-part}), we easily see that
 \begin{align}\label{star2}
     \lim_{\gamma\to 0^+} \lim_{\mu\to 0^+}\lim_{k\to\infty}&\int_{B_\varepsilon(x)}\xi_k^2(z)h_s(x,\cdot)(z)g_\gamma(y,\cdot)(z) (-\Delta)^s \phi_x^{\mu}(z)dz= 0   .
 \end{align}
By combining  \eqref{inquiet}, \eqref{tart1}, \eqref{tart2}, \eqref{star} and \eqref{star2}, we conclude that 
 \begin{align} 
    &\lim_{\gamma\to 0^+} \lim_{\mu\to 0^+}\lim_{k\to\infty}
\int_{\Omega}\xi_k^2(z)g_s(x,\cdot)(z)g_\gamma(y,\cdot)(z) (-\Delta)^s \phi_x^{\mu}(z)dz \nonumber\\
 &=G^s_\Omega(x,y)\,b_Y[\rho](x). \label{tart5}
 \end{align}
Combining  \eqref{C1-3} and \eqref{tart5}, we 
 obtain     \eqref{eq-15-01} and therefore 
the proof of \eqref{C}  is finished.
\end{proof}
\section{Derivation of the formula \eqref{var-green} and Proof of Theorem \ref{THM-1.5}}\label{Sec-proof-main-result}
 Here we complete the proof of the Theorem \ref{THM-1.5}. The first part of the theorem follows from Lemma \ref{diff-reg-part} and the decomposition \eqref{Eq-apliting-of-Green}. Indeed, by Lemma \ref{diff-reg-part}, we know that the function $t\mapsto H_{\Omega_t}(\Phi_t(x), \Phi_t(y))$ is differentiable at zero for all $x,y\in \Omega$. It is also plain that the function $t\mapsto|\Phi_t(x)-\Phi_t(y)|^{2s-N}$ is differentiable for all $x\neq y$. Therefore, since 
 $$
 G^s_{\Omega_t}(\Phi_t(x),\Phi_t(y))=b_{N,s}|\Phi_t(x)-\Phi_t(y)|^{2s-N}-H^s_{\Omega_t}(\Phi_t(x),\Phi_t(y)),
 $$
 we conclude that the function  $t\mapsto G^s_{\Omega_t}(\Phi_t(x),\Phi_t(y))$ is differentiable at zero for all $x\neq y$. By standard regularity result, we also know that $G_\Omega^s(\cdot,\cdot)$ is smooth off the diagonal. Hence the directional shape derivative as defined in \eqref{def-direc-shape-deriv} is well defined. To finish the proof of the theorem it therefore remains to establish the formula \eqref{var-green} which we restate in the lemma below for the reader convenience
 \begin{Lemma}\label{babyformula}
Let $x,y$ in $\Omega$ and let $g_s(x,y)$ be given as in \eqref{gsxy}. Then, the following formula holds. 
\begin{align}
&g_s(x,y)-\nabla_x G^s_\Omega(x,y)\cdot Y(x)-\nabla_y G^s_\Omega(x,y)\cdot Y(y)\nonumber\\
&= \Gamma^2(1+s)\int_{\partial\Omega}\gamma_0^s(G_\Omega^s(x,\cdot))\gamma_0^s(G^s_\Omega(y,\cdot))Y\cdot\nu d\sigma.
\end{align}     
 \end{Lemma}
%
 %%%
 %%%

%
 \begin{proof}
Let $x,y$ in $\Omega$ such that $x\neq y$ and defined  
\begin{equation}
W^{x}_y(\cdot):=\xi_k^2g_\gamma(y,\cdot)\phi_x^\mu:=\xi_k^2(\cdot)G^s_\Omega(y,\cdot)\phi_y^{\gamma}(\cdot) \phi_x^{\mu}(\cdot)  \qquad\text{for $\mu,\gamma\in (0,1)$,}
\end{equation}
 where $\xi_k$ and $\phi_z^{\mu}$ are defined as in \eqref{eta-k} and  \eqref{psi-mu}. Clearly, we have $W^x_{y}\in C^\infty_0(\Omega)$, and thus it is admissible as a test function into \eqref{eq-ahap-green}. Hence
\begin{align}\label{Key-idd-02}
 \int_{\Omega}g_s(x,\cdot)(-\Delta)^s(W^{x}_y)dz+ E_Y\big(G_\Omega^s(x,\cdot), W^x_{y}\big)=0.
\end{align}
Note that since $ W^x_{y}\in C^\infty_0(\Omega)$, and $g_s(x,\cdot)\in L^1(\Omega)$ --see e.g \eqref{l1-deltas-reg}-- then by Lemma \ref{FundLemma}, the quantities in the identity above are well defined.
\par\;

For a real valued function $u$, and two positions $p$ and $q$, set $[ u] := u(p) - u(q)$. Then we observe that for any triple of functions  $u,v, w$, 
\begin{align}\label{64}
  [uv]    [w] =  [u]    [vw] +  [v]  \Big(u(q)w(p)-u(p)w(q)\Big).
  \end{align}
To establish this, it is only a matter of developing on both sides. The left-hand side then contains four terms, the right-hand side contains twice as many, but we may observe that two terms with different signs of $u(q)v(p)w(p)$ cancel out as well as two terms with different signs of $u(p)v(q)w(q)$. The remaining terms match, hence the identity above. 
Therefore
 \begin{equation}
     E_Y\big(uv,w\big)= E_Y\big(u,vw\big)+\pv\iint_{\R^{2N}}\Big(v(p)-v(q)\Big)\Big(u(q)w(p)-u(p)w(q)\Big)\kappa_Y(p,q)\, dp\, dq  .
\end{equation}
Applying the identity above with
\begin{align*}
w=g_0(x,\cdot), \quad v=\xi_k\phi_x^{\mu}\quad \text{and} \quad u=\xi_kg_\gamma(y,\cdot): 
\end{align*}
gives
\begin{align}
 E_Y\big(g_0(x,\cdot), W^x_{y}\big)= E_Y\big(g_0(x,\cdot), \xi_k^2 g_\gamma(y,\cdot)\phi_x^{\mu}\big)= E_Y(\xi_kg_\mu(x,\cdot),\; \xi_k g_\gamma(y,\cdot)\Big) + R^k_{\mu,\gamma}(x,y),
\end{align}
where the error term $R^k_{\mu,\gamma}$ is given by 
\begin{align*}
 R^k_{\mu,\gamma}(x,y)=\pv\iint_{\R^{2N}}\Big(v(p)-v(q)\Big)\Big(u(q)w(p)-u(p)w(q)\Big)\kappa_Y(p,q)\, dp\, dq, 
\end{align*}
with $u,v,w$ defined above. Since $\xi_kg_\mu(x,\cdot)$ and $ \xi_k g_\gamma(y,\cdot)$ are in $C^\infty_0(\Omega)$, we may apply \cite[Lemma 2.1]{DFW2023} with two different coordinates to get: 
\begin{align}
E_Y\big(\xi_kg_\mu(x,\cdot), \xi_k g_\gamma(y,\cdot)\big)
&=-\int_{\Omega}Y\cdot\nabla(\xi_kg_\mu(x,\cdot))(-\Delta)^s(\xi_k g_\gamma(y,\cdot))dz\nonumber\\
&-\int_{\Omega}Y\cdot\nabla\big(\xi_k g_\gamma(y,\cdot)\big)(-\Delta)^s(\xi_kg_\mu(x,\cdot))dz,\nonumber\\
&:=-A^{k}_{\mu,\gamma}(x,y)- B^{k}_{\mu,\gamma}(x,y).
\end{align}

Applying the product rule  \eqref{pl} we decompose again each of the two terms above into two terms:
\begin{align}
A^k_{\mu,\gamma}(x,y)&=    \int_{\Omega}Y\cdot\nabla(\xi_kg_\mu(x,\cdot))\Big\{g_\gamma(y,\cdot)(-\Delta)^s\xi_k-\mathcal I_s\big[\xi_k,g_\gamma(y,\cdot)\big]\Big\}dz\nonumber\\
&+\int_{\Omega}\xi_kY\cdot\nabla(\xi_kg_\mu(x,\cdot))(-\Delta)^s g_\gamma(y,\cdot)dz,\nonumber\\
&:=A^{k,1}_{\mu,\gamma}(x,y)+ A^{k,2}_{\mu,\gamma}(x,y),
\end{align}
and 
\begin{align}
B^k_{\mu,\gamma}(x,y)&=    \int_{\Omega}Y\cdot\nabla(\xi_kg_\gamma(y,\cdot))\Big\{g_\mu(x,\cdot)(-\Delta)^s\xi_k-\mathcal I_s\big[\xi_k,g_\mu(x,\cdot)\big]\Big\}dz\nonumber\\
&+\int_{\Omega}\xi_kY\cdot\nabla(\xi_kg_\gamma(y,\cdot))(-\Delta)^s g_\mu(x,\cdot)dz, \nonumber\\
&:= B^{k,1}_{\mu,\gamma}(x,y)+ B^{k,2}_{\mu,\gamma}(x,y).
\end{align}
With these notations, the identity \eqref{Key-idd-02} becomes
\begin{align}\label{Key-idd-03}
& C^k_{\mu,\gamma}(x,y)- A^{k,2}_{\mu,\gamma}(x,y)- B^{k,2}_{\mu,\gamma}(x,y)+R^k_{\mu,\gamma}(x,y)= A^{k,1}_{\mu,\gamma}(x,y)+ B^{k,1}_{\mu,\gamma}(x,y), 
\end{align}
with 
\begin{align}\label{def-C}
C^{k}_{\mu,\gamma}(x,y) :=  \int_{\Omega}g_s(x,z)(-\Delta)^s\big(\xi_k^2 g_\gamma(y,\cdot)\phi_x^{\mu}\big)(z)dz.
\end{align}
First using \cite[Proposition 2.2]{DFW} and then passing to the limit as $\mu\to 0^+$ and then as $\gamma\to 0^+$, we get
\begin{align}
    \lim_{\gamma\to 0^+} \lim_{\mu\to 0^+}\lim_{k\to\infty}A^{k,1}_{\mu,\gamma}(x,y)=\frac{\Gamma^2(1+s)}{2}\int_{\partial\Omega}\gamma_0^s(G_\Omega^s(x,\cdot))\gamma_0^s(G^s_\Omega(y,\cdot))Y\cdot\nu d\sigma,\label{lim-A1}\\
    \lim_{\gamma\to 0^+} \lim_{\mu\to 0^+}\lim_{k\to\infty} B^{k,1}_{\mu,\gamma}(x,y)=\frac{\Gamma^2(1+s)}{2}\int_{\partial\Omega}\gamma_0^s(G_\Omega^s(x,\cdot))\gamma_0^s(G^s_\Omega(y,\cdot))Y\cdot\nu d\sigma, \label{lim-B1} 
\end{align}
where $\nu$ denotes the outward unit normal.  Moreover, a similar argument as in \cite[Lemma 2.4]{Sidy-Franck-BP} yields 
\begin{align}
   \lim_{\gamma\to 0^+} \lim_{\mu\to 0^+}\lim_{k\to\infty} A^{k,2}_{\mu,\gamma}(x,y)=\nabla_y G^s_\Omega(x,y)\cdot Y(y),\label{lim-A2}\\
   \lim_{\gamma\to 0^+} \lim_{\mu\to 0^+}\lim_{k\to\infty} B^{k,2}_{\mu,\gamma}(x,y)=\nabla_x G^s_\Omega(x,y)\cdot Y(x).\label{lim-B2}
\end{align}
On the other hand, appealing to Lemma \ref{Lem.sec-4.3} and Lemma \ref{main-res-sec4.4}, and recalling Lemma \ref{appen}. we also have:
\begin{align}\label{toproove}
&\lim_{\gamma\to 0^+}\lim_{\mu\to 0^+}\lim_{k\to\infty} \Big[ C^k_{\mu,\gamma}(x,y)+ R^k_{\mu,\gamma}(x,y)    \Big],\nonumber\\
&= G^s_\Omega(y,x)b_{N,s}\iint_{\R^{2N}}\frac{\big(\rho(|y|^2)-\rho(|z|^2)\big)\big(|y|^{2s-N}-|z|^{2s-N}\big)}{|y-z|^{N+2s}} \omega^x_{Y}(y,z)\, dydz,\nonumber\\
&+g_s(x,y)+b_{N,s}(N-2s)G^s_\Omega(x,y)\int_{\R^N}\frac{[DY(x)\cdot z]\cdot z}{|z|^{N-2s+2}}(-\Delta)^s(\rho\circ |\cdot|^2)(z)dz,\nonumber\\
&=-G^s_\Omega(y,x)b_{N,s}(N-2s)\int_{\R^N}\frac{[DY(x)\cdot z]\cdot z}{|z|^{N-2s+2}}(-\Delta)^s(\rho\circ|\cdot|^2)(z)dz,\nonumber\\
&+g_s(x,y)+b_{N,s}(N-2s)G^s_\Omega(x,y)\int_{\R^N}\frac{[DY(x)\cdot z]\cdot z}{|z|^{N-2s+2}}(-\Delta)^s(\rho\circ |\cdot|^2)(z)dz,\nonumber\\
&=g_s(x,y).
\end{align}
We conclude by combining \eqref{Key-idd-03}, \eqref{def-C}, \eqref{lim-A1}, \eqref{lim-B1}, \eqref{lim-A2}, \eqref{lim-B2}, and \eqref{toproove}. The proof is finished. 
 \end{proof}
\section{Proof of Corollary \ref{cor6}}
\label{sec-ftt}
We first note that \eqref{var-Robin-func} is a simple consequence of \eqref{var-green}. Indeed, observe that 
\begin{equation}\label{remark}
\partial_{t}(|x_t-y_t|^{2s-N}){\Big|_{t=0}}-\nabla_x(|x-y|^{2s-N})\cdot Y(x)-\nabla_y(|x-y|^{2s-N})\cdot Y(y) = 0.
\end{equation}
Therefore, by writing  
$$ G^s_{\Omega_t}(x_t,y_t)  = b_{N,s}|x_t-y_t|^{2s-N}-H^s_{\Omega_t}(x_t, y_t)$$
and recalling \eqref{remark}, we reduce the identity \eqref{var-green} to: 
\begin{align*}
  &-\partial_tH^s_{\Omega_t
  }(x_t,y_t)\big|_{t=0}+\nabla_xH^s_\Omega(x,y)\cdot Y(x)+\nabla_yH^s_\Omega(x,y)\cdot Y(y)\nonumber\\
  &= \Gamma^2(1+s)\int_{\partial\Omega}\gamma_0^s(G_\Omega^s(x,\cdot))\gamma_0^s(G^s_\Omega(y,\cdot))Y\cdot\nu \;d\sigma.
\end{align*}
Now that we got rid of the singularity, we can take $x= y$ in the identity above to derive \eqref{var-Robin-func}.
%thanks to the chain rule
%$$  \partial_t R^s_{\Omega_t}(x){\big|_{t=0}} =\partial_t R^s_{\Omega_t}(x_t)\big|_{t=0}
%A_{t}^x{\big|_{t=0}}\Big)(y)
%-\nabla_x H^s_\Omega (x,x)\cdot Y(x)-\nabla_z H^s_\Omega(x,x)\cdot Y(x) .
%$$
%%
%which in view of the identity above gives the result.
\section{Proof of Theorem \ref{THM-1.2}}\label{Proof-THM-1.2}
We start by proving two preliminary lemmas that will be essential for the proof. We begin with the following which is a consequence of \cite[Lemma 3.1.3]{Abatangelo}.
\begin{Lemma}\label{Ab}Let $f$ in $C^0(\partial\Omega)$ and $\psi\in L^\infty(\Omega)$. Then we have  
$$
\int_{\Omega}\Big(\int_{\partial\Omega}\gamma^s_0(G_\Omega^s(x,\cdot))(\sigma)f(\sigma)d\sigma\Big)\psi(x)dx = \int_{\partial\Omega}f(\sigma)\gamma_0^s(\phi)(\sigma)d\sigma,
$$
where we set 
$$
\gamma_0^s(w):=w/\delta^s \qquad\text{and}\qquad\phi(x)=\int_{\Omega}G^s_\Omega(x,y)\psi(y)dy.
$$
\end{Lemma}
\begin{proof}We note that the integrals above are all well defined. Indeed, it is a classical result in potential theory that $\int_{\partial\Omega}(G_\Omega^s(x,\cdot)/\delta^s)(\sigma)d\sigma\leq C\delta^{s-1}(x)$ and clearly $\int_{\Omega}\delta^{s-1}(x)\psi(x)dx<\infty$. We also know that $\phi/\delta^s\in L^\infty(\overline{\Omega})$ by boundary regularity (see e.g \cite{RS}) and therefore the RHS of the identity is also well defined.
\par\;
We now proceed with the proof of the identity. Set $\psi_k:=\psi\xi_k$. Then $\psi_k$ is bounded and compactly supported in $\Omega$. We apply \cite[Lemma 3.1.3]{Abatangelo} with $\psi=\psi_k$ to get 
\begin{align}\label{Approx-Id}
\int_{\Omega}\Big(\int_{\partial\Omega}\gamma^s_0(G_\Omega^s(x,\cdot))(\sigma)f(\sigma)d\sigma\Big)\psi_k(x)dx = \int_{\partial\Omega}f(\sigma)\gamma_0^s(\phi_k)(\sigma)d\sigma,    
\end{align}
with $$
\phi_k(x)=\int_{\Omega}G^s_\Omega(x,y)\psi_k(y)dy.
$$
We note that although the statement of \cite[Lemma 3.1.3]{Abatangelo} assumes $\psi\in C^\infty_0(\Omega)$, only the fact that $\psi$ \textit{has a compact support in $\Omega$} is relevant for the proof. We can therefore apply it with $\psi_k$. We now pass to the limit as $k$ tends to infinity in \eqref{Approx-Id}. On the other hand, we have by the Lebesgue dominated convergence theorem that 
\begin{align}\label{L1}
\int_{\Omega}\Big(\int_{\partial\Omega}\gamma^s_0(G_\Omega^s(x,\cdot))(\sigma)f(\sigma)d\sigma\Big)\psi_k(x)dx\to\int_{\Omega}\Big(\int_{\partial\Omega}\gamma^s_0(G_\Omega^s(x,\cdot))(\sigma)f(\sigma)d\sigma\Big)\psi(x)dx.    
\end{align}
On the other hand, because 
$$
(-\Delta)^s(\phi_k-\phi) = \psi_k-\psi = -\rho_k\psi\quad\text{in $\Omega$}\qquad\text{and}\qquad \phi_k-\phi\equiv 0\quad\text{in}\quad \R^N\setminus\Omega,
$$
then by  the standard boundary regularity theory (see e.g \cite[Corollary 1.2]{Fall}), we have, for any $p>N/s$, that
\begin{align*}
\|(\phi_k-\phi)/\delta^s\|_{C^0(\overline{\Omega})}&\leq C\big(\|\phi_k-\phi\|_{L^2(\Omega)}+\|\rho_k\psi\|_{L^p(\Omega)}\big)\\
&\leq C\big(\|\rho_k\psi\|_{L^2(\Omega)}+\|\rho_k\psi\|_{L^p(\Omega)}\big)\to 0\quad \text{as $k\to\infty$},
\end{align*}
where in the last line we used the classical fact that the $L^2$-norm of the solution is controlled, up to a constant, by the $L^2$-norm of the right-hand side. It therefore follows that 
\begin{align}\label{L2}
\int_{\partial\Omega}f(\sigma)\gamma_0^s(\phi_k)(\sigma)d\sigma\to\int_{\partial\Omega}f(\sigma)\gamma_0^s(\phi)(\sigma)d\sigma.      
\end{align}
Combining \eqref{Approx-Id}, \eqref{L1} and \eqref{L2} we get the result and the proof is finished.
\end{proof}
\begin{Lemma}\label{reduc-FY}
Let $h$ be a globally Lipschitz continuous function in $\Omega$ and $u$ be the classical solution of the equation

\begin{equation}
\left\{ \begin{array}{rcll} (-\Delta)^s u&=& h  &\textrm{in }\;\;\Omega \\ u&=&0&
\textrm{in }\;\;\R^N\setminus\Omega , \end{array}\right. 
\end{equation}
Let $Y$ be a globally Lipschitz vector field in $\Omega$. Then for all $x\in \Omega$, there holds
\begin{equation}\label{Last-Claim}
\int_{\Omega}G_\Omega^s(x,y)\div(h Y)dy+
\int_{\Omega}\Bigg(\nabla_y G^s_\Omega(x,y)\cdot Y(y)+\nabla_x G^s_\Omega(x,y)\cdot Y(x)\Bigg)h(y)dy=\nabla u\cdot Y(x).
\end{equation}
\end{Lemma}
\begin{proof}
We note that the LHS of the identity above is well defined. Indeed, the first integral is clearly finite. And also, since $Y$ is globally Lipschitz continuous, we have that $\nabla_y G_\Omega^s(x,\cdot)\cdot Y(\cdot)+\nabla_x G_\Omega^s(x,\cdot)\cdot Y(x)\in L^1(\Omega)$. This follows since 
\begin{align*}
&\nabla_y G^s_\Omega(x,\cdot)\cdot Y(\cdot)+\nabla_x G^s_\Omega(x,\cdot)\cdot Y(x)\nonumber\\
&= (2s-N)\frac{(x-\cdot)\cdot(Y(x)-Y(\cdot))}{|x-\cdot|^{N-2s+2}}-\nabla_y H^s_\Omega(x,\cdot)\cdot Y(\cdot)-\nabla_x H^s_\Omega(x,\cdot)\cdot Y(x)\in L^1(\Omega).
\end{align*}  
Since $h$ is bounded, it follows that the second integral is also finite.
\par\;
We now proceed with the proof. We distinguish two different cases depending on whether $2s>1$ or $2s\leq 1$.
\par\;

\textbf{Case 1}: $2s\leq 1$. 
By \cite[Theorem 1.6]{Sidy-Franck-BP}, we have, for all $j=1,2,\cdots N$, that
\begin{equation*}
 \partial_{x_i} G^s_\Omega(x,y)+\partial_{y_i} G^s_\Omega(x,y)=  -\Gamma^2(1+s)\int_{\partial \Omega} 
 \gamma_0^s(G^s_\Omega(x,\cdot)) \gamma_0^s(G^s_\Omega(y,\cdot))
 \nu_i \, d\sigma.    
\end{equation*}
Consequently,
\begin{align*}
\nabla_x G_\Omega^s(x,y)\cdot Y(x) = -\nabla_y G_\Omega^s(x,y)\cdot Y(x)-\Gamma^2(1+s)\int_{\partial \Omega} 
\gamma_0^s(G^s_\Omega(x,\cdot)) \gamma_0^s(G^s_\Omega(y,\cdot))Y(x)\cdot\nu \, d\sigma.   
\end{align*}
It follows, in view of Lemma \ref{Ab}, that
\begin{align}\label{Tr-term}
&\int_{\Omega}\Bigg(\nabla_y G^s_\Omega(x,y)\cdot Y(y)+\nabla_x G^s_\Omega(x,y)\cdot Y(x)\Bigg)h(y)dy\nonumber\\
&=-\Gamma^2(1+s)\int_{\Omega}\Bigg(\int_{\partial \Omega} 
\gamma_0^s(G^s_\Omega(x,\cdot)) \gamma_0^s(G^s_\Omega(y,\cdot))Y(x)\cdot\nu \, d\sigma   \Bigg)h(y)dy\nonumber\\
&+\int_{\Omega}\nabla_y G_\Omega^s(x,y)\cdot (Y(y)-Y(x))h(y)dy\nonumber\\
&=-\Gamma^2(1+s)\int_{\partial\Omega}\gamma_0^s(G^s_\Omega(x,\cdot)) \gamma_0^s(u)Y(x)\cdot\nu \, d\sigma+\int_{\Omega}\nabla_y G_\Omega^s(x,y)\cdot (Y(y)-Y(x))h(y)dy.  
\end{align}
It is not difficult to see that 
\begin{align}\label{Fr-D}
\int_{\Omega}\nabla_y G_\Omega^s(x,y)\cdot (Y(y)-Y(x))h(y)dy
%&= \lim_{\varepsilon\to 0^+}\int_{\Omega\cap\{|x-y|>\varepsilon\}}\nabla_y G_\Omega^s(x,y)\cdot (Y(y)-Y(x))h(y)dy. \nonumber\\
=-\int_{\Omega}G_\Omega^s(x,y)\big[ \div(Yh)-\nabla h(y)\cdot Y(x)\big]dy.
\end{align}
Indeed, for every $\varepsilon>0$, we have by integration by parts 
\begin{align}\label{FC2}
\int_{\Omega\setminus  B_{\varepsilon}(x)}\nabla_y G_\Omega^s(x,y)\cdot (Y(y)-Y(x))&h(y)dy
= \int_{\partial B_\varepsilon(x)}G_\Omega^s(x,y)\,\frac{(Y(y)-Y(x))\cdot(x-y)}{|x-y|}\,h(y)\,d\sigma(y)\nonumber\\
&-\int_{\Omega\setminus  B_{\varepsilon}(x)}G_\Omega^s(x,y)\big[ \div(Yh)-\nabla h(y)\cdot Y(x)\big]dy.
\end{align}
It is clear that
\begin{align}\label{FC3}
&\lim_{\varepsilon\to 0^+}\int_{\Omega\setminus  B_{\varepsilon}(x)}G_\Omega^s(x,y)\big[ \div(Yh)-\nabla h(y)\cdot Y(x)\big]dy
=\int_{\Omega}G_\Omega^s(x,y)\big[ \div(Yh)-\nabla h(y)\cdot Y(x)\big]dy.
\end{align}
In the other hand, we also have
\begin{align}\label{FC4}
&\Big|\int_{\partial B_\varepsilon(x))}G_\Omega^s(x,y)\,(Y(y)-Y(x))\cdot(x-y)|x-y|^{-1}\,h(y)\,d\sigma(y)\Big| \nonumber\\
&\leq C\int_{\partial B_\varepsilon(x)}|x-y|^{1+2s-N}d\sigma(y)=C\varepsilon^{1+2s-N}|\partial B_\varepsilon|\leq C\varepsilon^{2s}\to 0\quad\text{as $\varepsilon\to 0^+$.}
\end{align}
Combining \eqref{FC2}, \eqref{FC3} and \eqref{FC4} we get \eqref{Fr-D}. Next, plug \eqref{Fr-D}  into \eqref{Tr-term} to finally obtain
\begin{align*}
&\int_{\Omega}\Bigg(\nabla_y G^s_\Omega(x,y)\cdot Y(y)+\nabla_x G^s_\Omega(x,y)\cdot Y(x)\Bigg)h(y)dy\\
&=-\int_{\Omega}G_\Omega^s(x,y)\cdot \div(hY)dy-\Gamma^2(1+s)\int_{\partial\Omega}\gamma_0^s(G^s_\Omega(x,\cdot)) \gamma_0^s(u)Y(x)\cdot\nu \, d\sigma\\
&+\int_{\Omega}G_\Omega^s(x,y)Y(x)\cdot\nabla h(y)dy.
\end{align*}
In other words, 
\begin{align*}
&\int_{\Omega}G_\Omega^s(x,y)\cdot \div(hY)dy+\int_{\Omega}\Bigg(\nabla_y G^s_\Omega(x,y)\cdot Y(y)+\nabla_x G^s_\Omega(x,y)\cdot Y(x)\Bigg)h(y)dy\nonumber\\
&=-\Gamma^2(1+s)\int_{\partial\Omega}\gamma_0^s(G^s_\Omega(x,\cdot)) \gamma_0^s(u)Y(x)\cdot\nu \, d\sigma+\int_{\Omega}G_\Omega^s(x,y)Y(x)\cdot\nabla h(y)dy\\
&=\nabla u\cdot Y(x)
\end{align*}
where in the last line we used  \cite[Eq (8), Theorem 1.1]{Sidy-Franck-BP}. This proves the claim in the case $2s\leq 1$.\par\;

\textbf{Case 1}: $2s>1$. We first integrate by parts and write
\begin{align*}
\int_{\Omega}G_\Omega^s(x,y)\div(hY)dy=-\int_{\Omega}\nabla_yG_\Omega^s(x,y)\cdot Y(y)h(y)dy.  
\end{align*}
Note that the RHS of the identity above is well defined since the gradient of the Green function is integrable for $2s>1$. In view of this, the proof of the claim reduces into proving that
\begin{equation}\label{CL-daughter}
\nabla u\cdot Y(x)=\int_{\Omega}\nabla_x G^s_\Omega(x,y)\cdot Y(x)h(y)dy.   
\end{equation}
The identity \eqref{CL-daughter} is an easy consequence of \cite[Theorem 1.1]{Sidy-Franck-BP}, \cite[Theorem 1.6]{Sidy-Franck-BP} and Lemma \ref{Ab}
\par\;
Indeed, by \cite[Theorem 1.1]{Sidy-Franck-BP} we have 
\begin{align}\label{nablau.Y}
\nabla u\cdot Y(x)= -\Gamma^2(1+s)\int_{\partial \Omega} \gamma_0^s(u)\gamma_0^s(G^s_\Omega(x,\cdot)) Y(x)\cdot\nu\,d\sigma -\int_\Omega \nabla_y G^s_\Omega(x,y)\cdot Y(x)h(y)\,dy.
\end{align}
in the other hand, by \cite[Theorem 1.6]{Sidy-Franck-BP} we also have 
\begin{equation*}
 \nabla_{x} G^s_\Omega(x,y)\cdot Y(x)=-\nabla_{y} G^s_\Omega(x,y)\cdot Y(x)  -\Gamma^2(1+s)\int_{\partial \Omega} 
 \gamma_0^s(G^s_\Omega(x,\cdot)) \gamma_0^s(G^s_\Omega(y,\cdot))
 Y(x)\cdot\nu \, d\sigma.    
\end{equation*}
Multiply the latter by $h(y)$ and integrate over $\Omega$ with respect to $y$ and use Lemma \ref{Ab} (with $\psi=h$ and $f=f_x:=\gamma_0^s(G_\Omega^s(x,\cdot))Y(x)\cdot\nu\in C^0(\partial\Omega)$)  to get 
\begin{align}\label{scal-nablaGr.Y-h}
&\int_{\Omega}\nabla_x G^s_\Omega(y,x)\cdot Y(x)h(y)dy\nonumber\\
&=-\Gamma^2(1+s)\int_{\Omega}\Bigg(\int_{\partial \Omega} 
 \gamma_0^s(G^s_\Omega(x,\cdot)) \gamma_0^s(G^s_\Omega(y,\cdot))
 Y(x)\cdot\nu \, d\sigma\Bigg)h(y)dy-\int_{\Omega}\nabla_{y} G^s_\Omega(x,y)\cdot Y(x)  h(y)dy\nonumber\\
 &=-\Gamma^2(1+s)\int_{\partial \Omega} \gamma_0^s(u)\gamma_0^s(G^s_\Omega(x,\cdot)) Y(x)\cdot\nu\,d\sigma-\int_{\Omega}\nabla_{y} G^s_\Omega(x,y)\cdot Y(x)  h(y)dy.
\end{align}
Comparing \eqref{nablau.Y} and \eqref{scal-nablaGr.Y-h} we get \eqref{CL-daughter} which is what we want to prove.
\end{proof}
We are now ready to complete the proof of Theorem \ref{THM-1.2}.
\begin{proof}[Proof of Theorem \ref{THM-1.2} (Completed)]
Let $u_t$ be the unique weak solution of 
\begin{equation}\label{Feq}
\left\{ \begin{array}{rcll} (-\Delta)^s u_t&=& h  &\textrm{in }\;\;\Omega_t \\ u_t&=&0&
\textrm{in }\;\;\R^N\setminus\Omega_t , \end{array}\right. 
\end{equation}
It is a classical result in potential theory that for all $x\in \Omega_t$, we have
$$
u_t(x) = \int_{\Omega_t}G_{\Omega_t}^s(x,y)h(y)dy.
$$
It follows that, for $x\in \Omega$, there holds:  
\begin{align}\label{vt}
u_t\circ \Phi_t(x)=\int_{\Omega_t}G_{\Omega_t}^s(\Phi_t(x),y)h(y)dy=\int_{\Omega}g_t(x_t,y_t)h(\Phi_t(y))\textrm{Jac}_{\Phi_t}(y)dy.
\end{align}
where, as above, we set
$$
g_t(x_t,y_t):=G_{\Omega_t}^s(\Phi_t(x),\Phi_t(y))\quad\text{and}\quad
z_t := \Phi_t(z)\qquad\text{for $z\in \R^N$}.
$$
By \eqref{TGH}, we know that 
\begin{align*}
t^{-1}\big|g_t(x_t,y_t)-g_0(x,y)\big|\leq C(1+|x-y|^{2s-N})\in L^1(\Omega),
\end{align*}
\par\;
In view of this, we may differentiate \eqref{vt} to get 
\begin{align*}
v'(x) := \partial_t(u_t\circ\Phi_t)\big|_{t=0}(x)&=\int_{\Omega}g_s(x,y)h(y)dy+\int_{\Omega}G_\Omega^s(x,y)\big[\nabla h\cdot Y+h\div Y\big]dy\nonumber\\
&=\int_{\Omega}g_s(x,y)h(y)dy+\int_{\Omega}G_\Omega^s(x,y)\div(h Y)dy.
\end{align*}
By Lemma \ref{babyformula}, we know that
\begin{align*}
g_s(x,y)
&= \nabla_y G^s_\Omega(x,y)\cdot Y(y)+\nabla_x G^s_\Omega(x,y)\cdot Y(x)\\
&+\Gamma^2(1+s)\int_{\partial\Omega}\gamma_0^s(G_\Omega^s(x,\cdot))\gamma_0^s(G^s_\Omega(y,\cdot))Y\cdot\nu d\sigma.
\end{align*}  
%
%The idea now would be to multiply the above identity by $h$ and integrate over $\Omega$. However, one can only do this for $s$ in the range $s\in (1/2,1)$ since for $s\in (0,1/2)$ we do not know whether the function $y\mapsto\nabla_y G_\Omega(x,\cdot)\cdot Y(y)$ integrated against $h$ is finite or not. To go around this difficulty and treat both cases simultaneously, we remove the singularity by multiplying with $h\widetilde{\upsilon}^x_\varepsilon$ where $\widetilde{\upsilon}_\varepsilon^x(\cdot) := 1-\upsilon(\frac{x-\cdot}{\varepsilon})$ with $\upsilon\in C^\infty_c(-2,+2)$ such that $\upsilon\equiv 1$ in $(-1,+1)$ and integrate over $\Omega$ to get:
Next, multiply the above identity with $h$ and integrate over $\Omega$ to get
\begin{align*}
v'(x)&=\int_{\Omega}g_s(x,y)h(y)dy+\int_{\Omega}G_\Omega^s(x,y)\div(h Y)dy\nonumber\\
&=\Gamma^2(1+s)\int_{\Omega}\Bigg(\int_{\partial\Omega}\gamma_0^s(G_\Omega^s(x,\cdot))\gamma_0^s(G^s_\Omega(y,\cdot))Y\cdot\nu d\sigma\Bigg)h(y)dy\nonumber\\
&+\int_{\Omega}G_\Omega^s(x,y)\div(h Y)dy+\int_{\Omega}\Bigg(\nabla_y G^s_\Omega(x,y)\cdot Y(y)+\nabla_x G^s_\Omega(x,y)\cdot Y(x)\Bigg)h(y)dy\nonumber\\
&=\Gamma^2(1+s)\int_{\partial\Omega}\gamma_0^s(u)\gamma_0^s(G_\Omega^s(x,\cdot))Y\cdot\nu\,d\sigma+\nabla u\cdot Y(x)
\end{align*}
where we used Lemma \ref{Ab} with $\psi=h$ and $f=f_x:=\gamma_0^s(G_\Omega^s(x,\cdot))Y\cdot\nu\in C^0(\partial\Omega)$ and Lemma \ref{reduc-FY}.
This is the identity we want to prove. The proof of Theorem \ref{THM-1.2} is therefore finished.
\end{proof}
\appendix
\section{Proof of Lemma \ref{Useful-bounds}}\label{Appendix.A}

\begin{Lemma}
Let $\Omega$ and $\{\Phi_t\}_{t\in (-1/2,1/2)}$ satisfy Assumption \ref{transformations}. Then for all $\beta$ in $\R$, there exist $C>0$ and $t_0>0$ such that, for all $t$ in $(-t_0,t_0)$ and all $y,z,h$ in $\R^N$, the following estimates hold:
\begin{align}
    \Big||\Phi_t(y)-\Phi_t(y+z)|^{-\beta}-|\Phi_t(y)-\Phi_t(y-z)|^{-\beta}\Big|&\leq C|z|^{1-\beta}, \label{WWW-append}\\
    |\kappa_t(y+h,z+h)-\kappa_t(y,z)|&\leq C\frac{|h|}{|y-z|^{N+2s}},\label{Sauveur1-append}
\end{align}
where $\kappa_t$ is defined in \eqref{ktyz}.
\end{Lemma}
\begin{proof}
We first recall that, for $t\neq 0$, we may write
\begin{align}\label{ExpPhit}
  \Phi_t(x)=x+tY_t(x),
\end{align}
where $Y_t:=(\Phi_t-\mathrm{Id})/t$. Assumption \ref{transformations} implies that the family $\{Y_t\}_{t\in (-1/2,1/2)}$ is uniformly Lipschitz and that $\{D\Phi_t\}_{t\in (-1/2,1/2)}$ is uniformly Lipschitz. In particular, there exists $C>0$ such that for all $t$ in $(-1/2,1/2)$ and all $u,v$ in $\R^N$,
\begin{align}\label{LipYt}
  |Y_t(u)-Y_t(v)|\leq C|u-v|,
  \qquad
  \|D\Phi_t(u)-D\Phi_t(v)\|\leq C|u-v|.
\end{align}
Moreover, shrinking $t_0>0$ if needed, \eqref{ellip-reg} yields the uniform bi-Lipschitz bounds
\begin{align}\label{bilipPhi}
  C^{-1}|u-v|\leq |\Phi_t(u)-\Phi_t(v)|\leq C|u-v|,
  \qquad
  t\in (-t_0,t_0).
\end{align}

\medskip

We now prove \eqref{WWW-append}. For fixed $t$ in $(-t_0,t_0)$, $y$ in $\R^N$ and $z\neq 0$, we set
\begin{align}\label{defApm}
  A_+(y,z):=\Phi_t(y)-\Phi_t(y+z),
  \qquad
  A_-(y,z):=\Phi_t(y)-\Phi_t(y-z),
\end{align}
and $r_\pm(y,z):=|A_\pm(y,z)|^2$. Using \eqref{bilipPhi}, we have
\begin{align}\label{r-comp}
  C^{-2}|z|^2\leq r_\pm(y,z)\leq C^{2}|z|^2.
\end{align}
Moreover,
\begin{align}\label{diff-r}
  r_+(y,z)-r_-(y,z)=\big(A_+(y,z)-A_-(y,z)\big)\cdot\big(A_+(y,z)+A_-(y,z)\big).
\end{align}
By \eqref{bilipPhi}, $|A_+(y,z)-A_-(y,z)|=|\Phi_t(y-z)-\Phi_t(y+z)|\leq C|z|$. On the other hand, the Lipschitz bound on $D\Phi_t$ in \eqref{LipYt} implies the second-difference estimate
\begin{align}\label{second-diff-Phi}
  |A_+(y,z)+A_-(y,z)|
  =|2\Phi_t(y)-\Phi_t(y+z)-\Phi_t(y-z)|
  \leq C|z|^2.
\end{align}
Combining \eqref{diff-r} and \eqref{second-diff-Phi} gives $|r_+(y,z)-r_-(y,z)|\leq C|z|^3$.

We apply the mean value theorem to the function $f(r):=r^{-\beta/2}$ on $(0,+\infty)$. Using \eqref{r-comp} we obtain
\begin{align}\label{mvtWWW}
  \Big||A_+(y,z)|^{-\beta}-|A_-(y,z)|^{-\beta}\Big|
  &=\big|f(r_+(y,z))-f(r_-(y,z))\big|\nonumber\\
  &\leq \frac{|\beta|}{2}\sup_{\theta\in (0,1)}\big(\theta r_+(y,z)+(1-\theta)r_-(y,z)\big)^{-\beta/2-1}|r_+(y,z)-r_-(y,z)|\nonumber\\
  &\leq C|z|^{-\beta-2}\,|z|^3=C|z|^{1-\beta},
\end{align}
which is \eqref{WWW-append}.

\medskip

We now prove \eqref{Sauveur1-append}. Recall that
\begin{align}\label{defkappa-append}
  \kappa_t(y,z)=\frac{c_{N,s}}{2}\frac{\mathrm{Jac}_{\Phi_t}(y)\mathrm{Jac}_{\Phi_t}(z)}{|\Phi_t(y)-\Phi_t(z)|^{N+2s}}.
\end{align}
We write
\begin{align}\label{kappa-diff-decomp}
  \kappa_t(y+h,z+h)-\kappa_t(y,z)=\mathcal A_t+\mathcal B_t,
\end{align}
where
\begin{align}\label{defAB}
  \mathcal A_t
  &:=\frac{c_{N,s}}{2}\frac{\mathrm{Jac}_{\Phi_t}(y+h)\mathrm{Jac}_{\Phi_t}(z+h)-\mathrm{Jac}_{\Phi_t}(y)\mathrm{Jac}_{\Phi_t}(z)}{|\Phi_t(y+h)-\Phi_t(z+h)|^{N+2s}},\nonumber\\
  \mathcal B_t
  &:=\frac{c_{N,s}}{2}\mathrm{Jac}_{\Phi_t}(y)\mathrm{Jac}_{\Phi_t}(z)\Big(|\Phi_t(y+h)-\Phi_t(z+h)|^{-N-2s}-|\Phi_t(y)-\Phi_t(z)|^{-N-2s}\Big).
\end{align}
The Jacobian map $y\mapsto \mathrm{Jac}_{\Phi_t}(y)$ is Lipschitz uniformly in $t$ (since $D\Phi_t$ is uniformly Lipschitz), hence
$|\mathrm{Jac}_{\Phi_t}(y+h)\mathrm{Jac}_{\Phi_t}(z+h)-\mathrm{Jac}_{\Phi_t}(y)\mathrm{Jac}_{\Phi_t}(z)|\leq C|h|$.
Together with \eqref{bilipPhi}, this gives
\begin{align}\label{boundA}
  |\mathcal A_t|\leq C\frac{|h|}{|y-z|^{N+2s}}.
\end{align}

For $\mathcal B_t$, we set $a:=\Phi_t(y)-\Phi_t(z)$ and $b:=\Phi_t(y+h)-\Phi_t(z+h)$. By the fundamental theorem of calculus and the Lipschitz bound on $D\Phi_t$ in \eqref{LipYt},
\begin{align}\label{bmina}
  |b-a|
  &=\Big|\int_0^1\big(D\Phi_t(y+\theta h)-D\Phi_t(z+\theta h)\big)h\, d\theta\Big|
  \leq C|y-z|\,|h|.
\end{align}
Since $|a|\simeq |y-z|$ and $|b|\simeq |y-z|$ by \eqref{bilipPhi}, the mean value theorem applied to $r\mapsto r^{-N-2s}$ yields
\begin{align}\label{boundB}
  \Big||b|^{-N-2s}-|a|^{-N-2s}\Big|\leq C|y-z|^{-N-2s-1}\,|b-a|
  \leq C\frac{|h|}{|y-z|^{N+2s}}.
\end{align}
Using the uniform boundedness of $\mathrm{Jac}_{\Phi_t}$, \eqref{boundB} implies
\begin{align}\label{boundB2}
  |\mathcal B_t|\leq C\frac{|h|}{|y-z|^{N+2s}}.
\end{align}
Combining \eqref{kappa-diff-decomp}, \eqref{boundA}, and \eqref{boundB2} proves \eqref{Sauveur1-append}.
\end{proof}

\section{Proof of Lemma \ref{FundLemma2}}\label{Appendix.B}
We recall the statement of the lemma for the reader convinience
\begin{Lemma}
    Let $\kappa_t(\cdot,\cdot)$ be given as in \eqref{ktyz} and set
\begin{equation}
\overline{\kappa_t}(y,z):=\frac{\kappa_t(y,y+z)-\kappa_0(y,y+z)}{t}\qquad\textrm{for $t\neq 0$,}
\end{equation}
Let $ \overline{\kappa_t}^o(\cdot,\cdot)$ be the odd part of $ \overline{\kappa_t}(\cdot,\cdot)$, i.e, 
\begin{equation}
\overline{\kappa_t}^o(y,z)=\frac{\overline{\kappa_t}(y,y+z)- \overline{\kappa_t}(y,y-z)}{2}. 
\end{equation}
Then, there holds  
\begin{equation}\label{SauV}
    \big|\overline{\kappa_t}^o(y,z)\big|\leq C\min(|z|^{-2s-N},|z|^{1-2s-N}), 
    \end{equation}
  for some constant $C>0$ that is independent of $t$ and $z$.
\end{Lemma}
\begin{proof}
The fact that $|\overline{\kappa_t}^o(y,z)|\leq C|z|^{-2s-N}$ follows from the expansion \eqref{expansion-k-t-notations}. To get the other bound, we use the fundamental theorem of calculus and write:
\begin{align}\label{E1}
2\overline{\kappa_t}^o(y,z)&=\overline{\kappa_t}(y,y+z)-\overline{\kappa_t}(y,y-z)\nonumber  \\
&=\frac{\kappa_t(y,y+z)-\kappa_0(y,y+z)}{t}-\frac{\kappa_t(y,y-z)-\kappa_0(y,y-z)}{t}\nonumber\\
&=t^{-1}\int_{0}^t\partial_\tau \Big[\kappa_\tau(y,y+z)-\kappa_\tau(y,y-z)\Big]d\tau.
\end{align}
We are going to prove that the quantity under the integral sign can be controlled uniformly in $t$ and $z$ by the power function $|z|^{1-2s-N}$. It will be convenient to use from now on the following notation. 
$$
z_t:=\Phi_t(z)\qquad\text{for $z\in\R^N$}.
$$  
For the sake of simplicity, we also set
\begin{align*}
&\Lambda_\tau(y,y+z):=(2s+N)\big(\partial_\tau y_\tau-\partial_\tau (y+z)_\tau\big)\cdot(y_\tau-(y+z)_\tau)|y_\tau-(y+z)_\tau|^{-2-2s-N},\\
&\Psi_\tau(y,y+z):=   \big(\partial_\tau y_\tau-\partial_\tau (y+z)_\tau\big)\cdot(y_\tau-(y+z)_\tau),\\
&\Gamma_\tau(y,y+z):=\textrm{Jac}_{\Phi_\tau}(y)\partial_\tau \textrm{Jac}_{\Phi_\tau}(y+z)+\textrm{Jac}_{\Phi_\tau}(y+z)\partial_\tau\textrm{Jac}_{\Phi_\tau}(y).
\end{align*}
\par\;
Then a direct calculation gives 
\begin{align*}
    \partial_\tau \kappa_\tau(y,y+z) &= |y_\tau-(y+z)_\tau|^{-2s-N}\Gamma_\tau(y,y+z) -\textrm{Jac}_{\Phi_\tau}(y)\textrm{Jac}_{\Phi_\tau}(y+z)\Lambda_\tau(y,y+z).
\end{align*}
Consequently, we have 
\begin{align}\label{E2}
& \partial_\tau \Big[\kappa_\tau(y,y+z)-\kappa_\tau(y,y-z)\Big]\nonumber\\
%&=\textrm{Jac}_{\Phi_\tau}(y)\textrm{Jac}_{\Phi_\tau}(y-z)\Lambda_\tau(y,y-z)-\textrm{Jac}_{\Phi_\tau}(y)\textrm{Jac}_{\Phi_\tau}(y+z)\Lambda_\tau(y,y+z)\\
%&+|y_\tau-(y+z)_\tau|^{-2s-N}\Gamma_\tau(y,y+z)-|y_\tau-(y-z)_\tau|^{-2s-N}\Gamma_\tau(y,y-z)\\
&=\textrm{Jac}_{\Phi_\tau}(y)\Big[\textrm{Jac}_{\Phi_\tau}(y-z)-\textrm{Jac}_{\Phi_\tau}(y+z)\Big]\Lambda_\tau(y,y-z)\nonumber\\
&+\textrm{Jac}_{\Phi_\tau}(y)\textrm{Jac}_{\Phi_\tau}(y+z)\Big[\Lambda_\tau(y,y-z)-\Lambda_\tau(y,y+z)\Big]\nonumber\\
&+\Big[|y_\tau-(y+z)_\tau|^{-2s-N}-|y_\tau-(y-z)_\tau|^{-2s-N}\Big]\Gamma_\tau(y,y+z)\nonumber\\
&+|y_\tau-(y-z)_\tau|^{-2s-N}\Big[\Gamma_\tau(y,y+z)-\Gamma_\tau(y,y-z)\Big].
\end{align}
\par\;
Next, we rewrite the second contribution from the above expansion as:
\begin{align}
&\frac{1}{2s+N}\Big|\Lambda_\tau(y,y-z)-\Lambda_\tau(y,y+z)\Big|\nonumber\\
&=\Big|\frac{\Psi_\tau(y,y-z)}{|y_\tau-(y-z)_\tau|^{N+2s+2}} -\frac{\Psi_\tau(y,y+z)}{|y_\tau-(y+z)_\tau|^{N+2s+2}}\Big|\nonumber\\
&=\big|\Psi_\tau(y,y-z)\big|\Big||y_\tau-(y-z)_\tau|^{-2-2s-N}-|y_\tau-(y+z)_\tau|^{-2-2s-N}\Big|\nonumber\\
&+|y_\tau-(y+z)_\tau|^{-2s-2-N}\Big|\Psi_\tau(y,y-z)-\Psi_\tau(y,y+z)\Big|.\label{E3}
\end{align}
The last term in \eqref{E3} can also be rewritten as:
\begin{align}\label{E4}
\Big|\Psi_\tau(y,y-z)-\Psi_\tau(y,y+z)\Big|
&=\Big|\Big(2\partial_\tau y_\tau-\partial_\tau(y+z)_\tau-\partial_\tau(y-z)_\tau\Big)\cdot\big(y_\tau-(y-z)_\tau\big)\nonumber\\
\qquad\qquad&+\big(\partial_\tau(y+z)_\tau-\partial_\tau y_\tau\big)\cdot\Big(2y_\tau-(y+z)_\tau-(y-z)_\tau\Big)\Big|\nonumber\\
&\leq \Big|2\partial_\tau y_\tau-\partial_\tau(y+z)_\tau-\partial_\tau(y-z)_\tau\Big|\big|y_\tau-(y-z)_\tau\big|\nonumber\\
&+\big|\partial_\tau(y+z)_\tau-\partial_\tau y_\tau\big|\Big|2y_\tau-(y+z)_\tau-(y-z)_\tau\Big|.
\end{align}
\par\;

By using the fundamental theorem of calculus and the hypothesis \eqref{transformations}, we see that (see proof of Lemma \ref{Useful-bounds})
\begin{align}\label{E5}
 \Big|2y_\tau-(y+z)_\tau-(y-z)_\tau\Big| &= \Big|2\Phi_\tau(y)-\Phi_\tau(y+z)-\Phi_\tau(y-z)\Big| \leq C|z|^2,\\
 \Big|2\partial_\tau y_\tau-\partial_\tau(y+z)_\tau-\partial_\tau(y-z)_\tau\Big|&=\Big|2\partial_\tau \Phi_\tau(y)-\partial_\tau\Phi_\tau(y+z)-\partial_\tau\Phi_\tau(y-z)\Big|\leq C|z|^2.
\end{align}
\par\;

By \eqref{ellip-reg} we already know that 
\begin{align}\label{E6}
 \big|y_\tau-(y-z)_\tau\big|\leq C|z|,
\end{align}
for some $C>0$ independent of $\tau,z$, and also
\begin{equation}\label{E6.1}
\big|\partial_\tau(y+z)_\tau-\partial_\tau y_\tau\big| \leq C|z|.
\end{equation}
Hence 
\begin{align}\label{E7}
    |\Psi_\tau(y,y+z)|=   \big|\partial_\tau y_\tau-\partial_\tau (y+z)_\tau\big||y_\tau-(y+z)_\tau|\leq C|z|^2.
\end{align}
By Lemma \ref{Useful-bounds} we also know that 
\begin{align}\label{E8}
\Big||y_\tau-(y+z)_\tau|^{-2-2s-N}-|y_\tau-(y-z)_\tau|^{-2-2s-N}\Big|\leq C|z|^{-1-2s-N},
\end{align}
and 
\begin{equation}\label{E9}
\Big||y_\tau-(y+z)_\tau|^{-2s-N}-|y_\tau-(y-z)_\tau|^{-2s-N}\Big|\leq C|z|^{1-2s-N}.
\end{equation}
\par\;

%Thanks to the assumption \eqref{transformations}, we may write:
%\begin{align}\label{ExpPhit}
%\Phi_t(x)=x+t Y_t(x)\qquad \text{with}\qquad Y_t(x) = \frac{1}{t}\int_{0}^t\partial_\tau \Phi_\tau(x)d\tau.
%\end{align}
%Since the mapping  $(-1,+1)\ni t\mapsto \partial_t\Phi_t\in C^{1,1}(\R^N,\R^N)$ is continuous, then , for $|t|\ll 1/2$ small enough, we have:
%\begin{align*}
%\big|Y_t(u)-Y_t(v)\big|&=\frac{1}{t}\int_{0}^t\big|\partial_\tau\Phi_\tau(u)-\partial_\tau\Phi_\tau(v)\big|d\tau\nonumber\\
%&\leq |u-v|\sup_{\tau\in [-1/2, +1/2]}\|\partial_\tau\Phi_\tau\|_{C^{0,1}(\R^N,\R^N)}\nonumber\\
%&\leq C|u-v|,\quad\text{$\forall\, u,v\in\R^N$.}
%\end{align*}
%It follows that 
%$$
%DY_t(u)=\frac{1}{t}\int_{0}^tD(\partial_\tau\Phi_\tau)(u)d\tau,
%$$
%and hence
%\begin{align*}
%\big|DY_t(u)-DY_t(v)\big|&=\frac{1}{t}\int_{0}^t\big|D(\partial_\tau\Phi_\tau)(u)-D(\partial_\tau\Phi_\tau)(v)\big|d\tau\nonumber\\
%&\leq |u-v|\sup_{\tau\in [-1/2, +1/2]}\|D\partial_\tau\Phi_\tau\|_{C^{0,1}(\R^N,\R^N)}\nonumber\\
%&\leq C|u-v|,\quad\text{$\forall\, u,v\in\R^N$.}
%\end{align*}
%In other words, for $|t|\ll 1/2$ sufficiently small we have 
%\begin{equation}\label{KB}
% \|Y_t\|_{C^{0,1}(\R^N,\R^N)}\leq C_0,\quad  
%   \| DY_t\|_{L^\infty(\R^N\times\R^N)}\leq %C_0\qquad\text{and}\qquad \|DY_t\|_{C^{0,1}(\R^N,\,\R^N\times\R^N)}\leq C_0, 
%\end{equation}
%for some $C>0$ that is independent of $t$
%\par\;

Since $\textrm{Jac}_{\Phi_\tau}(z)= \det(\textrm{Id}+\tau DY_\tau(z))$ and we know that the determinant is globally Lipschitz continuous in bounded sets of matrices, then in view of \eqref{LipYt}, we have, for all $u,v\in \R^N$, that
\begin{align}\label{Jacob-Lip}
\Big|\textrm{Jac}_{\Phi_\tau}(u)-\textrm{Jac}_{\Phi_\tau}(v)\Big| &= \Big|\det(\textrm{Id}+\tau DY_\tau(u))-\det(\textrm{Id}+\tau DY_\tau(v))\Big|\nonumber\\
&\leq C\|DY_\tau(u)-DY_\tau(v)\|\nonumber\\
&\leq C|u-v|,
\end{align}
for some $C>0$ that is independent of $\tau$.
It follows that
\begin{align}\label{E10}
\Big|\textrm{Jac}_{\Phi_\tau}(y-z)-\textrm{Jac}_{\Phi_\tau}(y+z)\Big| \leq C|z|,
\end{align}
for some $C>0$ that is independent of $\tau$ and $z$. Next using the classical identity
$$
\partial_t \det\big(\mathrm{Id}+\,M(t)\big)
=
\det\big(\mathrm{Id}+\,M(t)\big)\,
\operatorname{tr}\!\left[
\big(\mathrm{Id}+\,M(t)\big)^{-1}
\,M'(t)
\right],
$$
we write 
\begin{align}
\partial_\tau\textrm{Jac}_{\Phi_\tau}(z)
&=\partial_\tau\det(\textrm{Id}+\tau DY_\tau(z))\nonumber\\
&=\partial_\tau\det\Big(\textrm{Id}+\int_{0}^\tau D(\partial_r\Phi_r)(z)dr\Big)\nonumber\\
&=\det\big(\textrm{Id}+\tau DY_\tau(z)\big)\operatorname{tr}\Big[\big(\textrm{Id}+\tau DY_\tau(z)\big)^{-1}D(\partial_\tau\Phi_\tau)(z)\Big]\nonumber\\
&=\textrm{Jac}_{\Phi_\tau}(z)\operatorname{tr}\Big[\big(\textrm{Id}+\tau DY_\tau(z)\big)^{-1}D(\partial_\tau\Phi_\tau)(z)\Big].\label{partial-Jac}
\end{align}
By the hypothesis \eqref{transformations} we clearly have
\begin{equation}\label{c1-partial-Jac}
\|D(\partial_\tau\Phi_\tau)\|_{L^\infty(\R^N\times\R^N)}\leq \sup_{\tau\in [-1/2,+1/2]} \|\partial_\tau\Phi_\tau\|_{C^{0,1}(\R^N,\R^N)}\leq C.
\end{equation}
And also, since the mapping $\tau\mapsto D(\partial_\tau\Phi_\tau)\in C^{0,1}(\R^N,\R^N)$ is continuous, then we have 
\begin{equation}\label{c2-partial-Jac}
    \|D(\partial_\tau\Phi_\tau)(u)-D(\partial_\tau\Phi_\tau)(v)\|_{L^\infty(\R^N)}\leq \sup_{\tau\in [-1/2,+1/2]}\|D(\partial_\tau\Phi_\tau)\|_{C^{0,1}(\R^N,\R^N)}|u-v|\leq C|u-v|,
\end{equation}
for all $u,v\in \R^N$. And finally, using the identity
$$
\big(\textrm{Id}+\tau DY_\tau(z)\big)^{-1} = \sum_{k=0}^\infty (-1)^k\tau^k(DY_\tau(z))^k,
$$
we have, first 
\begin{equation}\label{c3-partial-Jac}
\Big\|\big(\textrm{Id}+\tau DY_\tau(z)\big)^{-1}\Big\|_{L^\infty(\R^N\times\R^N)}\leq \sum_{k=0}^\infty |\tau|^kC_0^k\leq 2\quad\text{for $|\tau|<1/(2C_0)$}.
\end{equation}
On the other hand, we have in view of \eqref{LipYt}, that
\begin{align}\label{c4-partial-Jac}
 &\Big\|\big(\textrm{Id}+\tau DY_\tau(u)\big)^{-1}-\big(\textrm{Id}+\tau DY_\tau(v)\big)^{-1}\Big\|\nonumber\\
 &\leq \sum_{k=0}^\infty (-1)^k|\tau|^k\Big\|(DY_\tau(u))^k-(DY_\tau(v))^k\Big\|\nonumber\\
 &\leq \sum_{k=0}^\infty (-1)^k|\tau|^k\Big\|DY_\tau(u)-DY_\tau(v)\Big\|\sum_{j=0}^{k-1}\Big\|DY_\tau(v)\Big\|^{k-1-j}\Big\|DY_\tau(u)\Big\|^j\nonumber\\
 &\leq C_0|u-v|\sum_{k=0}^\infty |\tau|^kkC_0^{k-1}\nonumber\\
 &\leq |u-v|\frac{C_0|\tau|}{(1-C_0|\tau|)^2}\leq 2|u-v|\quad\text{for $|\tau|<1/(2C_0)$},
\end{align}
for all $u,v\in \R^N$. Combining \eqref{partial-Jac}, \eqref{c1-partial-Jac}, \eqref{c2-partial-Jac}, \eqref{c3-partial-Jac} and \eqref{c4-partial-Jac} we get
\begin{align}\label{E11}
\Big|\partial_\tau\textrm{Jac}_{\Phi_\tau}(y-z)-\partial_\tau\textrm{Jac}_{\Phi_\tau}(y+z)\Big| \leq C|z|.
\end{align}

Combining \eqref{E1}--\eqref{E10} and \eqref{E11} we get

\begin{align}
  \Big|\partial_\tau(\kappa_\tau(y,y+z)-\kappa_\tau(y,y-z))\Big|\leq C|z|^{1-2s-N},
\end{align}
for some $C>0$ that is independent of $\tau$ and $z$. Use this into \eqref{E1} gives 
$$
|\overline{\kappa_t}^o(y,z)|\leq C|z|^{1-2s-N},
$$
as wanted.    
\end{proof}
\section{A useful integral identity involving the kernel $\omega^x_{Y}(\cdot,\cdot)$}

Recall the definition of $\omega^x_{Y}$ in \eqref{defWsam}. For a fixed $x$ in $\Omega$, we set
\begin{align}\label{defYtilde}
  \widetilde{Y_x}(y):=DY(x)\cdot y,
\end{align}
so that $\widetilde{Y_x}$ is a linear (hence globally Lipschitz) vector field on $\R^N$.

\begin{Lemma}\label{appen}
For all $x$ in $\Omega$ and all $w$ in $C^\infty_0(\R^N)$, we have
\begin{align}\label{surprise2}
  & \iint_{\R^{2N}}\frac{\big(w(y)-w(z)\big)\big(|y|^{2s-N}-|z|^{2s-N}\big)}{|y-z|^{N+2s}}\omega^x_{Y}(y,z)\, dy\, dz \nonumber\\
  &\qquad=-(N-2s)\int_{\R^N}\frac{\big[DY(x)\cdot z\big]\cdot z}{|z|^{N-2s+2}}(-\Delta)^s w(z)\, dz.
\end{align}
\end{Lemma}
\begin{proof}
Let $\eta$ in $C^\infty(\R)$ satisfy $0\leq \eta\leq 1$, $\eta\equiv 0$ on $(-\infty,1]$, and $\eta\equiv 1$ on $[2,+\infty)$. For $\mu>0$ and $R>0$, we set
\begin{align}\label{def-cutoff-muR}
  \eta_\mu(y):=\eta\Big(\frac{|y|}{\mu}\Big),
  \qquad
  \widetilde{\eta}_R(y):=\eta\Big(\frac{|y|}{R}\Big),
  \qquad
  a^R_{\mu}(y):=|y|^{2s-N}\eta_\mu(y)\widetilde{\eta}_R(y).
\end{align}
Since $w$ has compact support, dominated convergence gives
\begin{align}\label{approx-kernel}
  &\iint_{\R^{2N}}\frac{\big(w(y)-w(z)\big)\big(|y|^{2s-N}-|z|^{2s-N}\big)}{|y-z|^{N+2s}}\omega^x_{Y}(y,z)\, dy\, dz \nonumber\\
  &\qquad=\lim_{\mu\to 0^+}\lim_{R\to +\infty}\iint_{\R^{2N}}\frac{\big(w(y)-w(z)\big)\big(a^R_{\mu}(y)-a^R_{\mu}(z)\big)}{|y-z|^{N+2s}}\omega^x_{Y}(y,z)\, dy\, dz.
\end{align}
Applying Lemma \ref{preumlem} with $v=a^R_{\mu}$, we get
\begin{align}\label{step-preumlem}
  &\iint_{\R^{2N}}\frac{\big(w(y)-w(z)\big)\big(a^R_{\mu}(y)-a^R_{\mu}(z)\big)}{|y-z|^{N+2s}}\omega^x_{Y}(y,z)\, dy\, dz \nonumber\\
  &\qquad= -\int_{\R^N}\nabla w(y)\cdot \widetilde{Y_x}(y)(-\Delta)^s a^R_{\mu}(y)\, dy
          -\int_{\R^N}\nabla a^R_{\mu}(y)\cdot \widetilde{Y_x}(y)(-\Delta)^s w(y)\, dy.
\end{align}
We claim that the first term on the right-hand side vanishes as $\mu\to 0^+$ and $R\to+\infty$. Indeed, since $\widetilde{Y_x}(y)=DY(x)\cdot y$, the factor $\nabla w(y)\cdot \widetilde{Y_x}(y)$ vanishes at $y=0$, and $(-\Delta)^s a^R_{\mu}$ converges (in the sense of distributions) to a multiple of $\delta_0$ as $\mu\to 0^+$ and $R\to+\infty$.
Therefore,
\begin{align}\label{claim2prove}
  \lim_{\mu\to 0^+}\lim_{R\to +\infty}\int_{\R^N}\nabla w(y)\cdot \widetilde{Y_x}(y)(-\Delta)^s a^R_{\mu}(y)\, dy=0.
\end{align}
For the second term, we use $\nabla a^R_{\mu}(y)\to \nabla(|y|^{2s-N})$ for all $y\neq 0$, and
\begin{align}\label{grad-fundsol}
  \nabla(|y|^{2s-N})=-(N-2s)\frac{y}{|y|^{N-2s+2}}.
\end{align}
Since $(-\Delta)^s w$ is smooth and satisfies the decay bound $|(-\Delta)^s w(y)|\leq C(1+|y|)^{-N-2s}$, another application of dominated convergence yields
\begin{align}\label{combibi2}
  &\lim_{\mu\to 0^+}\lim_{R\to +\infty}\int_{\R^N}\nabla a^R_{\mu}(y)\cdot \widetilde{Y_x}(y)(-\Delta)^s w(y)\, dy \nonumber\\
  &\qquad=-(N-2s)\int_{\R^N}\frac{\big[DY(x)\cdot y\big]\cdot y}{|y|^{N-2s+2}}(-\Delta)^s w(y)\, dy.
\end{align}
Combining \eqref{approx-kernel}--\eqref{combibi2} proves \eqref{surprise2}.
\end{proof}
\begin{Lemma}\label{preumlem}
For all $x$ in $\Omega$ and all $w,v$ in $C^\infty_0(\R^N)$, we have
\begin{align}\label{FS-int-01}
    &\iint_{\R^{2N}}\frac{\big(w(y)-w(z)\big)\big(v(y)-v(z)\big)}{|y-z|^{N+2s}}\omega^x_{Y}(y,z)\, dy\, dz \nonumber\\
    &\qquad=-\int_{\R^N}\nabla w(z)\cdot \widetilde{Y_x}(z)(-\Delta)^s v(z)\, dz
           -\int_{\R^N}\nabla v(z)\cdot \widetilde{Y_x}(z)(-\Delta)^s w(z)\, dz.
\end{align}
In particular, when $w=v$ we obtain
\begin{align}\label{FS-int}
    \iint_{\R^{2N}}\frac{\big(w(y)-w(z)\big)^2}{|y-z|^{N+2s}}\omega^x_{Y}(y,z)\, dy\, dz
    =-2\int_{\R^N}\nabla w(y)\cdot \widetilde{Y_x}(y)(-\Delta)^s w(y)\, dy.
\end{align}
\end{Lemma}
\begin{proof}
For $x$ in $\Omega$, we consider the linear vector field $\widetilde{Y_x}$ defined in \eqref{defYtilde}. By construction, $\omega^x_{Y}=\omega_{\widetilde{Y_x}}$ in the notation of \eqref{defoY-notations}.
Therefore, the left-hand side in \eqref{FS-int-01} is precisely $\mathcal E_{\widetilde{Y_x}}(w,v)$ as defined in \eqref{duality-pairing}.

Since $\widetilde{Y_x}$ is globally Lipschitz and $w,v$ are compactly supported, we may apply \cite[Lemma 2.1]{DFW2023} with the vector field $\widetilde{Y_x}$ and obtain \eqref{FS-int-01}. Taking $w=v$ gives \eqref{FS-int}.
\end{proof}

\section{Estimates on the cut-off function $\xi_k$}
The following results are consequences of the estimates in \cite[Lemma 6.7 and Lemma 6.8]{DFW}. They were used in the proof of Lemma \ref{Lem.sec-4.3} and Lemma \ref{main-res-sec4.4}.
\begin{Lemma}\label{lem-crucial} Let $\Omega$ be a bounded open set of class $C^{1,1}$.
 Let $\rho$ in $C^\infty_c(-2,2)$ such that $0\leq \rho\leq 1$ and $\rho\equiv 1$ in $(-1,1)$. 
For any $k\in \mathbb N^*$ we recall \eqref{eta-k}. 
Then for all $x\in \Omega$, there holds:
\begin{align}\label{apendix2.claim1}
& |(-\Delta)^s\xi_k^2(x)|\leq C(N,s,\Omega)\delta_{\Omega}^{-2s}(x)\quad\text{and}\quad (-\Delta)^s\xi_k^2(x)\to 0\quad\text{as}\quad k\to\infty, 
%\\
%&\Big|\mathcal I_s[\xi_k^2,\xi_k](x)\Big|\leq C(N,s)\delta_{\Omega}^{-s}(x)\quad\text{and}\quad  \mathcal I_s[\xi_k,\xi_k](x)\to 0\quad \text{as}\quad k\to\infty.\label{78} \;\textcolor{red}{\text{To be re-checked carefully!!!}}
\end{align}
Let $h$ in $C^s(\overline{\Omega})$ and $h=0$ in $\R^N\setminus \Omega$. Then for all $x\in \Omega$, there holds
\begin{align}\label{apendix2.claim2}
&\int_{\R^N}\frac{|\xi_k(x)-\xi_k(y)||h(x)-h(y)|}{|x-y|^{N+2s}}dy \leq C(N,s,\Omega)\delta_{\Omega}^{-s}(x).
\end{align}
Moreover, we have 
\begin{align}\label{xi-k-limit}
\int_{\R^N}\frac{|\xi_k(x)-\xi_k(y)||h(x)-h(y)|}{|x-y|^{N+2s}}dy\rightarrow 0\quad \text{as}\quad k\to\infty.
\end{align}
\end{Lemma}
\begin{remark}\label{remark-Append}
Since $|\xi_k^2(x)-\xi_k^2(y)|\leq 2|\xi_k(x)-\xi_k(y)|$, it follows that 
\begin{align}\label{Eq1-rm-Append}
&\int_{\R^N}\frac{|\xi_k^2(x)-\xi_k^2(y)||h(x)-h(y)|}{|x-y|^{N+2s}}dy \leq C(N,s,\Omega)\delta_{\Omega}^{-s}(x).
\end{align}
and 
\begin{align}\label{xi-k2-limit}
\int_{\R^N}\frac{|\xi_k^2(x)-\xi_k^2(y)||h(x)-h(y)|}{|x-y|^{N+2s}}dy\rightarrow 0\quad \text{as}\quad k\to\infty.
\end{align}
\end{remark}
\begin{proof} 
This follows by adapting the argument used in \cite[Lemma 6.8, Proposition 6.3]{DFW}. We briefly sketch the proof and refer to \cite{DFW} for details. Let $0<\upsilon$ in $C^\infty_0(\R^N)$ such that $\int_{\R^N}\upsilon=1$ and let $\upsilon_t(\cdot):=t^{-N}\upsilon(\frac{\cdot}{t})$. Let $x$ in $\Omega$ be fixed. Then for $t>0$ sufficiently small, we may write
\begin{align}\label{Akeps+Bkeps}
\big[\upsilon_t*(-\Delta)^s\xi_k^2\big](x) &= \int_{\R^N}(-\Delta)^s\xi_k^2(y)\upsilon_t(x-y)dy\nonumber\\
&=\int_{\Omega}(-\Delta)^s\xi_k^2(y)\upsilon_t(x-y)dy\nonumber\\
&=\int_{\Omega\setminus\Omega^\varepsilon_+}(-\Delta)^s\xi_k^2(y)\upsilon_t(x-y)dy+\int_{\Omega^\varepsilon_+}(-\Delta)^s\xi_k^2(y)\upsilon_t(x-y)dy\nonumber\\
&:=A^\varepsilon_k(t,x)+B^\varepsilon_k(t,x)
\end{align}
where in the above we let  $\Omega^{\varepsilon}_+
:= \{x\in \Omega : 0< \dist(x,\partial\Omega)<\varepsilon\}$ for every $\varepsilon>0$ small.  Next, we claim that for $k$  sufficiently large, we have 
\begin{equation}\label{unif-away-bd}
    \|(-\Delta)^s\xi_k^2\|_{L^\infty(\Omega\setminus\Omega_+^\eps)}\leq C\qquad\text{and}\qquad (-\Delta)^s\xi_k^2(y)\to 0\quad\text{as $k\to\infty$},
\end{equation}
for some $C>0$ independent of $k$. Indeed, let $K\subset\Omega$ be a compact subset of $\Omega\setminus\Omega_+^\varepsilon$. Then  we have 
$$
\frac{1}{c_{N,s}}(-\Delta)^s\xi_k^2(x)=\int_{\R^N\setminus K}\frac{1-\xi_k^2(y)}{|x-y|^{N+2s}}dy\quad\text{for $x\in\Omega\setminus\Omega_+^\varepsilon$ and $k$ sufficiently large}
$$
where $|1-\xi_k^2(y)||x-y|^{-N-2s}\leq C(1+|y|^{-N-2s})$ for $x\in\Omega\setminus\Omega_+^\varepsilon$ and $k$ sufficiently large where $C$ is independent of $k$. Since $1-\xi^2_k\to 0$ almost everywhere in $\R^N$, the claim \eqref{unif-away-bd} follows. Consequently, for $k$ sufficiently large and $t$ sufficiently small, we have 
\begin{align}\label{Akeps-es}
\big|A^\varepsilon_k(t,x)\big|\leq    C\quad\text{for some $C>0$ that is independent of $k$ and $t$}
\end{align}
where we used that $\int_{\R^N}\upsilon_t=1$.
Moreover, there holds 
\begin{align}\label{limAkeps}
\lim_{k\to\infty}\lim_{t\to0^+}A^\varepsilon_k(t,x)=0.
\end{align}
To estimate $B^\varepsilon_k(t,x)$, we define the mapping $\Psi: \partial\Omega\times (0,+\varepsilon)\to \Omega^\varepsilon:=\{y\in\Omega:0<\delta(y)< \eps\}$ by 
$$
\Psi(\sigma,r) := \sigma-r\nu(\sigma),
$$ where $\nu$ denotes the outward unit normal to the boundary. Next, we change variables using the transformation above and write 
\begin{align}\label{Beps}
    B^\varepsilon_k(t,x) &= \int_{\Omega^\varepsilon_+}(-\Delta)^s\xi_k^2(y)\upsilon_t(x-y)dy\nonumber\\
    &=k^{-1}\int_{0}^{k\varepsilon}\int_{\partial\Omega}(-\Delta)^s\xi_k^2(\Psi(\sigma, \frac{r}{k}))\upsilon_t(x-\Psi(\sigma, \frac{r}{k}))j(\sigma, \frac{r}{k})drd\sigma
\end{align}
Quoting \cite[Proposition 6.3]{DFW} we know  that there exists $\varepsilon'>0$ with the property that
\begin{equation}\label{xisquare-es}
\Big|\big[(-\Delta)^s\xi_k^2\big]\big(\Psi(\sigma,\frac{r}{k})\big)
\Big| \leq  C\frac{k^{2s}}{1+r^{1+2s}}\leq C\frac{k^{2s}}{1+r^{2s}}\quad  \text{for}\;\; k\in \N,\;\; 0\leq r \leq k\varepsilon',\; \sigma\in\partial\Omega,
\end{equation}
for some $C>0$ independent of $k$. Indeed, we have $(-\Delta)^s\xi_k^2=-2(-\Delta)^s\rho_k+(-\Delta)^s\rho_k^2=-2(-\Delta)^s\xi_k+(-\Delta)^s(1-\rho_k^2)$. Then by \cite[Proposition 6.3]{DFW} we know $\big|\big[(-\Delta)^s\xi_k\big]\big(\Psi(\sigma,\frac{r}{k})\big)
\big| \leq C\frac{k^{2s}}{1+r^{2s}}$. But the same argument applies without any change to $(-\Delta)^s(1-\rho_k^2)(\Psi(\sigma,r/k))$ since $\rho^2_k$ and $\rho_k$ have the same regularity required for the proof.
\par\;
Plugging \eqref{xisquare-es} into \eqref{Beps} we get, for $k$ large enough and $t$ sufficiently small that 
\begin{align}\label{Bkeps-es}
|B^{\varepsilon'}_k(t,x)|&\leq \frac{C}{k}k^{2s}\int_{0}^{k\varepsilon'}\int_{\partial\Omega}\frac{\upsilon_t(x-\Psi(\sigma, r/k))}{1+r^{2s}}j(\sigma, \frac{r}{k})drd\sigma\nonumber\\
&\leq \frac{C}{k}\int_{0}^{k\varepsilon'}\int_{\partial\Omega}\frac{\upsilon_t(x-\Psi(\sigma, r/k))}{(r/k)^{2s}}j(\sigma, \frac{r}{k})drd\sigma\nonumber\\
&\leq C\int_{0}^{\varepsilon'}\int_{\partial\Omega}\frac{\upsilon_t(x-\Psi(\sigma, r))}{r^{2s}}j(\sigma, r)drd\sigma\nonumber\\
&\leq C\int_{\Omega^\varepsilon_+}\frac{\upsilon_t(x-y)}{\delta^{2s}(y)}dy \leq C\int_{\Omega}\frac{\upsilon_t(y)}{\delta^{2s}(x-y)}dy \nonumber\\
&\leq C\delta^{-2s}(x)
\end{align}
for some $C>0$ that is independent of $k$ and $t$. On the other hand, using instead the bound\\ $\big|\big[(-\Delta)^s\xi_k^2\big]\big(\Psi(\sigma,\frac{r}{k})\big)
\big| \leq  C\frac{k^{2s}}{1+r^{1+2s}}$, we also have 
\begin{align}\label{limBkeps}
|B^{\varepsilon'}_k(t,x)|&\leq \frac{C}{k}k^{2s}\int_{0}^{k\varepsilon'}\int_{\partial\Omega}\frac{\upsilon_t(x-\Psi(\sigma, r/k))}{1+r^{1+2s}}j(\sigma, \frac{r}{k})drd\sigma\nonumber\\
&\leq \frac{C}{k}\int_{0}^{k\varepsilon'}\int_{\partial\Omega}\frac{\upsilon_t(x-\Psi(\sigma, r/k))}{r(r/k)^{2s}}j(\sigma, \frac{r}{k})drd\sigma\nonumber\\
&\leq \frac{C}{k}\int_{0}^{\varepsilon'}\int_{\partial\Omega}\frac{\upsilon_t(x-\Psi(\sigma, r))}{r^{1+2s}}j(\sigma, r)drd\sigma\nonumber\\
&\leq \frac{C}{k}\int_{\Omega^\varepsilon_+}\frac{\upsilon_t(x-y)}{\delta^{1+2s}(y)}dy \leq \frac{C}{k}\int_{\Omega}\frac{\upsilon_t(y)}{\delta^{1+2s}(x-y)}dy \nonumber\\
&\leq \frac{C}{k}\delta^{-1-2s}(x)\to 0\quad\text{as $k\to\infty$}.
\end{align}
Combining \eqref{Akeps+Bkeps}, \eqref{Akeps-es}, \eqref{limAkeps}, \eqref{Bkeps-es} and \eqref{limBkeps} we end up with 
\begin{align}
 \big|\big[\upsilon_t*(-\Delta)^s\xi_k^2\big](x)\big|\leq C\delta^{-2s}(x)\qquad\text{and}\qquad \lim_{k\to\infty}\lim_{t\to 0^+}\big[\upsilon_t*(-\Delta)^s\xi_k^2\big](x)=0,
\end{align}
for some $C>0$ that is independent of $k$ and $t$. Now since 
$$
|(-\Delta)^s\xi_k^2(x)| =\lim_{t\to 0^+}\big|\big[\upsilon_t*(-\Delta)^s\xi_k^2\big](x) \big|,
$$
we conclude that 
$$
 |(-\Delta)^s\xi_k^2(x)|\leq C\delta^{-2s}(x)\quad\text{and}\quad (-\Delta)^s\xi_k^2(x)\to 0\quad\text{as}\quad k\to\infty, 
$$
which conclude the first claim \eqref{apendix2.claim1}. The second claim \eqref{apendix2.claim2} is proved similarly by using \cite[Proposition 6.8]{DFW} and the fact that $|\xi_k^2(x)-\xi_k^2(y)|\leq 2|\xi_k(x)-\xi_k(y)|$.

\end{proof}
\section{Estimates on the Green function}\label{est-green}
The following estimates on the Green function are repeatedly used throughout this manuscript.
Let $\Omega$ be a bounded open set of $\R^N$ of class $C^{1,1}$. Then for all $x,y\in\Omega$ with $x\neq y$, there holds (see e.g \cites{Tedeuz, CS})
\begin{equation}\label{green100}
G^s_\Omega(x,y)\le c_1\min\left\{\frac{1}{|x-y|^{N-2s}}, \frac{\delta^s(x)}{|x-y|^{N-s}}, \frac{\delta^s(y)}{|x-y|^{N-s}}\right\}, \quad \textup{ for a.e. } x, y \in \Omega,
\end{equation}
for some constant $c_1>0$. 

%%%%%%%%%%%%%%%%%%%%%%%%%%%%%
%

%

\begin{thebibliography}{99}

\bibitem{AROS}
A. Audrito, and X. Ros-Oton. "The Dirichlet problem for nonlocal elliptic operators with $C^{0,\alpha}$ exterior data." Proceedings of the American Mathematical Society 148.10 (2020): 4455-4470.
\bibitem{Abatangelo}
N. Abatangelo, \textit{Large $ s $-harmonic functions and boundary blow-up solutions for the fractional Laplacian.}, Discrete Contin. Dyn. Syst. 35 (2015), no. 12, 5555--5607. 
%\bibitem{Abdellaoui}
%B. Abdellaoui, et al. "Global fractional Calder\'on-Zygmund type regularity." arXiv preprint arXiv:2107.06535 (2021).

%K. Bogdan, and T. Jakubowski. \textit{Estimates of the Green function for the fractional Laplacian perturbed by gradient}. Potential Anal. 36 (2012), no. 3, 455--481.


%\bibitem{Bucur}
%C. Bucur, \textit{Some observations on the Green function for the ball in the fractional Laplace framework.} Commun. Pure Appl. Anal. 15 (2016), no. 2, 657--699.



%\bibitem{DalibardVaret}
%A. L. Dalibard, and D. G\'erard-Varet, \textit{On shape optimization problems involving the fractional Laplacian.} ESAIM Control Optim. Calc. Var. 19 (2013), no. 4, 976--1013.
%\bibitem{Bogdan2011}
%K. Bogdan, and T. Jakubowski. \textit{Estimates of the Green function for the fractional Laplacian perturbed by gradient}. Potential Anal. 36 (2012), no. 3, 455--481.
\bibitem{Zolesio}
M. Delfour and J. Zolesio, \textit{Shapes and Geometries. Analysis, Differential Calculus, and Optimization,} Advances in Design and Control, Vol. 4, Society for Industrial and Applied Mathematics (SIAM), Philadelphia, PA, 2001.
%\bibitem{KN}
%N. De Nitti, and T. K\"onig. \textit{Critical functions and blow-up asymptotics for the fractional Brezis–Nirenberg problem in low dimension}. Calculus of variations and partial differential equations, 62(4), p.114. 2023
%\bibitem{Dias-et-al}
%J. I. Diaz, D. Gómez-Castro, and J. L. Vázquez. "The fractional Schr\"dinger equation with general nonnegative potentials. The weighted space approach." Nonlinear Analysis 177 (2018): 325-360.

\bibitem{DFW} 
S. M. Djitte, M. M. Fall, and T. Weth, \textit{A fractional Hadamard formula and applications.} Calc. Var. Partial Differential Equations 60 (2021), no. 6, Paper No. 231, 31 pp.

\bibitem{DFW2023}
S. M. Djitte, M. M. Fall and T. Weth, \textit{A generalized fractional Pohozaev identity and applications.} Adv. Calc. 0 (2023).

\bibitem{Sidy-Franck-BP}
S. M. Djitte, and F. Sueur. "A Brezis and Peletier type result for the fractional Robin function." Potential Analysis 64.1 (2026): 21.%\textcolor{red}{To be completed}}

%\bibitem{Sidy-Franck-RKHS}
%S. M. Djitte and F. Sueur, \textit{On Lions' formula for RKHS of $s$-harmonic functions}, Preprint 2024. %\textcolor{red}{To be completed}}

\bibitem{Fall}
 M. M. Fall, \textit{Regularity estimates for nonlocal Schr\"odinger equations}.  Discrete \& Continuous Dynamical Systems - A, 2019, 39 (3) : 1405-1456. doi: 10.3934/dcds.2019061
%\bibitem{ELPL}
%M. Englis, D. Lukkassen, J. Peetr, and L. E.  Persson, \textit{On the formula of Jacques-Louis Lions for reproducing kernels of harmonic and other functions.}  J. Reine Angew. Math. 570 (2004), 89--129. 



%\bibitem{FGMP}
%M. M. Fall, M. Ghimenti, A. M. Micheletti, and  A. Pistoia (2023), \textit{Generic properties of eigenvalues of the fractional Laplacian.}  Calculus of Variations and Partial Differential Equations, 62(8), 233.

\bibitem{FallSven}
M. M. Fall, and S. Jarohs, \textit{Gradient estimates in fractional Dirichlet problems.}  Potential Anal. 54 (2021), no. 4, 627--636. 


 \bibitem{FRO}
X. Fern\'andez-Real, and X. Ros-Oton, \textit{Integro-Differential Elliptic Equations }(2023).  Progress in Mathematics, Birkh\"auser Cham.


\bibitem{Oz}
D. Fujiwara and S. Ozawa, \textit{The Hadamard variational formula for the Green functions of some normal elliptic boundary value problems.} Proc. Japan Acad. Ser. A Math. Sci. 54 no. 8 (1978), 215-220.
\bibitem{Garabedian}
P. R. Garabedian, and M. Schiffer. "Convexity of domain functionals." Journal d'Analyse Math\'ematique 2.2 (1952): 281-368.

%\bibitem{Garafalo}
%N. Garofalo "Fractional thoughts." arXiv preprint arXiv:1712.03347 (2017).
%\bibitem{GMP}
%M. Ghimenti, A. M. Micheletti, and A. Pistoia (2023), \textit{Generic properties of eigenvalues of the fractional Laplacian.} arXiv preprint arXiv:2304.07335.



%\bibitem{Grubb-2020} 
%G. Grubb, \textit{Exact Green's formula for the fractional Laplacian and perturbations.} Math. Scand. 126 (2020), no. 3, 568--592.

%\bibitem{G}
%B. Gustafsson, \textit{On the convexity of a solution of Liouville ’s equation}. Duke Math. J., 60 (1990), pp. 303-311.

\bibitem{Hadamard}
 J. Hadamard, \textit{M\'emoire sur le probleme d’analyse relatif à  l’equilibre des plaques \'elastiques encastr\'ees.} Oeuvres, C. N. R. S., 2, Anatole France, 1968, 515-631.
  

\bibitem{HenrotPierre}
A. Henrot, and M. Pierre, \textit{Shape variation and optimization,} EMS Tracts in Mathematics, 28. European Mathematical Society (EMS), Z\"urich, 2018. xi+365 pp. ISBN: 978-3-03719-178-1 MR379.
\bibitem{SvenSaldanaTobias}
S. Jarohs, A. Saldana, and T. Weth, \textit{A new look at the fractional Poisson problem via the logarithmic Laplacian.} Journal of Functional Analysis 279.11 (2020): 108732.
\bibitem{KW}
 M. Kim, and M. Weidner. \textit{"Optimal boundary regularity and Green function estimates for nonlocal equations in divergence form."} arXiv preprint \url{https://arxiv.org/pdf/2408.12987} (2024).
 \bibitem{Kozono-Ushikoshi}
H. Kozono, and E. Ushikoshi, "\textit{Hadamard variational formula for the Green’s function of the boundary value problem on the Stokes equations.}" Archive for Rational Mechanics and Analysis 208.3 (2013): 1005-1055.
\bibitem{Tedeuz}
T. Kulczycki, \textit{Properties of Green function of symmetric stable processes,} Probab. Math. Statist. 17 (1997), 339-364.



%\bibitem{O.Rey}
%O. Rey, \textit{The role of the Green's function in a non-linear elliptic equation involving the critical Sobolev exponent.} J. Funct. Anal. 89.1 (1990): 1-52.

\bibitem{RS}
X. Ros-Oton, and J. Serra, \textit{The Dirichlet problem for the fractional Laplacian: regularity up to the boundary.}  J. Math. Pures Appl. (9) 101 (2014), no. 3, 275--302.

%\bibitem{RO-a}
%X. Ros-Oton, and J. Serra, \textit{The Pohozaev identity for the fractional Laplacian.}  Arch. ratio. Mech. Anal. 213 (2014), no. 2, 587--628.


\bibitem{Sokolowski}
L. Sokolowski, et al, \textit{Introduction to shape optimization.} Springer Berlin Heidelberg, 1992.
\bibitem{Schiffer}
M. Schiffer, "\textit{Hadamard's formula and variation of domain-functions.}" American Journal of Mathematics 68.3 (1946): 417-448.
\bibitem{CS} 
Z. Q. Chen, R. Song, \textit{Estimates on Green functions and Poisson kernels for symmetric stable processes.} Math. Ann. 312 (1998), no. 3, 465--501.
\bibitem{Sylvestre}
L. Silvestre. "Regularity of the obstacle problem for a fractional power of the Laplace operator." Communications on Pure and Applied Mathematics: A Journal Issued by the Courant Institute of Mathematical Sciences 60.1 (2007): 67-112.
\bibitem{Ushikoshi}
E. Ushikoshi, "\textit{Hadamard variational formula for the Green function for the velocity and pressure of the Stokes equations.}" Indiana University Mathematics Journal (2013): 1315-1379.

\end{thebibliography}
\end{document}